\documentclass[10pt,nocut]{article}
\usepackage[all]{xy}

\usepackage{./leoTemplate}
\usepackage{./notations}

\begin{document}
\title{Fundamental limits of symmetric low-rank matrix estimation}
\author{Marc Lelarge \& L\'eo Miolane\footnote{M.L. and L.M. are with INRIA-ENS, Paris France, emails: marc.lelarge@ens.fr and leo.miolane@inria.fr}}
\date{}
\maketitle
\begin{abstract}
We consider the high-dimensional inference problem where the signal is a low-rank symmetric matrix which is corrupted by an additive Gaussian noise. Given a probabilistic model for the low-rank matrix, we compute the limit in the large dimension setting for the mutual information between the signal and the observations, as well as the matrix minimum mean square error, while the rank of the signal remains constant. We also show that our model extends beyond the particular case of additive Gaussian noise and we prove an universality result connecting the community detection problem to our Gaussian framework. We unify and generalize a number of recent works on PCA, sparse PCA, submatrix localization or community detection by computing the information-theoretic limits for these problems in the high noise regime. In addition, we show that the posterior distribution of the signal given the observations is characterized by a parameter of the same dimension as the square of the rank of the signal (i.e.\ scalar in the case of rank one). 
This allows to locate precisely the information-theoretic thresholds for the above mentioned problems.
Finally, we connect our work with the hard but detectable conjecture in statistical physics.
\end{abstract}

\section{Introduction}

The estimation of a low-rank matrix observed through a noisy channel is a fundamental problem in statistical inference with applications in machine learning, signal processing or information theory. We shall consider the high dimensional setting where the low-rank matrix to estimate is symmetric and where the noise is additive and Gaussian:
\begin{equation} \label{eq:problem}
	\bY = \sqrt{\frac{\lambda}{n}} \, \bX\bX^{\intercal} +\bZ
\end{equation}
where $n$ is the dimension and $\lambda$ captures the strength of the signal.
Our framework can encompass a wide range of low-rank signal $\bX$ where the components of the vector are i.i.d.\ with a given prior distribution $P_0$.
Moreover, thanks to the universality property first introduced in~\cite{DBLP:conf/allerton/LesieurKZ15} and proved in~\cite{krzakala2016mutual}, our results with additive Gaussian noise have direct implications for a wide range of channels. In the context of community detection, we will prove another universality result showing the equivalence between Bernoulli channel and Gaussian channel that will allow us to transfer our results about rank-one matrix estimation to the community detection problem in the limit of large degrees.
In this paper, we aim at computing the best achievable performance (in term of mean square error) for the estimation of the low-rank signal. More precisely, we prove limiting expressions for the mutual information $I(\bbf{X};\bbf{Y})$ and the minimum mean square error (MMSE), as conjectured in \cite{DBLP:conf/allerton/LesieurKZ15}. This allows us to compute the information-theoretic threshold for this estimation problem,
i.e.\ the critical value $\lambda_c$ such that when $\lambda < \lambda_c$ no algorithm can retrieve the signal better than a ``random guess'' whereas for $\lambda > \lambda_c$ the signal can be estimated more accurately. 

As we explain below, particular instances of our result (corresponding to various choices for the prior distribution $P_0$) have been studied recently. Bounds based on second moment computations have been derived (see the recent works~\cite{banks2016information,perry2016optimality} and the references therein) but they are not expected to be tight in the regime considered in this paper.
Random matrix theory also provides some bounds~\cite{baik2005phase,feral2007largest,benaych2011eigenvalues} and we will comment their tightness in the sequel. Another proof technique relies on the careful analysis of an approximate message passing (AMP) algorithm first introduced in~\cite{rangan2012iterative} for the matrix factorization problem and studied in~\cite{deshpande2014information,montanari2015finding,deshpande2016asymptotic}. 

The underlying idea behind the study of AMP is that the estimation problem \eqref{eq:problem} can be characterized by a single scalar equation~\cite{rangan2012iterative}, a behavior called ``replica-symmetric'' in statistical physics.
From a physics point of view, one can see the components of $\bbf{X}$ as a system of $n$ spins distributed according to the (random) posterior distribution $\PP(\bbf{X}\, | \, \bbf{Y})$. This system is expected to have a ``replica-symmetric'' behavior (see for instance~\cite{zdeborova2016statistical}). This means that the correlations between the spins vanish in the $n \to \infty$ limit so that important quantities will concentrate around their means. For the rank-one case, this implies that the behavior of the system will be characterized by a single scalar parameter.
This ``replica-symmetric'' scenario is well known in the physics literature. It corresponds to the high-temperature behavior of the Sherrington-Kirkpatrick (SK) model, studied by Mezard, Parisi and Virasoro in their groundbreaking book~\cite{mezard1987spin}. At low-temperature,~\cite{mezard1987spin} predicted a ``replica-symmetry breaking'' for the SK model, so that the system is no more described by a simple scalar but a function.
However, this would not be the case for our estimation problem, (mostly) because it arises from a planted problem.
This class of models enjoys specific properties due to the presence of the planted (hidden) solution of the estimation problem and to the fact that the parameters of the inference channel (noise, priors...) are supposed to be known by the statistician. In the statistical physics jargon, the system is on the ``Nishimori line'' (see~\cite{nishimori2001statistical,iba1999nishimori,korada2009exact}), a region of the phase diagram where no ``replica-symmetry breaking'' occurs. These properties will play a crucial role in our proofs. For a detailed introduction to the connections between statistical physics and statistical inference, see~\cite{zdeborova2016statistical}.

Our proof technique will therefore be built on the mathematical approach developed by Talagrand~\cite{talagrand2010meanfield1} and Panchenko~\cite{panchenko2013SK} to study the Sherrington-Kirkpatrick model. 
We proves limiting expressions for the mutual information (i.e.\ the free energy) and the MMSE, confirming a conjecture from~\cite{DBLP:conf/allerton/LesieurKZ15}. This conjecture was recently proved by~\cite{barbier2016mutual} (under some additional assumptions) for rank-one matrix estimation, using AMP and spatial coupling techniques.
In the present paper, we are able to show that a sample $\bx$ drawn from the posterior distribution has an asymptotic deterministic overlap with the signal $\bX$: $\left(\frac{1}{n} \sum_i x_iX_i\right)^2\to q^*(\lambda)^2$ in $L^2$. Suppose for simplicity that $\E_{P_0} X = 0$, then we show that as soon as $q^*(\lambda)>0$, it is possible to strictly improve over dummy estimators, i.e.\ estimators that do not depend on the observed data $\bY$. Hence our result gives an explicit formula to compute the minimal value of the signal strength $\lambda$ in order to do strictly better than the dummy estimator for a wide range of low-rank matrix estimation problem. Moreover, it gives the best possible performance as a function of $\lambda$ and the prior $P_0$, achievable by any algorithm (with no computational constraint). Finally, our work leads to an extension of the hard but detectable conjecture from statistical physics that we present in Section~\ref{sec:conjecture}.

Our main results are presented in the next section where some applications are also described. In Section~\ref{sec:rs_formula}, we make the connection with the statistical physics approach, Section~\ref{sec:proof_rs_formula} contains the proof of our main first result and Section~\ref{sec:overlap_concentration_without_perturbation} contains the proof of the concentration for the overlap. The generalization to finite rank and general priors is done in Section~\ref{sec:gen}. Finally, the connection with the community detection problem is done in Section~\ref{sec:sbm}.

\section{Main results}

\subsection{Rank-one matrix estimation}\label{sec:rank1}
Let $P_0$ be a probability distribution on $\R$ with finite second moment. Consider the following Gaussian additive channel for $\lambda>0$,
\BEA
\label{defmodel}Y_{i,j} = \sqrt{\frac{\lambda}{n}} X_i X_j + Z_{i,j}, \quad \text{for } 1 \leq i<j \leq n,
\EEA
where $X_i \overset{\text{\tiny i.i.d.}}{\sim} P_0$ and $Z_{i,j} \overset{\text{\tiny i.i.d.}}{\sim} \mathcal{N}(0,1)$. We denote the input vector by $\bX=(X_1,\dots, X_n)$, the output matrix by $\bY=(Y_{i,j})_{1\leq i<j\leq n}$ and the noise matrix by $\bZ=(Z_{i,j})_{1\leq i<j\leq n}$. We denote by $\EE$ the expectation with respect to the randomness of $\bX,\bY,\bZ$. Notice that we suppose here to observe only the coefficients of $\sqrt{\lambda/n} \bbf{X} \bbf{X}^{\intercal} + \bbf{Z}$ that are above the diagonal. The case where all the coefficients are observed can be directly deduced from this case.

Our first main result is an exact computation of the limit when $n$ tends to infinity of the mutual information $\frac{1}{n} I(\bbf{X},\bbf{Y})$ for this Gaussian channel as well as the matrix minimum mean square error defined by:
\begin{align*}
\MMSE_n(\lambda) 
&= \min_{\hat{\theta}} \frac{2}{n(n-1)} \sum_{1\leq i<j\leq n} \EE\left[ \left(X_iX_j- \hat{\theta}_{i,j}(\bbf{Y}) \right)^2\right] \label{eq:def_mmse_min_intro} \\
&=\frac{2}{n(n-1)} \sum_{1\leq i<j\leq n} \EE\left[ \left(X_iX_j-\EE\left[X_iX_j|\bbf{Y} \right]\right)^2\right],
\end{align*}
where the minimum is taken over all estimators $\hat{\theta}$ (i.e.\ measurable functions of the observations $\bbf{Y}$ that could also depend on auxiliary randomness).
We define the following function
\BEA
\label{eq:defF}\mathcal{F} : (\lambda,q) \in \R_+^2 \mapsto -\frac{\lambda}{4}q^2 + \E \log \left( \int dP_0(x) \exp\left(\sqrt{\lambda q}Zx + \lambda q x X - \frac{\lambda}{2}q x^2\right)\right),
\EEA
where $Z \sim \mathcal{N}(0,1)$ and $X \sim P_0$ are independent random variables.

\begin{theorem}\label{th:1}
For $\lambda>0$, we have
\BEAS
\lim_{n\to \infty} \frac{1}{n} I(\bbf{X},\bbf{Y}) = \frac{\lambda \E_{P_0} [X^2]^2}{4} - \sup_{q\geq 0}\mathcal{F} (\lambda,q).
\EEAS
This limit is a concave function of $\lambda$. Let $D \subset (0, + \infty)$ be the set of points where this function is differentiable. By concavity, $D$ is equal to $(0, + \infty)$ minus a countable set.
Then, for all $\lambda \in D$, the maximizer $q^*(\lambda)$ of $q\geq 0\mapsto \Fcal(\lambda,q)$ is unique and is such that
\BEAS
\lim_{n\to \infty}\MMSE_n(\lambda) = \E_{P_0} [X^2]^2-q^*(\lambda)^2.
\EEAS
\end{theorem}
To the best of our knowledge, the rigorous result closest to ours is provided by~\cite{barbier2016mutual} (for discrete priors) where a restrictive assumption is made on $P_0$, namely the function $q\mapsto \Fcal(\lambda,q)$ is required to have at most three stationary points. Our most general result will generalize Theorem~\ref{th:1} to any probability distribution $P_0$ over $\RR^k$ with finite second moment and with $k$ fixed (see Section~\ref{sec:introg}). For the sake of clarity, we first concentrate on the rank-one case, provide a detailed proof and then generalize it to the general case.
\\

In order to get an upper bound on the matrix minimum mean square error, we will consider the ``dummy estimators'', i.e.\ estimators $\hat{\theta}$ that do not depend on $\bY$ (and that are thus independent of $\bX$). If $\hat{\theta}$ is a dummy estimator, its mean square error is equal to
\begin{align*}
	\text{MSE}(\hat{\theta}) = \frac{2}{n(n-1)} \sum_{1\leq i<j\leq n} \EE\left[ \left(X_iX_j- \E \hat{\theta}_{i,j} \right)^2\right] + \text{Var}(\hat{\theta}_{i,j})
\end{align*}
because $\bX$ and $\hat{\theta}$ are independent. Therefore, the ``best'' dummy estimator (in term of mean square error) is $\hat{\theta}_{i,j} = \E_{P_0} [X]^2$ for all $i<j$ which gives a ``dummy'' matrix mean square error of:
\BEAS
\DMSE = \E_{P_0} [X^2]^2-\E_{P_0} [X]^4\geq 0.
\EEAS
As we will see later in Proposition~\ref{prop:q_star}, the optimizer $q^*(\lambda)$ defined in Theorem~\ref{th:1} is such that: $q^*(\lambda) \xrightarrow[\lambda \to 0]{} \E_{P_0}[X]^2$ and $q^*(\lambda) \xrightarrow[\lambda \to \infty]{} \E_{P_0}[X^2]$. Consequently, Theorem~\ref{th:1} gives the limits of the MMSE for the low (i.e.\ $\lambda\to 0$) and high (i.e.\ $\lambda\to \infty$) signal regimes:
\begin{align*}
	\lim_{\lambda \to 0} \lim_{n \to \infty} \MMSE_n(\lambda) &= \DMSE \\
	\lim_{\lambda \to \infty} \lim_{n \to \infty} \MMSE_n(\lambda) &= 0
\end{align*}
It is important to note that the regime considered in this paper with $\lambda \in (0,\infty)$ corresponds to a high noise regime: the $\MMSE$ will be positive for any finite value of $\lambda$ in our model~\eqref{defmodel}. In particular, exact reconstruction of the signal is typically not possible.
\\

Theorem~\ref{th:1} indicates that the value of $q^*(\lambda)$ determines the best achievable performance for the estimation problem. We will now see that $q^*(\lambda)$ encodes the geometry of the posterior distribution of $\bbf{X}$ given $\bbf{Y}$.
The posterior distribution of $\bbf{X}$ given $\bbf{Y}$ is given by
\BEA
\label{def:post}
\quad dP(\bbf{x} | \bbf{Y}) = \frac{1}{Z_n(\lambda)} \Big( \prod_{i=1}^n dP_0(x_i) \Big) \exp \big( \sum_{i<j} x_i x_j \sqrt{\frac{\lambda}{n}} Y_{i,j} - \frac{\lambda}{2n} x_i^2 x_j^2 \big),
\EEA
where $Z_n(\lambda)$ is the normalization function. We will adopt a standard notation in statistical physics and denote by $\langle \cdot \rangle$ the average with respect to this random (because depending on $\bY$) distribution. We also denote by $\bx$ a random vector with distribution given by~\eqref{def:post}. This means that for any function $f$ on $\R^n$ that is integrable with respect to $P_0^{\otimes n}$, we have by definition:
$$
\langle f(\bbf{x}) \rangle = \frac{1}{Z_n(\lambda)} \int \Big( \prod_{i=1}^n dP_0(x_i) \Big) f(\bbf{x}) \exp \big( \sum_{i<j} x_i x_j \sqrt{\frac{\lambda}{n}} Y_{i,j} - \frac{\lambda}{2n} x_i^2 x_j^2 \big).
$$
This quantity is well-defined. Indeed if we write $L(\bbf{x})= e^{-\frac{1}{2}\sum_{i<j}(Y_{i,j} - \sqrt{\lambda} x_{i}x_j)^2} \in (0,1]$, we have $$\langle f(\bbf{x}) \rangle = \frac{\E_{P_0^{\otimes n}} [L(\bbf{x}) f(\bbf{x})]}{\E_{P_0^{\otimes n}} [L(\bbf{x})]}.$$
Note that this quantity is random and we will see it as a function of the random vectors $\bX$ and $\bZ$:
\BEAS
\langle f(\bbf{x}) \rangle = \frac{1}{Z_n(\lambda)} \int \Big( \prod_{i=1}^n dP_0(x_i) \Big) f(\bbf{x}) \exp \big( \sum_{i<j} x_i x_j \sqrt{\frac{\lambda}{n}} (Z_{i,j} + \sqrt{\frac{\lambda}{n}} X_i X_j ) - \frac{\lambda}{2n} x_i^2 x_j^2 \big),
\EEAS
and $\EE\langle f(\bbf{x}) \rangle$ is then its mean.

For $\bbf{u},\bbf{v} \in \R^n$ we define the overlap between the configurations $\bbf{u}$ and $\bbf{v}$ as: $\bbf{u}.\bbf{v} = \frac{1}{n} \sum_{i=1}^n u_i v_i$.
Let $m \in \N^*$. The geometry of the Gibbs distribution $\langle \cdot \rangle$ can be characterized by the matrix $(\bbf{x}^{(i)}.\bbf{x}^{(j)})_{1 \leq i,j \leq m}$ of the overlaps between $m$ i.i.d.\ samples $\bbf{x}^{(1)}, \dots, \bbf{x}^{(m)}$ from $\langle \cdot \rangle$. Indeed, one can easily verify (using Proposition~\ref{prop:nishimori}) that the rescaled norm of each sample $\sqrt{\bbf{x}^{(i)}.\bbf{x}^{(i)}}$ concentrates around $\sqrt{\E_{P_0}[X^2]}$, so that the matrix of the overlaps encodes the distances between $m$ samples from $\langle \cdot \rangle$.

The Nishimori identity (Proposition~\ref{prop:nishimori}) gives that $\bbf{x}^{(i)}.\bbf{x}^{(j)}$ is equal to $\bbf{x}.\bbf{X}$ in law. The behavior of this quantity is simple: the next result shows that $(\bx.\bX)^2$ concentrates asymptotically around $q^*(\lambda)^2$. In words, we see that if $\bY$ is obtained from $\bX$ thanks to~\eqref{defmodel} and if $\bx$ is a random vector distributed according to the posterior distribution~\eqref{def:post} given $\bY$, then the square of its overlap with the initial vector $\bX$, i.e.\ $\left(\bx.\bX\right)^2$ converges to the deterministic value $q^*(\lambda)^2$ as $n$ tends to infinity. Thus, $q^*(\lambda)$ encodes the geometry of the Gibbs distribution $\langle \cdot \rangle$.

\begin{theorem}\label{th:2}
 If $P_0$ has a bounded support then for all $\lambda \in D$ we have for $\bx$ with distribution given by~\eqref{def:post},
\BEAS
\E \Big\langle \big((\bbf{x}.\bbf{X})^2 - q^*(\lambda)^2\big)^2 \Big\rangle \xrightarrow[n \to \infty]{} 0.
\EEAS
\end{theorem}

To the best of our knowledge, the convergence in $L^2$ stated in Theorem~\ref{th:2} is a new contribution of our work.

\subsection{Effective Gaussian scalar channel}

We now study more carefully the quantity $q^*(\lambda)$ which characterizes the limit for the square of the overlap of two vectors drawn from the posterior distribution. We will relate it to the following scalar Gaussian channel:
\begin{equation} \label{eq:scalar_channel}
Y_0 = \sqrt{\gamma} X_0+ Z_0,
\end{equation}
where $X_0 \sim P_0$ and $Z_0 \sim \mathcal{N}(0,1)$ are independent random variables.
Note that the posterior distribution of $X_0$ knowing $Y_0$ is then given by
$dP(X_0=x|Y_0) = \frac{1}{Z(Y_0)} dP_0(x)e^{Y_0\sqrt{\gamma}x-\frac{\gamma x^2}{2}}$,
where the random variable $Z(Y_0)$ is the normalizing constant:
\BEAS
Z(Y_0) = \int dP_0(x)e^{Y_0\sqrt{\gamma}x-\frac{\gamma x^2}{2}}
=\int dP_0(x)e^{\gamma x X_0+\sqrt{\gamma}x Z_0 -\frac{\gamma x^2}{2}}.
\EEAS
We can then relate this quantity to the mutual information $\i(\gamma) = I(X_0,Y_0)$ of the scalar Gaussian channel
\begin{align}
\i(\gamma)
&= \frac{\gamma\EE[X_0^2]}{2}-\EE[\log Z(Y_0)]
= \frac{\gamma\EE[X_0^2]}{2}-\EE \log \Big[ \int dP_0(x) \exp(\sqrt{\gamma}Z_0x + \gamma x X_0 - \frac{\gamma}{2} x^2)\Big]. \label{eq:i_scalar}
\end{align}
Hence, playing with the equations, we can rewrite the first statement of Theorem~\ref{th:1} as follows:
\begin{equation} \label{eq:limsc}
	\lim_{n\to \infty} \frac{1}{n} I(\bbf{X},\bbf{Y}) = \frac{\lambda \E [X_0^2]^2}{4} - \sup_{q \geq 0} \left( \frac{\lambda q}{2}\left(\E[X_0^2]-\frac{q}{2}\right) -\i(\lambda q)\right).
\end{equation}
The minimum mean square error for the scalar Gaussian channel is defined as
\begin{equation}\label{eq:mmsescalar}
\mmse(\gamma) = \EE\left[ (X_0 - \EE[X_0|Y_0])^2\right]
= \EE\left[ X_0^2\right] - \EE\left[ \EE[X_0|Y_0]^2\right],
\end{equation}
where we used the identity $\EE\left[ \EE[X_0|Y_0]^2\right]=\EE\left[ X_0\EE[X_0|Y_0]\right]$.
The minimum mean square error is related to the mutual information by the following equation (from~\cite{guo2005mutual}):
$\frac{d \i}{d\gamma}(\gamma) = \frac{1}{2}\mmse(\gamma)$.
Now, we see thanks to this relation that the value $q^*(\lambda)$ attaining the supremum in the right-hand term of~\eqref{eq:limsc} should satisfy the following equation (see Proposition~\ref{prop:q_star}):
\BEA
\label{eq:fixp}q^*=\E[X_0^2]-\mmse(\lambda q^*).
\EEA
This equation was first derived in \cite{rangan2012iterative} for matrix factorization and appeared a number of times in settings similar to ours, in the context of community detection~\cite{montanari2015finding} and~\cite{deshpande2016asymptotic}, or sparse PCA~\cite{deshpande2014information}. We will discuss the application to the community detection problem in Sections~\ref{sec:sbmintro} and~\ref{sec:onec}. 
We now discuss the sparse PCA problem as introduced in~\cite{deshpande2014information}. With our notations, this setting corresponds to $P_0\sim\Ber(\epsilon)$ being a Bernoulli distribution with parameter $\epsilon>0$. It is proved in~\cite{deshpande2014information}, that there exists $\epsilon_*$ such that for $\epsilon>\epsilon_*$, the fixed point equation~\eqref{eq:fixp} has only one solution in $[0,\infty)$ and in this case Theorem 2 in~\cite{deshpande2014information} gives the asymptotic $\MMSE$ and shows that it is achieved by AMP algorithm. Our Theorem~\ref{th:1} allows us to compute the asymptotic $\MMSE$ for all values of $\epsilon$, indeed we have:

  \begin{proposition}\label{prop:spca}
    For $P_0\sim\Ber(\epsilon)$, we have
    \BEAS
    \lim_{n\to \infty} \frac{1}{n} I(\bbf{X},\bbf{Y}) = \frac{\lambda\epsilon^2}{4}-\sup_{q\geq 0} \left\{-\frac{\lambda q^2}{4}+\EE\left[\log\left( 1-\epsilon +\epsilon e^{\sqrt{\lambda q}Z+\lambda q X_0-\lambda q/2} \right) \right]\right\},
    \EEAS
    where $X_0\sim \Ber(\epsilon)$ and $Z\sim\Ncal(0,1)$. For almost all $\lambda>0$, the maximizer $q^*(\lambda)$ of the right hand term is unique and is such that
\BEAS
\lim_{n\to \infty}\MMSE_n(\lambda) = \epsilon^2-q^*(\lambda)^2.
\EEAS
  \end{proposition}
  
  As shown in Section~\ref{sec:onec} below, this sparse PCA model is connected to the problem of finding one community in a random graph: we will show that as the average degree tends to infinity, the Bernoulli channel can be approximated by an additive Gaussian channel. Another related problem is the submatrix localization as studied in~\cite{hajek2016information} which corresponds to a case where $P_0$ is a mixture of two Gaussian distributions with different means.

  \subsection{Phase transition in the case \texorpdfstring{$\E_{P_0}[X]=0$}{E[X]=0}}\label{sec:trans}

In this section, we concentrate 
on the particular case where $\E_{P_0}[X]=0$. Without loss of generality, we can also assume that $\E_{P_0}[X^2]=1$.
We first start with the particular case $P_0=\Ncal(0,1)$ where explicit formulas are available. 
The input-output mutual information for the Gaussian scalar channel~\eqref{eq:scalar_channel} is then the well-known channel capacity under input power constraint: $\i(\gamma)=\frac{1}{2}\log (1+\gamma)$ and then $\mmse(\gamma)=\frac{1}{1+\gamma}$. This is a case where~\eqref{eq:fixp} can be solved explicitly, namely, if $\lambda\leq 1$, then $q^*(\lambda) = 0$ and if $\lambda>1$, then two values are possible: $q^*(\lambda)\in \{0,1-\frac{1}{\lambda}\}$ but only $1-\frac{1}{\lambda}$ achieves the supremum in~\eqref{eq:limsc}. Hence we have $q^*(\lambda)=\max\left(0,1-\frac{1}{\lambda}\right)$ so that in the case $P_0=\Ncal(0,1)$, we have:
\BEA
\label{eq:mmsenorm}\MMSE_n(\lambda) \to \left\{
\begin{array}{ll}
  1&\mbox{ if } \lambda\leq 1,\\
  \frac{1}{\lambda}\left(2-\frac{1}{\lambda}\right)&\mbox{ otherwise.}
\end{array}\right.
\EEA
In particular, we see that as long as $\lambda\leq 1$, the dummy estimator $\hat{\theta}_{i,j}=0$ is optimal in term of matrix mean square error. Only when $\lambda>1$, the $\MMSE$ starts to decrease below $1$. 

Our probabilistic model~\eqref{defmodel} has been the focus of much recent work in random matrix theory~\cite{baik2005phase,feral2007largest,benaych2011eigenvalues}. The focus in this literature is the analysis of the extreme eigenvalues of the symmetric matrix $\bY/\sqrt{n}$ and its associated eigenvector leading to performance guarantee for principal component analysis (PCA). The main result of interest to us is the following: for any distribution $P_0$ such that $\E_{P_0}[X^2]=1$, we have
\begin{itemize}
\item if $\lambda\leq 1$, the top eigenvalue of $\bY/\sqrt{n}$ converges a.s.\ to $2$ as $n\to \infty$, and the top eigenvector $\bv$ (with norm $\|\bv\|^2=n$) has trivial correlation with $\bX$: $\bv.\bX\to 0$ a.s.
  \item if $\lambda>1$, the top eigenvalue of $\bY/\sqrt{n}$ converges a.s.\ to $\sqrt{\lambda}+1/\sqrt{\lambda}>2$ and the top eigenvector $\bv$ (with norm $\|\bv\|^2=n$) has nontrivial correlation with $\bX$: $\left(\bv.\bX\right)^2\to 1-1/\lambda$ a.s.
\end{itemize}
Note that this result has been proved under considerably fewer assumptions than we make in the present paper. We refer the interested reader to~\cite{feral2007largest,benaych2011eigenvalues} and the references therein for further details.
In our context, we can compare the performance of PCA with the information-theoretic bounds. If we take an estimator proportional to $v_iv_j$, i.e.\ $\hat{\theta}_{i,j}=\delta v_iv_j$ for $\delta\geq 0$, we can compute explicitly the $\MSE$ obtained as a function of $\delta$ and minimize it. The optimal value for $\delta$ depends on $\lambda$, more precisely if $\lambda<1$, then $\delta=0$ resulting in a limit for $\MSE$ of one while for $\lambda\geq 1$, the optimal of value for $\delta$ is $1-1/\lambda$ resulting in the following $\MSE$ for PCA:
\BEA
\label{eq:perfpca}\MSE^{\text{PCA}}_{n}(\lambda) \to \left\{
\begin{array}{ll}
  1&\mbox{ if } \lambda\leq 1,\\
  \frac{1}{\lambda}\left(2-\frac{1}{\lambda}\right)&\mbox{ otherwise.}
\end{array}\right.
\EEA
Comparing to~\eqref{eq:mmsenorm}, we see that in the particular case of $P_0=\Ncal(0,1)$, PCA is optimal: it is able to get a matrix mean square error strictly less than $1$ as soon as it is information theoretically possible and its mean square error is optimal.

\begin{figure}[thpb]
	\centering
	\includegraphics[width=7cm]{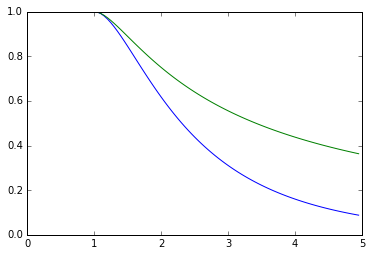}
	\caption{Performance of PCA for the $\ZZ /2$ synchronization problem: the blue curve is the limit of the $\MMSE$ (with the prior $P_0(+1)=P_0(-1)=\frac{1}{2}$) and the green curve is the limit of the $\MSE$ achieved by PCA~\eqref{eq:perfpca}, seen as functions of $\lambda$.}
	\label{fig:comp}
\end{figure}

We now discuss the $\ZZ /2$ synchronization problem studied in~\cite{bandeira2016tightness} which corresponds to the prior $P_0(+1)=P_0(-1)=\frac{1}{2}$. It turns out that exactly this model has been studied in~\cite{deshpande2016asymptotic} and the connection with the community detection problem will be made clear in the next section. We now compute the $\MMSE$ for this problem. The $\mmse$ for the effective Gaussian scalar channel~\eqref{eq:scalar_channel} can be computed explicitly. An easy computation gives
\BEAS
\EE[X_0|Y_0] = \tanh\left( \sqrt{\gamma}Y_0\right),
\EEAS
so that we have thanks to~\eqref{eq:mmsescalar}:
$$
\mmse(\gamma) = 1 - \EE\left[ \tanh\left( \sqrt{\gamma}Z_0+\gamma X_0\right)^2\right]\\
= 1 - \EE\left[ \tanh\left( \sqrt{\gamma}Z_0+\gamma \right)^2\right].
$$
In particular, the fixed point equation~\eqref{eq:fixp} reduces now to $q^* = \EE\left[ \tanh\left( \sqrt{\lambda q^*}Z_0+\lambda q^* \right)^2\right]$ which has one solution for $\lambda \leq 1$ equals to zero and an additional solution $q^*(\lambda)>0$ for $\lambda>1$ which is the one achieving the supremum in~\eqref{eq:limsc}. Hence extending $q^*(\lambda)$ to zero for $\lambda\leq 1$, we have in this case, $\MMSE_n(\lambda) \to 1-q^*(\lambda)^2$ for all values of $\lambda>0$. It turns out that for $\lambda>1$, we have $1-q^*(\lambda)^2<\frac{1}{\lambda}\left(2-\frac{1}{\lambda}\right)$ as shown on Figure~\ref{fig:comp}. This is a case where PCA is able to beat the dummy estimator as soon as it is information theoretically possible but still achieves a sub-optimal $\MMSE$.
\\

We now present a more general result. As we will see later in Proposition~\ref{prop:q_star}, $\lambda \in D \mapsto q^*(\lambda)$ is non-decreasing and, in the case of a centered distribution $P_0$, we have $\lim\limits_{\lambda\to 0} q^*(\lambda)=0$. We then define
\BEA
\label{def:lambdac}\lambda_c = \sup \{\lambda > 0 \ | \ q^*(\lambda) = 0 \}.
\EEA

A direct application of Theorem~\ref{th:1} gives
\begin{proposition}\label{prop:lc}
  We assume that $P_0$ is such that $\E_{P_0}[X]=0$ and $\E_{P_0}[X^2]=1$.
  For all $\lambda< \lambda_c$, we have
\BEAS
\lim_{n\to \infty} \frac{1}{n} I(\bbf{X},\bbf{Y}) = \frac{\lambda}{4} \ \mbox{and,} \ \lim_{n\to \infty}\MMSE_n(\lambda) = 1.
\EEAS
For almost all $\lambda>\lambda_c$, we have
\BEAS
\lim_{n\to \infty} \frac{1}{n} I(\bbf{X},\bbf{Y}) = \frac{\lambda}{4}-\Fcal(\lambda,q^*(\lambda)) \ \mbox{and,} \ \lim_{n\to \infty}\MMSE_n(\lambda) = 1-q^*(\lambda)^2<1,
\EEAS
where $q^*(\lambda)$ is the unique maximum of $q\mapsto \Fcal(\lambda,q)$ defined in~\eqref{eq:defF}.
\end{proposition}

We call $\lambda_c$ the threshold for nontrivial estimation as the $\MMSE$ is strictly less than the dummy mean square error $\DMSE$ only for $\lambda>\lambda_c$. We clearly have $\lambda_c\leq 1$ and the case $\lambda_c=1$ corresponds to cases where PCA is optimal in the sense that it achieves a nontrivial estimation as soon as it is information theoretically possible. Note however that even if $\lambda_c=1$, the $\MSE$ achieved by PCA can be larger than the $\MMSE$ as it is the case for the $\ZZ /2$ synchronization problem described above. Indeed, the performance of PCA does not depend on the prior $P_0$, so that it is not surprising to be sub-optimal in some cases.

There are even cases where $\lambda_c<1$, so that in the range $\lambda\in (\lambda_c,1)$ PCA has the same performance as the dummy estimator with $\MSE^{\text{PCA}}_{n}(\lambda) \to 1$ while the $\MMSE$ is strictly lower than one. Of course achieving this $\MMSE$ might be computationally hard and we comment more on this in Section~\ref{sec:conjecture}.
We now shortly describe an example where $\lambda_c<1$ which corresponds to the sparse Rademacher prior that has been recently studied in~\cite{banks2016information,perry2016optimality}.
This is another example of sparse PCA but now with a centered prior, namely for $\rho\in [0,1]$, we take $P_0(0)=1-\rho$ and $P_0(1/\sqrt{\rho})=P_0(-1/\sqrt{\rho}) = \frac{\rho}{2}$. We then denote by $\lambda_c(\rho)$ the threshold for nontrivial estimation at sparsity level $\rho$. It is easy to see that $\lambda_c(1)=1$ but there exists a critical value $\rho^*$ such that $\lambda_c(\rho)=1$ for $\rho\geq \rho^*$ and $\lambda_c(\rho)<1$ for $\rho< \rho^*$. \cite{banks2016information} provides bounds for $\lambda_c(\rho)$ and bounds for the value of $\rho^*$ are also given in~\cite{perry2016optimality}. The same exact characterization as the one given here for $\rho^*$ was proved in~\cite{krzakala2016mutual,barbier2016mutual}. By a numerical evaluation, we obtain $\rho^*\approx 0.09$ and we refer to~\cite{krzakala2016mutual,barbier2016mutual} for more details.

Another case where PCA is not optimal will be presented with more details in the next Section~\ref{sec:sbmintro} corresponding to the case $P_0\left( \sqrt{\frac{1-p}{p}}\right)=p$ and $P_0\left( - \sqrt{\frac{p}{1-p}}\right)=1-p$.
We will obtain the existence of $p^*$ such that $\lambda_c(p)<1$ for $p<p^*$.

\subsection{Optimal detection in the asymmetric stochastic block model}\label{sec:sbmintro}

In this section, we show how our results for matrix factorization apply to the problem of community detection on a random graph model.
We start by defining the random graph model that we are going to study.

\begin{definition}[Stochastic block model (SBM)]
	Let $M$ be a $2 \times 2$ symmetric matrix whose entries are in $[0,1]$. Let $n \in \N^*$ and $p \in [0,1]$. We define the stochastic block model with parameters $(M,n,p)$ as the random graph $\bbf{G}$ defined by: 

	\begin{enumerate}
		\item The vertices of $\bbf{G}$ are the integers in $\{1, \dots, n\}$. 
		\item For each vertex $i \in \{1, \dots, n \}$ one draws independently $X_i \in \{1,2\}$ according to $\PP(X_i = 1) = p$. $X_i$ will be called the label (or the class, or the community) of the vertex $i$.
		\item For each pair of vertices $\{i,j\}$ the unoriented edge $G_{i,j}$ is then drawn conditionally on $X_i$ and $X_j$ according to a Bernoulli distribution with mean $M_{X_i,X_j}$, independently of everything else. $i \sim j$ in $\bbf{G}$ is and only if $G_{i,j} = 1$.
	\end{enumerate}
\end{definition}

The graph $\bbf{G}$ is therefore generated according to the underlying partition of the vertices in two classes. Our main focus will be on the community detection problem: given the graph $\bbf{G}$, is it possible to retrieve the labels $\bbf{X}$ better than a random guess? 

We investigate this question in the asymptotic of large sparse graphs, where $n \rightarrow +\infty$ while the average degree remains fixed. We will then let the average degree tend to infinity. We note that the models studied in~\cite{DBLP:conf/allerton/LesieurKZ15},~\cite{krzakala2016mutual} or~\cite{barbier2016mutual} showing that the Gaussian additive model approximates well the graph model, deal with dense graphs where the average degree tends to infinity with $n$. \cite{deshpande2016asymptotic} dealing with the symmetric stochastic block model (i.e.\ $p=1/2$) is more closely related to our model as the average degree tends to infinity at an arbitrary slow rate.
We define the connectivity matrix $M$ as follows:
\begin{equation} \label{eq:M}
	M =
	\frac{d}{n}
	\begin{pmatrix}
		a & b \\ 
		b & c
	\end{pmatrix},
\end{equation}
where $a,b,c,d$ remain fixed as $n \rightarrow + \infty$. We will say that community detection is solvable if and only if there is some algorithm that recovers the communities more accurately than a random guess would.
A simple argument (see~\cite{caltagirone2016asymmetric}) shows that if $pa +(1-p)b \neq pb + (1-p)c$
then non-trivial information on the community of a vertex can be gained just by looking at its degree and the community detection is then solvable. In this section, we concentrate on the case:
\begin{equation} \label{eq:equal_degreesintro}
	pa +(1-p)b = pb + (1-p)c = 1.
\end{equation}

The average degree of a vertex is then equal to $d$, independently of its class.
As mentioned above, we are first going to let $n$ tend to infinity, while the other parameters remain constant, and then let $d$ tend to infinity. We need now to define the signal strength parameter $\lambda$ for this model and to relate it to our main model~\eqref{defmodel}. We define $\epsilon = 1 - b$ so that equation~\eqref{eq:equal_degreesintro} allows us to express all parameters in terms of $\epsilon$ and $p$:
\begin{equation} \label{eq:paramsintro}
		a = 1 + \frac{1-p}{p} \epsilon,\quad
		b = 1 - \epsilon, \quad
		c = 1 + \frac{p}{1-p} \epsilon.
\end{equation}
We now define $\tilde{\bX}$ by $\tilde{X}_i = \phi_p(X_i)$, where $\phi_p(1) = \sqrt{\frac{1-p}{p}}$ and $\phi_p(2)= - \sqrt{\frac{p}{1-p}}$. The signal is contained in $\tilde{\bX}$ and the observations are the edges of the graph which are independent Bernoulli random variables with means given by~\eqref{eq:paramsintro}, so that:
\BEA
\label{eq:defG}G_{i,j} = \Ber\left(\frac{d}{n}+\frac{d\epsilon}{n}\tilde{X}_i\tilde{X}_j\right).
\EEA
In a setting where $\epsilon \to 0$, we see that the variance of $G_{i,j}$ does not depend (at the first order) on the signal and we have: $\Var(G_{i,j}) \sim \frac{d}{n}$. Hence, if we try to approximate~\eqref{eq:defG} by a Gaussian additive model by matching the first and second moments, we would have:
\BEAS
\tilde{G}_{i,j} = \frac{d}{n}+\frac{d\epsilon}{n}\tilde{X}_i\tilde{X}_j + \sqrt{\frac{d}{n}}Z_{i,j},
\EEAS
so that by defining $Y_{i,j} = \sqrt{\frac{n}{d}}\left(\tilde{G}_{i,j}-\frac{d}{n}\right)$, we obtain:
\BEA
\label{eq:defGgauss}Y_{i,j} =\sqrt{\frac{d\epsilon^2}{n}}\tilde{X}_i\tilde{X}_j+Z_{i,j},
\EEA
which corresponds exactly to our model~\eqref{defmodel} with $\lambda=d \epsilon^2=d(1-b)^2$ and $P_0\left( \sqrt{\frac{1-p}{p}}\right)=p=1-P_0\left( -\sqrt{\frac{p}{1-p}}\right)$. This heuristic argument will be made rigorous in the sequel and we will show that the limit for the mutual information $\frac{1}{n}I(\bX,\bbf{G})$ is the same as the mutual information $\frac{1}{n}I(\tilde{\bX},\bY)$ of the channel~\eqref{eq:defGgauss}.
Note that in the case $p=1/2$ studied in~\cite{deshpande2016asymptotic}, we end up with exactly the $\ZZ /2$ synchronization problem studied in previous section (see Figure~\ref{fig:comp}).
For a general value of $p$, we are now in the framework of Proposition~\ref{prop:lc}, so that we can define $\lambda_c(p)$ by~\eqref{def:lambdac} for each $p$. We will show in the sequel that the matrix mean square error for the matrix factorization problem~\eqref{eq:defGgauss} corresponds to the ``community overlap'', a popular performance measure for community detection (see for instance~\cite{neeman2014asymSBM}) defined below. We are then able to characterize the solvability of the community detection problem (closing a gap left in~\cite{caltagirone2016asymmetric}). We start with the definition of an estimator of the graph's labels.
\begin{definition}[Estimator]
	An estimator of the labels $\bbf{X}$ is a function $\bbf{x}: \bbf{G} \mapsto \{1,2\}^n$ that could depend on auxiliary randomness (random variables independent of $\bbf{X}$).
\end{definition}

For a labeling $\bbf{x} \in \{1,2\}^n$ and $i \in \{1,2\}$ we define $S_i(\bbf{x}) = \{ k \in \{1, \dots ,n \} | x_k = i \}$, i.e.\ the indices of the nodes that have the label $i$ according to $\bbf{x}$. We now recall a popular performance measure for estimators.

\begin{definition}[Community Overlap]
	For $\bbf{x},\bbf{y} \in \{1,2\}^n$ we define the community overlap of the configuration $\bbf{x}$ and $\bbf{y}$ as
	$$
	\overlap(\bbf{x},\bbf{y}) = \frac{1}{n} \max_{\sigma} \sum_{i=1,2} \Big( \# S_i(\bbf{x}) \cap S_{\sigma(i)}(\bbf{y}) - \frac{1}{n} \# S_i(\bbf{x}) \# S_{\sigma(i)}(\bbf{y}) \Big)
	$$
	where the maximum is taken over the permutations of $\{1,2 \}$.
\end{definition}

Two configurations have thus a positive community overlap if they are correlated, up to a permutation of the classes. We will then say that the community detection problem is solvable, if there exists an estimator (i.e.\ an algorithm) that achieves a positive overlap with positive probability.

\begin{definition}[Solvability]
	We say that the community detection problem is solvable (in the limit of large degrees) if there exists an estimator $\bbf{x}(\bbf{G})$ such that
	$$
	\liminf_{d \to \infty} \liminf_{n \to \infty} \E(\overlap(\bbf{x}(\bbf{G}),\bbf{X})) >0.$$
\end{definition}

For a fixed $p$, we have seen that $\lambda_c(p)$ defined by~\eqref{def:lambdac} with prior $P_0\left( \sqrt{\frac{1-p}{p}}\right)=p=1-P_0\left( -\sqrt{\frac{p}{1-p}}\right)$, is the threshold for nontrivial matrix estimation and the following theorem shows that in the case of the stochastic block model, it is also the threshold for solvability.

\begin{theorem} \label{th:solvability_sbm}
	\begin{itemize}
		\item If $\lambda > \lambda_c(p)$, then the community detection problem is solvable.
		\item If $\lambda < \lambda_c(p)$, then the community detection problem is not solvable.
	\end{itemize}
\end{theorem}
\begin{figure}[thpb]
	\centering
	\includegraphics[width=10cm]{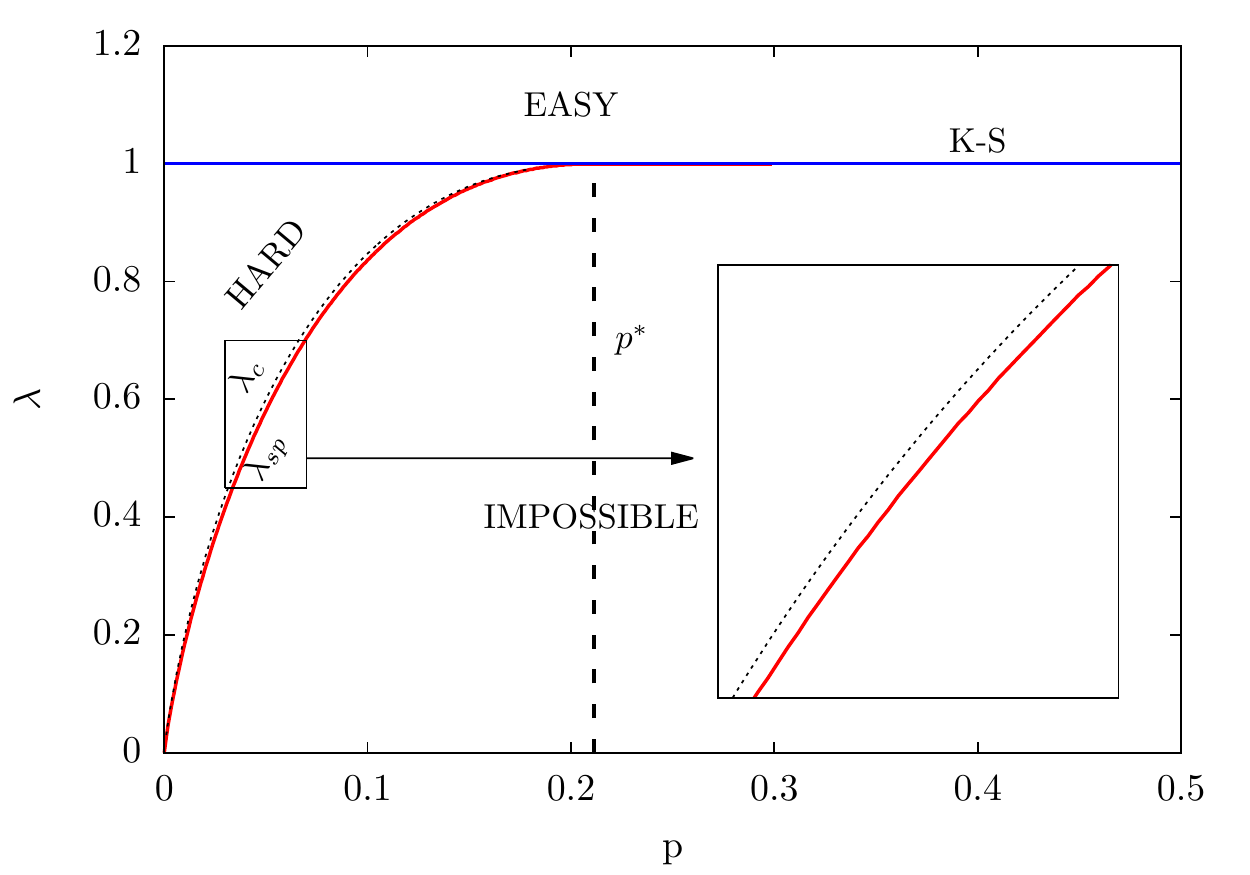}
	\caption{Phase diagram for the asymmetric community detection problem. The easy phase follows from~\cite{abbe2015detection}, the impossible phase below the spinodal curve $\lambda_{sp}(p)$ (red curve) was proved in~\cite{caltagirone2016asymmetric} and the hard phase is a conjecture. The dotted curve corresponding to $\lambda_c(p)$ is the curve for solvability of the community detection problem and is proved by our Theorem~\ref{th:solvability_sbm}.}
	\label{fig:phase_diagram}
\end{figure}

The function $p \mapsto \lambda_c(p)$ is plotted on Figure~\ref{fig:phase_diagram}. We see that the situation is similar to the sparse Rademacher prior described in previous section, where now the parameter $p$ controlling the asymmetry between the two communities play a role similar to the sparsity $\rho$ in the sparse PCA case. More precisely, for $p^*=\frac{1}{2}-\frac{1}{2\sqrt{3}}$ (computed in~\cite{barbier2016mutual} for the model~\eqref{eq:defGgauss} or~\cite{caltagirone2016asymmetric}), we have: if $p\geq p^*$, then $\lambda_c(p)=1$ which is known in this setting as the Kesten-Stigum bound and if $p<p^*$, then $\lambda_c(p)<1$. The regime where $p<p^*$ and $\lambda\in (\lambda_c(p),1)$ is conjectured to be ``hard but detectable'', see Section~\ref{sec:conjecture}.

We have no doubt that our analysis will extend to other models like the censored block model~\cite{saade2015spectral} or the labeled stochastic block model~\cite{saade2016clustering} in the large degree regime.

\subsection{Finding one community}\label{sec:onec}

As explained above, the solvability problem is trivial if the degrees are not homogeneous in the graph. However, the case where the condition~\eqref{eq:equal_degreesintro} is not satisfied is still interesting. The particular case of a single community where $c=b$ and $d=1$ (i.e.\ $
	M =
	\frac{1}{n}
	\begin{pmatrix}
		a & b \\ 
		b & b
	\end{pmatrix}
$) with our previous notations has been studied in~\cite{montanari2015finding}. We can conduct the same heuristic argument as above: let $\tilde{X}_i = 2-X_i \in \{0,1\}$, so that we have
\BEAS
G_{i,j} =\Ber\left( \frac{b}{n}+\frac{a-b}{n}\tilde{X}_i\tilde{X}_j\right).
\EEAS
Hence, in a setting where $a,b\to \infty$ with $a-b=o(b)$, we see that $\Var(G_{i,j}) \sim \frac{b}{n}$ so that the associated Gaussian additive model is given by:
\BEAS
\tilde{G}_{i,j} = \frac{b}{n}+\frac{a-b}{n}\tilde{X}_i\tilde{X}_j +\sqrt{\frac{b}{n}}Z_{i,j},
\EEAS
and then with $Y_{i,j} = \sqrt{\frac{n}{b}}\left(G_{i,j}-\frac{b}{n}\right)$, we get
\BEAS
Y_{i,j} =\sqrt{\frac{(a-b)^2}{nb}}\tilde{X}_i\tilde{X}_j+Z_{i,j},
\EEAS
which corresponds exactly to our model~\eqref{defmodel} with $\lambda=\frac{(a-b)^2}{b}$ and $P_0\left(1\right)=p=1-P_0\left( 0\right)$.
More precisely, the single community problem, with parameters $a,b\to \infty$ such that $\frac{(a-b)^2}{b}\to \lambda$ falls into our framework. Note that this case is exactly the sparse PCA setting studied in Proposition~\ref{prop:spca} which is consistent with the non-rigorous results stated in~\cite{montanari2015finding} (formula (41) for the free energy in~\cite{montanari2015finding} is exactly the right-hand term in Proposition~\ref{prop:spca}). In particular, the non-rigorous results of section 3.2 in~\cite{montanari2015finding} (i.e.\ in the large-degree asymptotics) are made rigorous by our work.

\subsection{An extension of the hard but detectable conjecture}\label{sec:conjecture}

We can now formalize and extend a conjecture emerging in statistical physics~\cite{DBLP:conf/isit/LesieurKZ15},~\cite{krzakala2016mutual} for sparse PCA\@.
The hard but detectable conjecture deals with the case where $\E_{P_0}[X]=0$ and can be stated as follows:
\begin{conjecture}\label{conj:hard}
 For the model~\eqref{defmodel} with $\E_{P_0}[X]=0$ and $\E_{P_0}[X^2]=1$, we define $\lambda_c$ by~\eqref{def:lambdac}. If $\lambda_c<1$ then achieving a better $\MSE$ than the dummy estimator (i.e.\ beating $\DMSE$) is hard for $\lambda\in (\lambda_c, 1)$.
\end{conjecture}

Clearly, Proposition~\ref{prop:lc} only shows that in the regime $\lambda\in (\lambda_c, 1)$ achieving a better $\MSE$ than $\DMSE$ is possible. For $\lambda>1$, we have seen that PCA beats $\DMSE$. Note also, that there are various natural notions of performance for our model (detection, reconstruction) and the conjecture should hold for all of them, see~\cite{banks2016information}. A similar conjecture for the problem of community detection in the symmetric stochastic block model emerged in~\cite{decelle2011asymptotic} and for the non-symmetric case in~\cite{caltagirone2016asymmetric} (see Figure~\ref{fig:phase_diagram}).
\\

Thanks to our Theorem~\ref{th:1}, we know that as soon as $q^*(\lambda) > \E_{P_0}[X]^2$, then it is possible to beat the dummy estimator, i.e.\ achieve a $\MSE$ strictly better than $\DMSE$.
There are now two questions: 
\BIT
\item is it easy to beat $\DMSE$?
\item is it easy to achieve $\MMSE$?
\EIT
Conjecture~\ref{conj:hard} is related to the first question and we now give a general conjecture which deals with both questions by giving the best $\MSE$ achievable efficiently.
For $\eta>0$, consider the following sequence $(q^t_\eta)_{t\in \N}$ defined by:
\begin{equation} \label{eq:rec_amp}
q^0_\eta = \eta, \quad \mbox{and for} \ t\geq 0, \quad q^{t+1}_\eta =\E_{P_0}[X^2]-\mmse(\lambda q^t_\eta).
\end{equation}

Let $\tilde{q}=\lim_{\eta \to 0}\lim_{t\to \infty}q^t_\eta$ (which is always well defined, see below), then we make the following conjecture:
\begin{conjecture}\label{conj:gen}
 For the model~\eqref{defmodel}, the best mean square error that can be achieved efficiently is $\E_{P_0}[X^2]^2-\tilde{q}^2$.
\end{conjecture}
Consequently,
\BIT
\item[(i)] if $\tilde{q} =q^*(\lambda)$, i.e.\ $\tilde{q}$ is the maximizer of $q\geq 0\mapsto \Fcal(\lambda,q)$, then the matrix minimum mean square error can be achieved efficiently (with a polynomial-time algorithm).
\item[(ii)] if $\tilde{q} \neq q^*(\lambda)$, i.e $\tilde{q}$ is not the maximizer of $q\geq 0\mapsto \Fcal(\lambda,q)$, then the best mean square error cannot be achieved efficiently.
\EIT

Let $\langle \cdot \rangle_{\gamma}$ denote the posterior distribution of $X_0$ given $Y_0$ in the scalar channel~\eqref{eq:scalar_channel}: for every continuous bounded function $f$, $\langle f(x) \rangle_{\gamma} = \EE[f(X_0)|Y_0]$. Define $G(\gamma) = \E \langle x X_0 \rangle_{\gamma}$, where $x$ is a sample from the posterior distribution $\langle \cdot \rangle_{\gamma}$, independently of everything else. This means $G(\gamma) = \EE \big[X_0 \EE[X_0|Y_0]\big]$.
The recursion~\eqref{eq:rec_amp} can then be rewritten $\lambda q^{t+1}_\eta = \lambda G(\lambda q^t_\eta)$. A computation shows that
$$
G'(\gamma) = \E \left[\left(\langle x^2 \rangle_{\gamma} - \langle x \rangle_{\gamma}^2\right)^2 \right] \geq 0.
$$
$G$ is thus non-decreasing and bounded, the limit $\tilde{q}=\lim_{\eta \to 0}\lim_{t\to \infty}q^t_\eta$ is well defined (and finite) and is a solution of~\eqref{eq:fixp}.
For $\gamma = 0$, the posterior distribution $\langle \cdot \rangle_{\gamma}$ of the scalar channel~\eqref{eq:scalar_channel} is equal to $P_0$. Therefore $G'(0) = \text{Var}(X_0)^2$.
\\

If $\E_{P_0}[X]=0$, then $0$ is a fixed point of the recursion~\eqref{eq:rec_amp}. In order to investigate its stability, one has thus to compare the quantity $\lambda G'(0) = \lambda \text{Var}_{P_0}(X^2) = \lambda (\E_{P_0} [X^2])^2$ to $1$. As a consequence of Conjecture~\ref{conj:gen}, we obtain:
\begin{itemize}
	\item if $\lambda > (\E_{P_0} [X^2])^{-2}$, then a polynomial-time algorithm can beat DMSE (i.e.\ do better than a dummy estimator). Indeed, in this case, PCA beats DMSE but does not necessary achieves MMSE (see Section~\ref{sec:trans}).
	\item if $\lambda < (\E_{P_0} [X^2])^{-2}$, then no efficient algorithm can beat DMSE (i.e.\ do better than a dummy estimator).
\end{itemize}
In particular we see that Conjecture~\ref{conj:gen} would imply Conjecture~\ref{conj:hard}.
\\

If $\E_{P_0}[X] \neq 0$, then $0$ is not a fixed point of the recursion~\eqref{eq:rec_amp} and Conjecture~\ref{conj:gen} implies that it is always possible to beat DMSE with an efficient algorithm.
\\

In the sparse PCA case where $P_0\sim \Ber(\epsilon)$, (i) has been proved in~\cite{deshpande2014information} for $\epsilon>\epsilon^*$ where AMP is shown to be optimal. Indeed in this case,~\eqref{eq:fixp} has only one solution which is $q^*(\lambda)$ so that $\tilde{q} =q^*(\lambda)$. But more generally, AMP is a candidate algorithm for achieving the best possible $\MSE$ for all values of the parameter $\lambda$ (as conjectured in~\cite{DBLP:conf/isit/LesieurKZ15}). For example, in the sparse PCA case, the analysis done in~\cite{deshpande2014information} and~\cite{montanari2015finding} shows that the performance of AMP gives a $\MSE=\epsilon^2-\tilde{q}^2<\DMSE$.
So that it is always easy to beat $\DMSE$ with a polynomial-time algorithm but we conjecture that the performance of AMP is the best possible achievable by an efficient algorithm.
So that as shown in~\cite{montanari2015finding}, for $\epsilon<\epsilon^*$,  there is a set for the parameters $\lambda$ and $\epsilon$ for which we have $\tilde{q} < q^*(\lambda)$ and we believe that achieving the $\MMSE$ is hard in this case.

\subsection{Proof techniques}\label{sec:proof_tech}

We now present the main general ideas for the proof of Theorems~\ref{th:1} and~\ref{th:2}.
As a first step (Section~\ref{sec:guerraton}), we recall a lower bound on the mutual information that was proved in~\cite{krzakala2016mutual} and follows from an application of Guerra's interpolation technique (see Proposition~\ref{prop:guerra_bound}).

Showing that this lower bound is tight requires some work. Two main ingredients will be particularly useful. The first one is called the Nishimori identity and is true in a very general setting. 
This was discovered by Nishimori (see for instance \cite{nishimori2001statistical}) and extensively used in the context of Bayesian inference, see \cite{iba1999nishimori,korada2010tight,zdeborova2016statistical}.
It express the fact that the planted configuration $\bbf{X}$ behaves like a sample $\bbf{x}$ from the posterior distribution $\PP(\bbf{X}=.|\bbf{Y})$, see Proposition~\ref{prop:nishimori}. The second one will consist in perturbing the original model by revealing a small fraction of the entries of the vector $\bX$. It is known that thanks to this perturbation, the correlations decay  so that the overlap will concentrate (see Proposition~\ref{prop:overlap}). This is again a very general result~\cite{andrea2008estimating}. We then need to show that this perturbation is negligible in the computation of the mutual information (see Section~\ref{sec:small_perturbation}). It remains then to do some ``cavity computations'', a technique introduced by Mezard, Parisi and Virasoro in~\cite{mezard1987spin}.
More precisely we will adapt to our setting the Aizenman-Sims-Starr scheme~\cite{aizenman2003extended} which is standard in the context of the SK model. Theorem~\ref{th:1} then follows.

In order to obtain Theorem~\ref{th:2}, we need to prove that the overlap concentrates without the perturbation induced by the revealed entries. To do so, we follow  an approach closely related to the proof of the Ghirlanda-Guerra identities in the SK model from~\cite{panchenko2013SK} and this is done in Section~\ref{sec:overlap_concentration_without_perturbation}.

As explained just after Theorem~\ref{th:1},~\cite{barbier2016mutual} proves Theorem~\ref{th:1} under some additional conditions on the prior $P_0$. The proof technique in~\cite{barbier2016mutual} is completely different from ours and relies on a careful analysis of the AMP algorithm. As explained above, if $\tilde{q}=q^*(\lambda)$, then AMP is expected to be optimal and to achieve the MMSE providing a proof for the tightness of the bound obtained by Guerra's interpolation technique. In order to deal with the case where $\tilde{q}\neq q^*(\lambda)$,~\cite{barbier2016mutual} introduces an auxiliary spatially coupled system and proves that this system has the same mutual information as the original one while $\tilde{q}=q^*(\lambda)$ on this new system. As opposed to~\cite{barbier2016mutual}, our proof does not rely on the analysis of AMP and is fully contained in this paper (see Section~\ref{sec:proof_rs_formula}).

The recent work~\cite{coja2016information} used similar techniques to prove the replica-symmetric formula in a sparse graph setting. However, restrictive assumptions are required to apply the interpolation scheme to their framework.

\subsection{Finite-rank matrix estimation}\label{sec:introg}

We now generalize our results to any probability distributions $P_0$ over $\R^k$ ($k \in \N^*$ is fixed) with finite second moment.
Consider the following Gaussian observation channel
$$
Y_{i,j} = \sqrt{\frac{\lambda}{n}} \bbf{X}_i^{\intercal} \bbf{X}_j + Z_{i,j} \ \text{for } 1 \leq i<j \leq n
$$
where $\bbf{X}_i \overset{\text{\tiny i.i.d.}}{\sim}  P_0$ and $Z_{i,j} \overset{\text{\tiny i.i.d.}}{\sim}  \mathcal{N}(0,1)$. 
Analogously to the unidimensional case, we define
\begin{align*}
	\MMSE_n(\lambda) 
	&= \min_{\hat{\theta}} \frac{2}{n(n-1)} \sum_{1\leq i<j\leq n} \E\left[ \left(\bbf{X}_i^{\intercal} \bbf{X}_j- \hat{\theta}_{i,j}(\bbf{Y}) \right)^2\right] \label{eq:def_mmse_min_multi} \\
	&=\frac{2}{n(n-1)} \sum_{1\leq i<j\leq n} \E\left[ \left(\bbf{X}_i^{\intercal} \bbf{X}_j-\E\left[\bbf{X}_i^{\intercal} \bbf{X}_j|\bbf{Y} \right]\right)^2\right],
\end{align*}
where the minimum is taken over all estimators $\hat{\theta}$ (i.e.\ measurable functions of the observations $\bbf{Y}$).
We now define
$$
\mathcal{F}: (\lambda,\bbf{q}) \in \R \times S_k^+ \mapsto -\frac{\lambda}{4}\|\bbf{q}\|^2 + \E \log \left[ \int dP_0(\bbf{x}) \exp\left(\sqrt{\lambda} (\bbf{Z}^{\intercal} \bbf{q}^{1/2} \bbf{x})+ \lambda \bbf{x}^{\intercal} \bbf{q} \bbf{X} - \frac{\lambda}{2}\bbf{x}^{\intercal} \bbf{q} \bbf{x}\right)\right]
$$
where $S_k^+$ denote the set of $k \times k$ symmetric positive-semidefinite matrices and $\bbf{Z} \sim \mathcal{N}(0,I_k)$ and $\bbf{X} \sim P_0$ are independent random variables.

\begin{theorem}\label{th:3}
  For $\lambda>0$, we have
  $$
  \lim_{n \rightarrow + \infty} \frac{1}{n}  I(\bbf{X},\bbf{Y}) = \frac{\lambda \|\E_{P_0} (\bbf{X} \bbf{X}^{\intercal})\|^2}{4} - \sup_{\bbf{q} \in S_k^+} \mathcal{F}(\lambda,\bbf{q}),
  $$
  For almost all $\lambda>0$, all the maximizers $\bbf{q}$ of $\bbf{q} \in S_k^+ \mapsto \Fcal(\lambda,\bbf{q})$ have the same norm $\| \bbf{q} \|^2 = q^*(\lambda)^2$ and
  $$
	\MMSE_n(\lambda) \xrightarrow[n \to \infty]{} \|\E_{P_0} \bbf{X} \bbf{X}^{\intercal} \|^2 - q^*(\lambda)^2
	$$
\end{theorem}

Theorem~\ref{th:3} is proved in Section~\ref{sec:gen}.
We refer to~\cite{DBLP:conf/isit/LesieurKZ15} where statistical physics arguments have been used to derive the same expression and explicit computations have been made for the sparse PCA problem of rank $k$.

\section{The Replica-Symmetric formula} \label{sec:rs_formula}

In this section, we connect our problem to a statistical physics model which will be closely related to the SK model.
For simplicity we first concentrate with rank-one matrix estimation, with priors discrete $P_0$. The extension to finite-rank and general priors is done in Section~\ref{sec:gen}.

\subsection{Main results}

Let $P_0$ be a probability distribution with finite support $S\subset [-K_0, K_0]$ for some $K_0 >0$. Consider the following observation channel
$$
Y_{i,j} = \sqrt{\frac{\lambda}{n}} X_i X_j + Z_{i,j}, \ \ \text{for } 1 \leq i<j \leq n
$$
where $X_i \overset{\text{\tiny i.i.d.}}{\sim}  P_0$ and $Z_{i,j} \overset{\text{\tiny i.i.d.}}{\sim}  \mathcal{N}(0,1)$ are independent random variables. 
In the following, $\E$ will denote the expectation with respect to the $\bX$ and $\bbf{Z}$ random variables.
We are going to write, for $\sigma \in S$ and $\bbf{x} \in S^n$, 
$$
\begin{cases}
	P_0(\sigma) = \PP(X=\sigma) & \ \text{where } X \sim P_0 \\
	P_0(\bbf{x}) = \prod_{i=1}^n P_0(x_i)
\end{cases}
$$
The mutual information for this Gaussian channel is (see Lemma~\ref{lem:mutual_information_general})
$$
I(\bbf{X},\bbf{Y}) = - \E \left[ \log \left(  \sum_{\bbf{x} \in S^n} P_0(\bbf{x}) \exp\left( \sum_{i<j} x_i x_j \sqrt{\frac{\lambda}{n}} Z_{i,j} - \frac{\lambda}{2n}(x_i x_j - X_i X_j)^2 \right)\right) \right]
$$
We are interested in computing the limit of $\frac{1}{n} I(\bbf{X},\bbf{Y})$. To do so, it will be more convenient to consider the free energy.  Let us define the random Hamiltonian $H_n(\bbf{x}) = \sum_{i<j} x_i x_j \sqrt{\frac{\lambda}{n}} (Z_{i,j} + \sqrt{\frac{\lambda}{n}} X_i X_j ) - \frac{\lambda}{2n} x_i^2 x_j^2$. The posterior distribution of $\bbf{X}$ given $\bbf{Y}$ is then
\BEA
\label{def:post2}P(\bbf{x} | \bbf{Y}) = \frac{1}{Z_n(\lambda)} P_0(\bbf{x}) e^{\sum_{i<j} x_i x_j \sqrt{\frac{\lambda}{n}} Y_{i,j} - \frac{\lambda}{2n} x_i^2 x_j^2} = \frac{1}{Z_n} P_0(\bbf{x}) e^{H_n(\bbf{x})}
\EEA
We define the free energy as
$$
F_n(\lambda) = \frac{1}{n} \E \Big[ \log( \sum_{\bbf{x} \in S^n} P_0(\bbf{x}) \ e^{H_n(\bbf{x})})\Big] = \frac{1}{n} \E \log Z_n(\lambda)
$$
We will express the limit of $F_n$ using the function
$$
\mathcal{F}: (\lambda,q) \mapsto -\frac{\lambda}{4}q^2 + \E \log \left( \sum_{x \in S} P_0(x) \exp\left(\sqrt{\lambda q}Zx + \lambda q x X - \frac{\lambda}{2}q x^2\right)\right)
$$
where $Z \sim \mathcal{N}(0,1)$ and $X \sim P_0$ are independent random variables.
\begin{theorem}[Replica-Symmetric formula] \label{th:rs_formula}
	\begin{align*}
		\lim_{n \rightarrow + \infty} F_n(\lambda) &= \sup_{q \geq 0} \mathcal{F}(\lambda,q)
	\end{align*}
\end{theorem}
The Replica-Symmetric formula allows us to compute the limit of the mutual information.
\begin{corollary} \label{cor:rs_information}
	$$
	\lim_{n \rightarrow + \infty} \frac{1}{n}  I(\bbf{X},\bbf{Y}) = \frac{\lambda \E_{P_0} (X^2)^2}{4} - \sup_{q \geq 0} \mathcal{F}(\lambda,q)
	$$
\end{corollary}

\begin{proof} Lemma~\ref{lem:mutual_information_general} gives
	\begin{align*}
		\frac{1}{n} I(\bbf{X},\bbf{Y}) &= -F_n - \frac{1}{n} \E \left[\log \exp(-\frac{\lambda}{2n} \sum_{i<j} X_i^2 X_j^2)\right] 
		= \frac{\lambda}{2n^2} \sum_{i<j} \E [X_i^2 X_j^2] - F_n
		= \frac{\lambda (n-1)}{4n} \E_{P_0}[X^2]^2 - F_n
	\end{align*}
	and Theorem~\ref{th:rs_formula} gives the result.
\end{proof}

\subsection{Consequences of the RS formula} \label{sec:consequences_rs}

We define $\phi: \lambda \mapsto \sup_{q \geq 0} \mathcal{F}(\lambda,q)$. 
$\phi$ is the limit of $\lambda \mapsto F_n(\lambda)$, which is convex. $\phi$ is therefore convex and is thus differentiable everywhere except on a countable set of points. Let $D \subset (0, + \infty)$ be the set of points where $\phi$ is differentiable.
One can easily show that this supremum is achieved over the compact set $[0,\E_{P_0}[X^2]]$. Thus, Corollary~4 from \cite{milgrom2002envelope} (combined with the arguments in the proof below) gives
$$
D = \left\{ \lambda > 0 \ \middle| \ \mathcal{F}(\lambda,\cdot) \ \text{has a unique maximizer} \right\}
$$

\begin{proposition} \label{prop:derivative_phi}
	For all $\lambda \in D$, the maximizer $q^*(\lambda)$ of $q \geq 0 \mapsto \mathcal{F}(\lambda,q)$ is unique and 
	$$
	\phi'(\lambda) = \frac{q^*(\lambda)^2}{4}
	$$
\end{proposition}

\begin{proof}
	One can rewrite $\phi$, using the change of variables $q' = \lambda q$, we have for all $\lambda >0$
	\begin{align*}
		\phi(\lambda) &= \sup_{q' \geq 0} -\frac{q'^2}{4 \lambda} + \E_{X,Z} \log \Big[ \sum_{x \in S} P_0(x) \exp(\sqrt{q'}Zx + q' x X - \frac{q'}{2} x^2)\Big] 
		= \sup_{q' \geq 0} \psi(\lambda,q')
	\end{align*}
	where $\psi(\lambda,q') = \mathcal{F}(\lambda,\frac{q'}{\lambda})$.
	Let $\lambda \in D$. $\phi$ is differentiable at $\lambda$, the envelope theorem from~\cite{milgrom2002envelope} gives us that for all $q' \in \text{argmax}_{q' \geq 0} \psi(\lambda,q')$
	$$
	\phi'(\lambda) = \frac{\partial \psi}{\partial \lambda}(q')= \frac{q'^2}{4\lambda^2}.
	$$
	Thus, the maximizer of $q'\geq 0 \mapsto \psi(\lambda,q')$ is unique. Using the change of variables $q' = \lambda q$, one has that the maximizer $q^*(\lambda)$ of $q \geq 0 \mapsto \mathcal{F}(\lambda,q)$ is also unique and verifies $\phi'(\lambda) = \frac{q^*(\lambda)^2}{4}$.
\end{proof}
\\

Let $\langle \cdot \rangle$ denote the Gibbs measure corresponding to the Hamiltonian $H_n$. This means that for any function $f$ on $S^n$ we have
$$
\langle f(\bbf{x}) \rangle = \frac{1}{Z_n(\lambda)} \sum_{\bbf{x} \in S^n} P_0(\bbf{x}) f(\bbf{x}) e^{H_n(\bbf{x})}.
$$
We recall that $\bx$ is a random sample with distribution $\langle \cdot \rangle$ defined in~\eqref{def:post2}. We also recall the notations: for $\bbf{u},\bbf{v} \in S^n$ we write
\begin{align*}
	\bbf{u}.\bbf{v} = \frac{1}{n} \sum_{i=1}^n u_i v_i \ \ \ \text{and } \ \
	\|\bbf{u}\| = \sqrt{\frac{1}{n} \sum_{i=1}^n (u_i)^2 }
\end{align*}
We call $\bbf{u}.\bbf{v}$ the \textit{overlap} between the configurations $\bbf{u}$ and $\bbf{v}$. Theorem~\ref{th:overlap_concentration_without_perturbation} shows that $q^*(\lambda)$ can be interpreted as the overlap between a random sample (such samples are called \textit{replicas}) $\bbf{x}$ from $\langle \cdot \rangle$ and the planted configuration $\bbf{X}$. 

The Nishimori property, as mentioned in Section~\ref{sec:proof_tech}, is a fundamental identity that will be used repeatedly and is true in a general setting.
It express the fact that the planted configuration $\bbf{X}$ behaves like a replica $\bbf{x}$ sampled from the posterior distribution $\PP(\bbf{X}=.|\bbf{Y})$.

\begin{proposition}[Nishimori identity] \label{prop:nishimori}
	Let $(\bbf{X},\bbf{Y})$ be a couple of random variables on a polish space. Let $k \geq 1$ and let $\bbf{x}^{(1)}, \dots, \bbf{x}^{(k)}$ be $k$ i.i.d.\ samples (given $\bbf{Y}$) from the distribution $\PP(\bbf{X}=. | \bbf{Y})$, independently of every other random variables. Let us denote $\langle \cdot \rangle$ the expectation with respect to $\PP(\bbf{X}=. | \bbf{Y})$ and $\E$ the expectation with respect to $(\bX,\bY)$. Then, for all continuous bounded function $f$
	$$
	\E \langle f(\bbf{Y},\bbf{x}^{(1)}, \dots, \bbf{x}^{(k)}) \rangle
	=
	\E \langle f(\bbf{Y},\bbf{x}^{(1)}, \dots, \bbf{x}^{(k-1)}, \bbf{X}) \rangle
	$$
\end{proposition}

\begin{proof}
	It is equivalent to sample the couple $(\bbf{X},\bbf{Y})$ according to its joint distribution or to sample first $\bbf{Y}$ according to its marginal distribution and then to sample $\bbf{X}$ conditionally to $\bbf{Y}$ from its conditional distribution $\PP(\bbf{X}=.|\bbf{Y})$. Thus the $(k+1)$-tuple $(\bbf{Y},\bbf{x}^{(1)}, \dots,\bbf{x}^{(k)})$ is equal in law to $(\bbf{Y},\bbf{x}^{(1)},\dots,\bbf{x}^{(k-1)},\bbf{X})$.

\end{proof}

\noindent We obtain an important corollary for the estimation of $\bbf{X} \bbf{X}^{\intercal}$ from the observations $\bbf{Y}$.

\begin{corollary} \label{cor:limit_mmse}
	For all $\lambda \in D$,
	$$
	\MMSE_n(\lambda) \xrightarrow[n \to \infty]{} (\E_{P_0} X^2)^2 - q^*(\lambda)^2
	$$
\end{corollary}

\begin{proof}
	$\lambda \mapsto F_n(\lambda)$ is differentiable with derivative
	\begin{align*}
		F_n'(\lambda) &= \frac{1}{n} \E \Big\langle \sum_{i<j} \frac{1}{2 \sqrt{\lambda n}} Z_{i,j} x_i x_j + \frac{x_i x_j X_i X_j}{n} - \frac{(x_i x_j)^2}{2n} \Big\rangle 
													   = \frac{1}{2 n^2} \E \Big\langle \sum_{i<j} x_i x_j X_i X_j\Big\rangle
													   = \frac{(n-1)}{4 n} \E \langle x_i x_j \rangle^2
	\end{align*}
	where we used Gaussian integration by parts and the Nishimori identity (Proposition~\ref{prop:nishimori}). 
	$\lambda \mapsto F_n(\lambda)$ is thus convex.
	Recall the following standard lemma:
	\begin{lemma} \label{lem:deriv_convex}
		Let $I \subset \R$ and $f_n$ a sequence of convex, differentiable functions over $I$. Suppose that for all $x \in I, \ f_n(x) \xrightarrow[n \to \infty]{} f(x)$. Let $D_f= \{x \in I | f \ \text{is differentiable in }x\}$. Then
		$$
		\forall x \in D_f, \ f_n'(x) \xrightarrow[n \to \infty]{} f'(x)
		$$
	\end{lemma}
Using Proposition~\ref{prop:derivative_phi} and the previous lemma we obtain that for all $\lambda \in D$,
$$
\frac{1}{2 n^2} \E \Big\langle \sum_{i<j} x_i x_j X_i X_j\Big\rangle
\xrightarrow[n \to \infty]{} \frac{q^*(\lambda)}{4}
$$
	Therefore, for $\lambda \in D$,
	\begin{align*}
		\MMSE_n(\lambda) &= \frac{2}{n(n-1)} \sum_{1 \leq i < j \leq n} \E \Big[ (X_i X_j - \E[X_i X_j | \bbf{Y}])^2 \Big]
		\\
		&=
		(\E_{P_0} X^2)^2 + \frac{2}{n(n-1)} \sum_{i<j} \E \Big[ \langle x_i x_j \rangle^2 - 2 \langle x_i x_j X_i X_j \rangle \Big] 
		\\
		&=
		(\E_{P_0} X^2)^2 - \frac{2}{n(n-1)} \sum_{i<j} \E \langle x_i x_j X_i X_j \rangle
		\xrightarrow[n \to \infty]{} (\E_{P_0} X^2)^2 - q^{*}(\lambda)
	\end{align*}
\end{proof}
\\

The study of $\lambda \in D \mapsto q^*(\lambda)$ is therefore of crucial importance. The next proposition states its main properties.
We recall that the minimum mean square error $\mmse(\gamma)$ for the scalar channel \eqref{eq:scalar_channel} is defined in equation~\eqref{eq:mmsescalar}.

\begin{proposition}[Properties of $q^*(\lambda)$] \label{prop:q_star}
	\begin{enumerate}[(i)]
		\item The function $\lambda \in D \mapsto q^*(\lambda)$ is non-decreasing.
		\item For all $\lambda \in D$, $q^*(\lambda) = \E_{P_0}(X^2) - \mmse(\lambda q^*(\lambda))$.
		\item $q^*(\lambda) \xrightarrow[\overset{\lambda \in D}{\lambda \to 0}]{} \E_{P_0}(X)^2$.
		\item $q^*(\lambda) \xrightarrow[\overset{\lambda \in D}{\lambda \to + \infty}]{} \E_{P_0}(X^2)$.
	\end{enumerate}
\end{proposition}
\begin{proof}
	The function $\phi$ is convex, so (i) is simply a consequence of Proposition~\ref{prop:derivative_phi}.
	To prove (ii), we remark that equation~\eqref{eq:i_scalar} implies that $\mathcal{F}(\lambda,q)= \frac{\lambda q}{2} \big(\E [X_0^2] - \frac{q}{2} \big) - \i(\lambda q)$. \cite{guo2005mutual} gives us $\i'(\gamma) = \frac{1}{2} \mmse(\gamma)$ so that
	$$
	\frac{\partial}{\partial q} \mathcal{F}(\lambda,q) = \frac{\lambda}{2} \Big( \E[X_0^2] - q - \mmse(\lambda q) \Big)
	$$
	Thus, if $q^*(\lambda) > 0$, then $\frac{\partial}{\partial q} \mathcal{F}(\lambda,q^*(\lambda)) = 0$ and (ii) is verified. 
	Suppose now that $q^*(\lambda) = 0$.
$q \mapsto \mathcal{F}(\lambda,q)$ achieves therefore its maximum in $0$: $\frac{\partial}{\partial q} \mathcal{F}(\lambda,0) \leq 0$.
	We have obviously $\mmse(0) = \E[X_0^2] - \E[X_0]^2$, so that $\frac{\partial}{\partial q} \mathcal{F}(\lambda,0) = \frac{\lambda}{2} \E[X_0]^2 \geq 0$. Consequently $\frac{\partial}{\partial q} \mathcal{F}(\lambda,q^*(\lambda)) = 0$ and (ii) is verified.  

	It remains to prove (iii) and (iv). We first notice that Corollary~\ref{cor:limit_mmse} implies that $q^*(\lambda) \in [0, \E[X_0^2]]$. One can verify easily that $\mmse(\gamma) \xrightarrow[\gamma \to 0]{} \E[X_0^2] - \E[X_0]^2$. Using (ii), this implies (iii).
	Similarly, one have $\mmse(\gamma) \xrightarrow[\gamma \to + \infty]{} 0$. Thus, we only have to show that $q^*(\lambda) > 0$ for $\lambda$ large enough to obtain (iv). If this is not verified, then (i) implies that $q^*(\lambda) = 0$ for all $\lambda \in D$. This implies also that $\E[X_0]=0$ because of (iii). Therefore, for all $q'\geq 0$ and all $\lambda \in D$, one would have
	$$
	\frac{\gamma}{2}(\E[X_0^2] - \frac{\gamma}{2 \lambda}) - \i(\gamma) \leq 0
	$$
	By letting $\lambda \to \infty$, one would obtain that $\i(\gamma) \geq \frac{\gamma}{2} \E[X_0^2]$. However, we have $\i(\gamma) \leq \frac{\gamma}{2} \E[X_0^2]$ because $\i'(\gamma) = \frac{1}{2} \mmse(\gamma)$ and $\mmse(\gamma) \leq \E[X_0^2]$ and $\i(0)=0$. Therefore one would have that for all $\gamma \geq 0$, $\i(\gamma) = \frac{\gamma}{2} \E[X_0^2]$ and thus $\mmse(\gamma) = \E[X_0^2]$. This is incompatible with $\mmse(\gamma) \xrightarrow[\gamma \to + \infty]{} 0$. This concludes the proof.
\end{proof}
\\

We will refine the result of Corollary~\ref{cor:limit_mmse} and show that the square of the overlap between two replicas (or equivalently the overlap between a replica and the planted configuration, because of Proposition~\ref{prop:nishimori}) concentrates around $q^*(\lambda)^2$.

\begin{theorem} \label{th:overlap_concentration_without_perturbation}
	For all $\lambda \in D$,
	$$
	\E \Big\langle \big((\bbf{x}^{(1)}.\bbf{x}^{(2)})^2 - q^*(\lambda)^2\big)^2 \Big\rangle \xrightarrow[n \to \infty]{} 0
	$$
\end{theorem}

The proof of this result is closely related to the proof of the Ghirlanda-Guerra identities in the SK model from~\cite{panchenko2013SK} and is done in Section~\ref{sec:overlap_concentration_without_perturbation}.

\section{Proof of the Replica-Symmetric formula (Theorem~\ref{th:rs_formula})} \label{sec:proof_rs_formula}

\subsection{The lower bound: Guerra's interpolation method}\label{sec:guerraton}

The following result comes from~\cite{krzakala2016mutual}. This is an application of Guerra's interpolation technique (see~\cite{guerra2003broken}). We reproduce the proof for the sake of completeness.

\begin{proposition} \label{prop:guerra_bound}
	\begin{equation} \label{eq:guerra_bound}
		\liminf_{n \to \infty} F_n \geq \sup_{q \geq 0} \mathcal{F}(\lambda,q)
	\end{equation}
\end{proposition}

\begin{proof}
	Let $q \geq 0$. For $t \in [0,1]$ we define
	$$
	H_n(\bbf{x},t) = \sum_{i<j} \sqrt{\frac{\lambda t}{n}} Z_{i,j} x_i x_j + \frac{\lambda t}{n} x_i x_j X_i X_j - \frac{\lambda t}{2n} x_i^2 x_j^2
	+ \sum_{i=1}^n \sqrt{(1-t)\lambda q} Z_i' x_i + (1-t)\lambda q x_i X_i - \frac{(1-t) \lambda q}{2} x_i^2
	$$
	$\langle \cdot \rangle_t$ will denote the Gibbs measure associated with the Hamiltonian $H_n(\bbf{x},t)$. Remark that $\langle \cdot \rangle_t$ correspond to the posterior distribution of $\bbf{X}$ conditionally to $\bbf{Y}$ and $\bbf{Y}'$ in the following inference channel:
	$$
	\begin{cases}
		Y_{i,j} = \sqrt{\frac{\lambda t}{n}} X_i X_j + Z_{i,j} & \ \text{for } 1 \leq i < j \leq n \\
		Y'_{i} = \sqrt{(1-t) \lambda q} X_i + Z'_{i} & \ \text{for } 1 \leq i \leq n
	\end{cases}
	$$
	where $X_i \overset{\text{\tiny i.i.d.}}{\sim}  P_0$ and $Z_{i,j},\ Z'_i \overset{\text{\tiny i.i.d.}}{\sim}  \mathcal{N}(0,1)$ are independent random variables. We will therefore be able to apply the Nishimori property (property~\ref{prop:nishimori}) with the Gibbs measure $\langle \cdot \rangle_t$. Let us define
	$$
	\psi: t \in [0,1] \mapsto \frac{1}{n} \E \log \sum_{\bbf{x} \in S^n} P_0(\bbf{x}) e^{H_n(\bbf{x},t)}
	\vspace{-0.3cm}
	$$
	We have $\psi(1)=F_n$ and 
	\vspace{-0.2cm}
	\begin{align*}
		\psi(0) 
		&= \frac{1}{n} \E \log \sum_{\bbf{x} \in S^n} P_0(\bbf{x}) 
		\exp\left( \sum_{i=1}^n \sqrt{\lambda q} Z_i' x_i + \lambda q x_i X_i - \frac{\lambda q}{2} x_i^2 \right) \\
		&= \frac{1}{n} \E \log \prod_{i=1}^n \left( \sum_{x_i \in S} P_0(x_i) 
	\exp\left( \sqrt{\lambda q} Z_i' x_i + \lambda q x_i X_i - \frac{\lambda q}{2} x_i^2\right) \right) \\
	&= \mathcal{F}(\lambda,q) + \frac{\lambda q^2}{4}
\end{align*}
$\psi$ is continuous, differentiable on $(0,1)$. For $0<t<1$,
\begin{align} \label{eq:deriv_psi}
	\psi'(t) = \frac{1}{n} \E \left\langle 
		\sum_{i<j} \frac{\sqrt{\lambda}}{2\sqrt{nt}} Z_{i,j} x_i x_j + \frac{\lambda}{n} x_i x_j X_i X_j - \frac{\lambda}{2n} x_i^2 x_j^2
		- \sum_{i=1}^n \frac{\sqrt{\lambda q}}{2 \sqrt{1-t}} Z_i' x_i + \lambda q x_i X_i - \frac{\lambda q}{2} x_i^2
	\right\rangle_t
\end{align}
For $1 \leq i < j \leq n$ we have, by Gaussian integration by parts and by the Nishimori property
\begin{align*}
	\E Z_{i,j} \Big\langle \frac{\sqrt{\lambda}}{2\sqrt{nt}} x_i x_j \Big\rangle_t
	&=
	\frac{\lambda}{2n} \Big( \E \langle x_i^2 x_j^2 \rangle_t - \E \langle x_i x_j \rangle_t^2 \Big)
	=
	\frac{\lambda}{2n} \Big( \E \langle x_i^2 x_j^2 \rangle_t - \E \langle x^{(1)}_i x^{(1)}_j x^{(2)}_i x^{(2)}_j \rangle_t \Big)
	\\
	&=
	\frac{\lambda}{2n} \Big( \E \langle x_i^2 x_j^2 \rangle_t - \E \langle x_i x_j X_i X_j \rangle_t \Big)
\end{align*}
Similarly, we have for $1 \leq i \leq n$
$$
\E \left\langle \frac{\sqrt{\lambda q}}{2 \sqrt{1-t}} Z_i' x_i \right\rangle_t = 
\frac{\lambda q}{2} \left( \E \langle x_i^2 \rangle_t - \E \langle x_i X_i \rangle_t \right)
$$
Therefore equation~\eqref{eq:deriv_psi} simplifies
\begin{align*}
	\psi'(t) &= \frac{1}{n} \E \Big\langle 
		\sum_{i<j} \frac{\lambda}{2n} x_i x_j X_i X_j 
		- \sum_{i=1}^n  \frac{\lambda q}{2} x_i X_i 
	\Big\rangle_t
	= \frac{\lambda}{4} \E \Big\langle (\bbf{x}.\bbf{X})^2 - 2 q \bbf{x}.\bbf{X} \Big\rangle_t + o(1)
	\\
	&=\frac{\lambda}{4} \E \Big\langle (\bbf{x}.\bbf{X} - q)^2 \Big\rangle_t - \frac{\lambda q^2}{4} + o(1) \geq -\frac{\lambda q^2}{4} + o(1)
\end{align*}
where $o(1)$ denotes a quantity that goes to $0$ uniformly in $t \in (0,1)$. Then
$$
F_n - \mathcal{F}(\lambda,q) -\frac{\lambda}{4} q^2 = \psi(1) - \psi(0)
= \int_0^1 \psi'(t) dt 
\geq - \frac{\lambda}{4} q^2 + o(1)
$$
Thus $\liminf\limits_{n \to \infty} F_n \geq \mathcal{F}(\lambda,q)$, for all $q \geq 0$.
\end{proof}

\subsection{Adding a small perturbation} \label{sec:small_perturbation}

It remains therefore to prove the converse bound of~\eqref{eq:guerra_bound}. 
As in the case of the SK model (see~\cite{panchenko2013SK} and~\cite{talagrand2010meanfield2}), it will be convenient to add a small perturbation to our Hamiltonian $H_n$. This is particularly useful to obtain identities involving the distribution of the overlaps under the Gibbs measure.
As we will see later in Section~\ref{sec:overlap_concentration}, this perturbation will force the overlaps to concentrate around their expectations. 
In our context of Bayesian estimation, adding additional observations will induce a perturbation in our Hamiltonian.

Let us fix $\epsilon\in [0,1]$, and suppose we have access to the additional information, for $1 \leq i \leq n$
$$
Y'_i =
\begin{cases}
	X_i &\text{if } L_i = 1 \\
	* &\text{if } L_i = 0
\end{cases}
$$
where $L_i \overset{\text{\tiny i.i.d.}}{\sim}  \Ber(\epsilon)$ and $*$ is a value that does not belong to $S$. 
The posterior distribution of $\bbf{X}$ is now
$$
\PP(\bbf{X}=\bbf{x}| \bbf{Y},\bbf{Y}') = \frac{1}{Z_{n,\epsilon}} \left(\prod_{i | Y'_i \neq *}1(x_i=Y'_i) \right) \left( \prod_{i | Y_i'=*} P_0(x_i) \right) e^{H_n(\bbf{x})}
$$
where $Z_{n,\epsilon}$ is the appropriate normalization constant.
For $\bbf{x} \in S^n$ we define the following (very convenient) notation
\begin{equation} \label{eq:bar}
	\bbf{\bar{x}} = (\bar{x}_1, \dots, \bar{x}_n) = (L_1 X_1 + (1-L_1) x_1, \dots, L_n X_n + (1-L_n)x_n)
\end{equation}
$\bbf{\bar{x}}$ is thus obtained by replacing the coordinates of $\bbf{x}$ that are revealed by $\bbf{Y}'$ by their revealed values. The notation $\bbf{\bar{x}}$ will allow us to obtain a very convenient expression for the free energy of the perturbed model which is defined as
$$
F_{n,\epsilon} = \frac{1}{n} \E \log Z_{n,\epsilon} = \frac{1}{n} \E \Big[ \log( \sum_{\bbf{x} \in S^n} P_0(\bbf{x}) \exp(H_{n}(\bbf{\bar{x}})))\Big] 
$$

\begin{proposition} \label{prop:approximation_f_n_epsilon}
	For all $n \geq 1$ and all $\epsilon,\epsilon' \in [0,1]$, we have
	$$
	 |F_{n,\epsilon} - F_{n,\epsilon'} | \leq \lambda K_0^4 |\epsilon - \epsilon'|.
	$$
\end{proposition}

\begin{proof}
	We are going to bound the derivative of $f:\epsilon \mapsto F_{n,\epsilon}$.
	To do so, we are going to consider a slightly more general model where the probability of revealing $X_i$ depends on $i$: for $1 \leq i \leq n$,
	$L_i \sim \Ber(\epsilon_i)$. We will show that $|\frac{\partial F_{n,\epsilon}}{\partial \epsilon_i}| \leq \lambda K_0^4$ for $i \in \{1, \dots, n\}$. By symmetry between the variables, it suffices to control $\frac{\partial F_{n,\epsilon}}{\partial \epsilon_1}$. Notice that
	\begin{align*}
		F_{n,\epsilon}
	&= \epsilon_1 \E\left[ \frac{1}{n} \log( \sum_{\bbf{x} \in S^n} P_0(\bbf{x}) \exp(H_{n}(\bbf{\bar{x}}))) \Big| L_1 = 1 \right]
	+ (1- \epsilon_1) \E\left[ \frac{1}{n} \log( \sum_{\bbf{x} \in S^n} P_0(\bbf{x}) \exp(H_{n}(\bbf{\bar{x}}))) \Big| L_1 = 0 \right]
	\\
	&= \epsilon_1 \E\left[ \frac{1}{n} \log( \sum_{\bbf{x} \in S^n} P_0(\bbf{x}) \exp(H_{n}(X_1,\bar{x}_2, \dots, \bar{x}_n))) \right]
	+ (1- \epsilon_1) \E\left[ \frac{1}{n} \log( \sum_{\bbf{x} \in S^n} P_0(\bbf{x}) \exp(H_{n}(x_1,\bar{x}_2, \dots, \bar{x}_n))) \right]
	\end{align*}
	Consequently
	\begin{align*}
		\frac{\partial F_{n,\epsilon}}{\partial \epsilon_1}
	&= \E\left[ \frac{1}{n} \log( \sum_{\bbf{x} \in S^n} P_0(\bbf{x}) \exp(H_{n}(X_1,\bar{x}_2, \dots, \bar{x}_n))) \right]
	 - \E\left[ \frac{1}{n} \log( \sum_{\bbf{x} \in S^n} P_0(\bbf{x}) \exp(H_{n}(x_1,\bar{x}_2, \dots, \bar{x}_n))) \right]
	\end{align*}
	Define the Hamiltonian $\tilde{H}_n(\bx)= \sum_{2 \leq i<j} \sqrt{\frac{\lambda}{n}} Z_{i,j} \bar{x}_i \bar{x}_j + \frac{\lambda}{n} \bar{x}_i \bar{x}_j X_i X_j - \frac{\lambda}{2n} \bar{x}_i^2 \bar{x}_j^2$. Let $\langle \cdot \rangle$ denote the Gibbs measure (on $S^{n-1}$) corresponding to the Hamiltonian $\tilde{H}_n$. We can rewrite
	\begin{align*}
		\frac{\partial F_{n,\epsilon}}{\partial \epsilon_1}
	&= \frac{1}{n} \E \log \Big\langle \exp(\sum_{2 \leq j \leq n}\sqrt{\frac{\lambda}{n}} Z_{1,j} X_1 \bar{x}_j + \frac{\lambda}{n} X_1^2 \bar{x}_j X_j - \frac{\lambda}{2n} X_1^2 \bar{x}_j^2) \Big\rangle
		\\
		&\quad - \frac{1}{n} \E \log \Big\langle \sum_{x_1 \in S} P_0(x_1) \exp(\sum_{2 \leq j \leq n}\sqrt{\frac{\lambda}{n}} Z_{1,j} x_1 \bar{x}_j + \frac{\lambda}{n} x_1 X_1 \bar{x}_j X_j - \frac{\lambda}{2n} x_1^2 \bar{x}_j^2) \Big\rangle
	\end{align*}
	where $(x_i)_{2 \leq i \leq n}$ is sampled from $\langle \cdot \rangle$, independently of everything else.
Let $\E_1$ denote the expectation with respect to the variables $(Z_{1,j})_{2 \leq j \leq n}$ only.
By Jensen's inequality
\begin{align*}
		\frac{\partial F_{n,\epsilon}}{\partial \epsilon_1}
		& \leq
\frac{1}{n} \E \log \Big\langle \E_1 \exp(\sum_{2 \leq j \leq n}\sqrt{\frac{\lambda}{n}} Z_{1,j} X_1 \bar{x}_j + \frac{\lambda}{n} X_1^2 \bar{x}_j X_j - \frac{\lambda}{2n} X_1^2 \bar{x}_j^2) \Big\rangle
\\
& \leq
\frac{1}{n} \E \log \Big\langle \exp(\sum_{2 \leq j \leq n} \frac{\lambda}{n} X_1^2 \bar{x}_j X_j ) \Big\rangle
\leq \frac{\lambda K_0^4}{n}
\end{align*}
Analogously one have $\frac{\partial F_{n,\epsilon}}{\partial \epsilon_1} \geq - \frac{\lambda K_0^4}{n}$. Consequently, for all $i \in \{1, \dots,n\}$, $\big| \frac{\partial F_{n,\epsilon}}{\partial \epsilon_i} \big| \leq \frac{\lambda K_0^4}{n}$.
This implies that $|f'| \leq \lambda K_0^4$ and proves the lemma. 
\end{proof}
\\

We define now $\epsilon$ as a uniform random variable over $[0, 1]$, independently of every other random variable. We will note $\E_{\epsilon}$ the expectation with respect to $\epsilon$. For $n \geq 1$, we define also $\epsilon_n = n^{-1/2} \epsilon \sim \mathcal{U}[0, n^{-1/2}]$. Proposition~\ref{prop:approximation_f_n_epsilon} implies that
$$
\big| F_n - \E_{\epsilon} [F_{n, \epsilon_n}] \big| \xrightarrow[n \to \infty]{} 0.
$$
It remains therefore to compute the limit of the free energy averaged over small perturbations.

  \subsection{Aizenman-Sims-Starr scheme} \label{sec:aizenman}
  
The Aizenman-Sims-Starr scheme was introduced in~\cite{aizenman2003extended} in the context of the SK model.
This is what physicists call a ``cavity computation'': one compare the system with $n+1$ variables to the system with $n$ variables and see what happen to the $(n+1)^{\text{th}}$ variable we add.
\\

\noindent With the convention $F_{0,\epsilon_0} = 0$, we have
$F_{n,\epsilon_n} = \frac{1}{n} \sum\limits_{k=0}^{n-1} (k+1) F_{k+1, \epsilon_{k+1}} - k F_{k,\epsilon_k}$ which ensures that
\vspace{-0.2cm}
\begin{equation*} 
	\limsup_{n \to \infty} \E_{\epsilon} [F_{n,\epsilon_n}]
	\leq
	\limsup_{n \to \infty} \E_{\epsilon} \Big[ (n+1) F_{n+1, \epsilon_{n+1}} - n F_{n,\epsilon_n} \Big]
	= \limsup_{n \to \infty} \E_{\epsilon} \Big[ (n+1) F_{n+1, \epsilon_{n}} - n F_{n,\epsilon_n} \Big]
\end{equation*}
because, by Proposition~\ref{prop:approximation_f_n_epsilon}, $|F_{n+1,\epsilon_{n+1}} - F_{n+1,\epsilon_n}| \leq \lambda K_0^4 |\epsilon_{n+1} - \epsilon_n|= O(n^{-3/2})$. Define
$$
A^{(0)}_{n} = (n+1) F_{n+1, \epsilon_{n}} - n F_{n,\epsilon_n} = \E [ \log(Z_{n+1,\epsilon_n}) ] - \E [ \log(Z_{n,\epsilon_n}) ]
$$
where we recall that $Z_{n,\epsilon_n} = \sum_{\bbf{x} \in S^n} P_0(\bbf{x}) e^{H_{n}(\bbf{\bar{x}})}$ where the notation $\bbf{\bar{x}}$ is defined by equation~\eqref{eq:bar}. 
\begin{equation} \label{eq:limsup1}
	\limsup_{n \to \infty} F_n =
	\limsup_{n \to \infty} \E_{\epsilon} [F_{n,\epsilon_n}]
	\leq
	\limsup_{n \to \infty} \E_{\epsilon} [A_{n}^{(0)}]
\end{equation}
Now we are going to compare $H_{n+1}$ with $H_{n}$. Let $\bbf{x} \in S^n$ and $\sigma \in S$. $\sigma$ plays the role of the $(n+1)^{\text{th}}$ variable. 
We decompose $H_{n+1}(\bbf{x},\sigma) = H_{n}'(\bbf{x}) + \sigma z_0(\bbf{x}) + \sigma^2 s_0(\bbf{x})$, where
\vspace{-0.1cm}
\begin{align*}
	H_{n}'(\bbf{x}) &= \sum_{1 \leq i<j \leq n} \sqrt{\frac{\lambda}{n+1}} Z_{i,j} x_i x_j + \frac{\lambda}{n+1} X_i X_j x_i x_j  - \frac{\lambda}{2(n+1)} x_i^2 x_j^2 
	\\
	z_0(\bbf{x}) &= \sum_{i=1}^n \sqrt{\frac{\lambda}{n+1}} Z_{i,n+1} x_i  + \frac{\lambda}{n+1} X_i X_{n+1} x_i \\
	s_0(\bbf{x}) &= -\frac{\lambda}{2(n+1)} \sum_{i=1}^n x_i^2 
\end{align*}
Let $(\tilde{Z}_{i,j})_{1\leq i < j \leq n}$ be independent, standard Gaussian random variables, independent of all other random variables. We have then $H_{n}(\bbf{x}) = H_{n}'(\bbf{x}) + y_0(\bbf{x})$ in law, where
$$
y_0(\bbf{x}) = \sum_{1 \leq i<j \leq n} \frac{\sqrt{\lambda}}{\sqrt{n(n+1)}} \tilde{Z}_{i,j} x_i x_j + \frac{\lambda}{n(n+1)} X_i X_j x_i x_j - \frac{\lambda}{2(n+1)n} x_i^2 x_j^2
$$
We recall that the notation $\bbf{\bar{x}}$ is defined in equation~\eqref{eq:bar} and define analogously $\bar{\sigma} = (1-L_{n+1}) \sigma + L_{n+1} X_{n+1}$. We can thus rewrite
\begin{align*}
	\E [ \log(Z_{n+1,\epsilon_n}) ] &= \E \log \Big( \sum_{\bbf{x} \in S^n} P_0(\bbf{x}) e^{H_{n}'(\bbf{\bar{x}})} \Big[\sum_{\sigma \in S} P_0(\sigma) \exp(\bar{\sigma} z_0(\bbf{\bar{x}}) + \bar{\sigma}^2 s_0(\bbf{\bar{x}}) ) \Big]  \Big) 
\end{align*}
We now define the Gibbs measure $\langle \cdot \rangle$ by
\begin{equation} \label{eq:def_gibbs}
	\langle f(\bbf{x}) \rangle = \frac{1}{Z_{n,\epsilon_n}} \sum_{\bbf{x} \in S^n} P_0(\bbf{x}) f(\bbf{\bar{x}}) \exp(H_n'(\bbf{\bar{x}}))
\end{equation}
for any function $f$ on $S^n$. We have then
\begin{align*}
	A^{(0)}_{n} = \E \log \Big\langle  \sum_{\sigma \in S} P_0(\sigma) \exp\big(\bar{\sigma} z_0(\bbf{x}) + \bar{\sigma}^2 s_0(\bbf{x}) \big) \Big\rangle
	- \E  \log \Big\langle \exp(y_0(\bbf{x})) \Big\rangle 
\end{align*}
It will be more convenient to use ``simplified'' versions of $z_0, s_0$ and $y_0$. We define
\begin{align*}
	z(\bbf{x}) &= \sum_{i=1}^n \sqrt{\frac{\lambda}{n}} Z_{i,n+1} x_i  + \frac{\lambda}{n} X_i X_{n+1} x_i = \sqrt{\frac{\lambda}{n}} \sum_{i=1}^n x_i Z_{i,n+1} + \lambda (\bbf{x}.\bbf{X}) X_{n+1} \\
	s(\bbf{x}) &= -\frac{\lambda}{2n} \sum_{i=1}^n x_i^2 = -\frac{\lambda}{2} \|\bbf{x}\|^2 \\
	y(\bbf{x}) &= \frac{\sqrt{\lambda}}{\sqrt{2}n} \sum_{i=1}^n Z_i'' x_i^2 
	+\frac{\lambda}{2 n^2} \sum_{i=1}^n \left(x_i^2 X_i^2 - \frac{x_i^4}{2}\right)
	+ \frac{\sqrt{\lambda}}{n} \sum_{1 \leq i<j \leq n} x_i x_j \left(\tilde{Z}_{i,j} 
	+ \frac{\sqrt{\lambda}}{n} X_i X_j \right) - \frac{\lambda}{2n^2} x_i^2 x_j^2
	\\
	&= \frac{\sqrt{\lambda}}{\sqrt{2}n} \sum_{i=1}^n Z_i'' x_i^2 
	+ \frac{\sqrt{\lambda}}{n} \sum_{1 \leq i<j \leq n} x_i x_j \tilde{Z}_{i,j} 
	+ \frac{\lambda}{2} \left((\bbf{x}.\bbf{X})^2 - \frac{1}{2} (\bbf{x}.\bbf{x})^2\right)
\end{align*}
where $Z_i'' \overset{\text{\tiny i.i.d.}}{\sim}  \mathcal{N}(0,1)$ independently of any other random variables. Define now
\begin{align*}
	A_{n} = \E \log \left\langle  \sum_{\sigma \in S} P_0(\sigma)\exp(\bar{\sigma} z(\bbf{x}) + \bar{\sigma}^2 s(\bbf{x}) ) \right\rangle
	- \E  \log \left\langle \exp(y(\bbf{x})) \right\rangle 
\end{align*}
Using Gaussian interpolation techniques, it is easy to show that $\E_{\epsilon} |A_{n} - A^{(0)}_{n} | \xrightarrow[n \to \infty]{} 0$, which ensure (using equation~\eqref{eq:limsup1}) that
\begin{equation} \label{eq:limsup2}
	\limsup_{n \rightarrow \infty} F_n
	\leq
	\limsup_{n \rightarrow \infty} \E_{\epsilon} [A_{n}]
\end{equation}

\subsection{Overlap concentration} \label{sec:overlap_concentration}

We will see in this section that the small perturbation that we considered in Section~\ref{sec:small_perturbation} forces the overlaps to concentrate.
Recall that $\langle \cdot \rangle$ is the Gibbs measure defined in equation~\eqref{eq:def_gibbs}. $\langle \cdot \rangle$ correspond to the posterior distribution of $\bbf{X}$ given $\bbf{Y}$ and $\bbf{Y'}$ in the following observation channel
\begin{align*}
	Y_{i,j} &= \sqrt{\frac{\lambda}{n+1}} X_i X_j + Z_{i,j}, \ \ \text{for } \ 1 \leq i<j \leq n \\
	Y'_i &=
	\begin{cases}
		X_i &\text{if } L_i = 1 \\
		* &\text{if } L_i = 0
	\end{cases} \ \text{for} \ 1 \leq i \leq n
\end{align*} 
where $X_i \overset{\text{\tiny i.i.d.}}{\sim}  P_0$, $Z_{i,j} \overset{\text{\tiny i.i.d.}}{\sim}  \mathcal{N}(0,1)$ and $L_i \overset{\text{\tiny i.i.d.}}{\sim}  \Ber(\epsilon_n)$ are independent random variables. The Nishimori property (Proposition~\ref{prop:nishimori}) will thus be valid under $\langle \cdot \rangle$.

The following lemma comes from~\cite{andrea2008estimating} (lemma 3.1). It shows that the extra information $\bbf{Y}'$ forces the correlations to decay.

\begin{lemma}
	$$
	n^{-1/2} \E_{\epsilon}  \left[ \frac{1}{n^2} \sum_{1 \leq i,j \leq n} I(X_i;X_j | \bbf{Y},\bbf{Y'}) \right] \leq \frac{2 H(P_0)}{n} 
	$$
\end{lemma}

This implies that the overlap between two replicas, i.e.\ two independent samples $\bbf{x}^{(1)}$ and $\bbf{x}^{(2)}$ from the Gibbs distribution $\langle \cdot \rangle$, concentrates. Let us define
\begin{align*}
	Q &= \Big\langle \frac{1}{n} \sum_{i=1}^n x^{(1)}_i x^{(2)}_i \Big\rangle \\
	b_i &= \langle x_i \rangle
\end{align*}
$Q$ is a random variable depending only on $(Y_{i,j})_{i<j\leq n}$ and $(Y_i')_{i \leq n}$. Notice that $Q = \frac{1}{n} \sum_i b_i^2 \geq 0$. 

\begin{proposition}[Overlap concentration] \label{prop:overlap}
	\begin{align*}
		\E_{\epsilon} \E \Big\langle \big( \frac{1}{n} \sum_{i=1}^n x^{(1)}_i x^{(2)}_i - Q \big)^2 \Big\rangle \xrightarrow[n \to \infty]{} 0
	\end{align*}
\end{proposition}

\begin{proof}
	\begin{align*}
		\big\langle (\bbf{x}^{(1)}.\bbf{x}^{(2)} - Q)^2 \big\rangle &=
		\langle (\bbf{x}^{(1)}.\bbf{x}^{(2)})^2 \rangle - \langle \bbf{x}^{(1)}.\bbf{x}^{(2)} \rangle^2 
		= \frac{1}{n^2} \sum_{1 \leq i,j \leq n} \langle x_i^{(1)} x_i^{(2)} x_j^{(1)} x_j^{(2)} \rangle - \langle x_i^{(1)} x_i^{(2)} \rangle \langle x_j^{(1)} x_j^{(2)} \rangle \\
		&= \frac{1}{n^2} \sum_{1 \leq i,j \leq n} \langle x_i x_j\rangle^2 - \langle x_i \rangle^2 \langle x_j\rangle^2
		\leq \frac{C}{n^2} \sum_{1 \leq i,j \leq n} | \langle x_i x_j\rangle - \langle x_i \rangle \langle x_j\rangle | \\
		&\leq \frac{C}{n^2} \sum_{1 \leq i,j \leq n} \big| \sum_{x_i,x_j} x_i x_j \PP(X_i=x_i, X_j=x_j | \bbf{Y},\bbf{Y'})- x_i x_j \PP(X_i=x_i | \bbf{Y},\bbf{Y'}) \PP(X_j = x_j |\bbf{Y},\bbf{Y'}) \big| \\
		&\leq \frac{C'}{n^2} \sum_{1 \leq i,j \leq n} \Dtv \big(\PP(X_i=., X_j=. |\bbf{Y},\bbf{Y'}); \PP(X_i=. | \bbf{Y},\bbf{Y'})\otimes \PP(X_j =. |\bbf{Y},\bbf{Y'}) \big)\\
		&\leq \frac{C''}{n^2} \sum_{1 \leq i,j \leq n} \sqrt{\Dkl \big(\PP(X_i=., X_j=. |\bbf{Y},\bbf{Y'}); \PP(X_i=. | \bbf{Y},\bbf{Y'})\otimes \PP(X_j =. |\bbf{Y},\bbf{Y'}) \big)}\\
		&\leq C'' \sqrt{\frac{1}{n^2} \sum_{1 \leq i,j \leq n} \Dkl \big(\PP(X_i=., X_j=. |\bbf{Y},\bbf{Y'}); \PP(X_i=. | \bbf{Y},\bbf{Y'})\otimes \PP(X_j =. |\bbf{Y},\bbf{Y'}) \big)}
	\end{align*}
	for some constants $C,C',C'' >0$, where we used Pinsker's inequality to compare the total variation distance $\Dtv$ with the Kullback-Leibler divergence $\Dkl$. So that:
	\begin{align*}
		\E_{\epsilon} \E \Big\langle \big( \frac{1}{n} \sum_{i=1}^n x^{(1)}_i x^{(2)}_i - Q \big )^2 \Big\rangle \leq
		C'' \sqrt{\E_{\epsilon} \Big[ \frac{1}{n^2} \sum_{1 \leq i,j \leq n} I(X_i;X_j | \bbf{Y},\bbf{Y'}) \Big]}
		\ \xrightarrow[n \to \infty]{} 0
	\end{align*}
\end{proof}
\\

As a consequence of the Nishimori property, the overlap between one replica and the planted solution concentrates around the same value as the overlap between two independent replicas.
\begin{corollary} \label{cor:overlap}
	$$
	\E_{\epsilon} \E \Big\langle (\bbf{x}.\bbf{X}-Q)^2 \Big\rangle \xrightarrow[n \to \infty]{} 0
	\ \ \ \text{and } \ \ 
	\E_{\epsilon} \E \Big\langle (\bbf{x}.\bbf{b}-Q)^2 \Big\rangle \xrightarrow[n \to \infty]{} 0
	$$
\end{corollary}
\begin{proof}
	The first limit is an application of the Nishimori property~\ref{prop:nishimori} and Proposition~\ref{prop:overlap}. For the second one,
	$$
	(\bbf{x}^{(1)}.\bbf{b} - Q)^2 = \langle \bbf{x}^{(1)}.\bbf{x}^{(2)} - Q\rangle^2 \leq \big\langle (\bbf{x}^{(1)}.\bbf{x}^{(2)} - Q)^2  \big\rangle
	$$
	where the Gibbs measure $\langle \cdot \rangle$ is only with respect to the variable $\bbf{x}^{(2)}$. Proposition~\ref{prop:overlap} concludes the proof.

\end{proof}

\subsection{The main estimate}

Let us denote, for $\epsilon \in [0,1]$,
$$
\mathcal{F}_{\epsilon}: (\lambda,q) \mapsto -\frac{\lambda}{4}q^2 + \epsilon (\E_{P_0} X^2) \frac{\lambda q}{2} + (1-\epsilon)\E \log \Big[ \sum_{x \in S} P_0(x) \exp(\sqrt{\lambda q}Zx + \lambda q x X - \frac{\lambda}{2}q x^2)\Big]
$$
where the expectation $\E$ is taken with respect to the independent random variables $X \sim P_0$ and $Z \sim \mathcal{N}(0,1)$.
The following proposition is one of the key steps of the proof.
\begin{proposition} \label{prop:main_estimate}
	$$
	\lim_{n \to \infty} \E_{\epsilon} \Big| A_{n} - \E \mathcal{F}_{\epsilon_n}(\lambda,Q)\Big| = 0
	$$
\end{proposition}
The proof of Proposition~\ref{prop:main_estimate} is reported to Section~\ref{sec:proof_main_estimate}. We deduce here Theorem~\ref{th:rs_formula} from Proposition~\ref{prop:main_estimate} and the results of the previous sections. Because of Proposition~\ref{prop:guerra_bound}, we only have to show that $\limsup\limits_{n \to \infty} F_n \leq \sup\limits_{q \geq 0} \mathcal{F}(\lambda,q)$.
\\

We have by equation~\eqref{eq:limsup2} and Proposition~\ref{prop:main_estimate}
$$
\limsup_{n \to \infty} F_n \leq \limsup_{n \to \infty} \E_{\epsilon} [A_{n}] = \limsup_{n \to \infty} \E_{\epsilon} \E \mathcal{F}_{\epsilon_n}(\lambda,Q).
$$
It remains therefore to show that $\limsup\limits_{n \to \infty} \E_{\epsilon} \E \mathcal{F}_{\epsilon_n}(\lambda,Q) \leq \sup\limits_{q \geq 0} \mathcal{F}(\lambda,q)$.
We have for $\epsilon \in [0,1]$,
\vspace{-0.2cm}
\begin{align*}
	\sup_{q \in [0,K_0^2]} \left| \mathcal{F}_{\epsilon}(\lambda,q) - \mathcal{F}(\lambda,q) \right| 
	&\leq \epsilon \sup_{q \in [0,K_0^2]} \Big(
\frac{\lambda q}{2}(\E_{P_0} X^2)  + \Big| \E \log \Big[ \sum_{x \in S} P_0(x) \exp(\sqrt{\lambda q}Zx + \lambda q x X - \frac{\lambda}{2}q x^2)\Big] \Big| \Big)
\\
&\leq C \epsilon
\end{align*}
where $C$ is a constant independent of $\epsilon$. Noticing that $Q \in [0, K_0^2]$ a.s., we have then $| \E \mathcal{F}_{\epsilon}(\lambda,Q) - \E \mathcal{F}(\lambda,Q) | \leq C \epsilon_0$, for all $\epsilon \in [0,\epsilon_0]$ and therefore
$$
\E_{\epsilon} \E \mathcal{F}_{\epsilon_n}(\lambda,Q) \leq \E_{\epsilon} \E \mathcal{F}(\lambda,Q) + C n^{-1/2}
\leq \sup_{q \geq 0}\mathcal{F}(\lambda,q) + C n^{-1/2}
$$
which implies $\limsup\limits_{n \to \infty} \E_{\epsilon} \E \mathcal{F}_{\epsilon_n}(\lambda,Q) \leq \sup\limits_{q \geq 0} \mathcal{F}(\lambda,q)$. Theorem~\ref{th:rs_formula} is proved.

\subsection{Proof of Proposition~\ref{prop:main_estimate}} \label{sec:proof_main_estimate}

In this section, we prove Proposition~\ref{prop:main_estimate}. This will be a consequence of the following Lemmas~\ref{lem:part1} and~\ref{lem:part2}.
In order to lighten the formulas, we will use the following notations
\begin{align*}
	X' = X_{n+1} \ \ \ \text{and } \ \ \
	Z'_i = Z_{i,n+1}
\end{align*}
Recall
\begin{equation} \label{eq:a_n_proof}
	A_{n} = \E \log \Big\langle  \sum_{\sigma \in S} P_0(\sigma)\exp(\bar{\sigma} z(\bbf{x}) + \bar{\sigma}^2 s(\bbf{x})) \Big\rangle
	- \E  \log \big\langle \exp(y(\bbf{x})) \big\rangle 
\end{equation}
Where we recall that for $\sigma \in S$, $\bar{\sigma}= (1-L_{n+1}) \sigma + L_{n+1} X'$. We are going to compute the asymptotic of $A_{n}$. The computations here are closely related to the cavity computations in the SK model, see for instance~\cite{talagrand2010meanfield1}.

\subsubsection{First part}
In this section, we deal with the first term in equation~\eqref{eq:a_n_proof}. Indeed, we prove the following lemma.

\begin{lemma} \label{lem:part1}
	\begin{align*}
		&\E_{\epsilon} \Big| \E \log \Big\langle  \sum_{\sigma \in S} P_0(\sigma)\exp(\bar{\sigma} z(\bbf{x}) + \bar{\sigma}^2 s(\bbf{x}) ) \Big\rangle
		\\
		&- \Big( \epsilon_n (\E_{P_0} X^2) \E \frac{\lambda Q}{2} + (1-\epsilon_n) \E \log \sum_{\sigma \in S} P_0(\sigma) \exp \big( \sqrt{\lambda Q} \sigma Z_0 + \lambda Q \sigma X' - \frac{\lambda \sigma^2}{2} Q \big) \Big) \Big| \xrightarrow[n \to \infty]{} 0
	\end{align*}
	where $Z_0 \sim \mathcal{N}(0,1)$ is independent of all other random variables.
\end{lemma}
We write also $f(z,s) = \sum\limits_{\sigma \in S} P_0(\sigma) e^{\bar{\sigma} z + \bar{\sigma}^2 s}$ and we define:
\begin{align*}
	U_1 &= \big\langle f(z(\bbf{x}),s(\bbf{x})) \big\rangle \\
	V_1 
	&= \sum_{\sigma \in S}P_0(\sigma) \exp\left(\bar{\sigma} \sqrt{\frac{\lambda}{n}} \sum_{i=1}^n b_i Z_i' + \lambda Q X' \bar{\sigma} - \frac{\lambda Q}{2} \bar{\sigma}^2 \right)
\end{align*}

\begin{lemma}
	$$
	\E_{\epsilon} \E \Big[ (U_1-V_1)^2 \Big] \xrightarrow[n \to \infty]{} 0
	$$
\end{lemma}

\begin{proof}
	We are going to show successively that $\E_{\epsilon} |\E U_1^2 -\E V_1^2| \xrightarrow[n \to \infty]{} 0$ and $\E_{\epsilon} |\E U_1 V_1 -\E V_1^2| \xrightarrow[n \to \infty]{} 0$. 
	\\
	We will write $\E_{\bbf{Z}'}$ to denote the expectation with respect to $\bbf{Z}'$ only.
	Let us compute
	\begin{align}
		\E_{\bbf{Z'}} V_1^2 &= \E_{\bbf{Z'}} \sum_{\sigma_1,\sigma_2 \in S} P_0(\sigma_1,\sigma_2) \exp \big(
			(\bar{\sigma}_1 + \bar{\sigma}_2) \sqrt{\frac{\lambda}{n}} \sum_{i=1}^n b_i Z_i' + \lambda Q X' (\bar{\sigma}_1 + \bar{\sigma}_2) - \frac{\lambda Q}{2} (\bar{\sigma}_1^2 + \bar{\sigma}_2^2) 
		\big) \nonumber \\
		&= \sum_{\sigma_1,\sigma_2 \in S} P_0(\sigma_1,\sigma_2) \exp \big(
		(\bar{\sigma}_1 + \bar{\sigma}_2)^2 \frac{\lambda}{2} Q + \lambda Q X' (\bar{\sigma}_1 + \bar{\sigma}_2) - \frac{\lambda Q}{2} (\bar{\sigma}_1^2 + \bar{\sigma}_2^2) 
	\big) \nonumber \\
	&= \sum_{\sigma_1,\sigma_2 \in S} P_0(\sigma_1,\sigma_2) \exp \big(
	\bar{\sigma}_1 \bar{\sigma}_2 \lambda Q + \lambda Q X' (\bar{\sigma}_1 + \bar{\sigma}_2) 
\big) \label{eq:esp_v2}
	\end{align}
	where we write for $i=1,2$, $\bar{\sigma}_i = (1-L_{n+1})\sigma_i + L_{n+1}X'$, as before.
	\\

	\noindent\textbf{Step 1:} $\E_{\epsilon} | \E U_1^2 -\E V_1^2 | \xrightarrow[n \to \infty]{} 0$
	\begin{align*}
		\E_{\bbf{Z'}} U_1^2 &= \E_{\bbf{Z'}} \langle f(z(\bbf{x}),s(\bbf{x})) \rangle^2 \\
																			   &= \E_{\bbf{Z'}} \langle f(z(\bbf{x}^{(1)}),s(\bbf{x}^{(1)})) f(z(\bbf{x}^{(2)}),s(\bbf{x}^{(2)})) \rangle \ \ \text{(where } \bbf{x}^{(1)} \ \text{and } \bbf{x}^{(2)} \ \text{are independent samples from } \langle \cdot \rangle \text{)} \\
																			&= \left\langle \E_{\bbf{Z'}} f(z(\bbf{x}^{(1)}),s(\bbf{x}^{(1)})) f(z(\bbf{x}^{(2)}),s(\bbf{x}^{(2)})) \right\rangle \\
																		 &= \left\langle 
		\sum_{\sigma_1,\sigma_2 \in S} P_0(\sigma_1,\sigma_2) \E_{\bbf{Z'}} \exp\Big(\bar{\sigma}_1 z(\bbf{x}^{(1)}) + \bar{\sigma}_1^2 s(\bbf{x}^{(1)}) 
		+ \bar{\sigma}_2 z(\bbf{x}^{(2)}) + \bar{\sigma}_2^2 s(\bbf{x}^{(2)}) \Big)
	\right\rangle
\end{align*}
The next lemma follows from the simple fact that for $N \sim \mathcal{N}(0,1)$ and $t \in \R$, $\E e^{tN} = \exp(\frac{t^2}{2})$.
\begin{lemma} \label{lem:esp_comp1}
	Let $\bbf{x}^{(1)}, \bbf{x}^{(2)} \in S^n$ and $\sigma_1, \sigma_2 \in S$ be fixed. Then
	\begin{align*}
		\E_{\bbf{Z'}} \exp\left( \sigma_1 \sqrt{\frac{\lambda}{n}} \sum_{i=1}^n x^{(1)}_i Z_i' + \sigma_2 \sqrt{\frac{\lambda}{n}} \sum_{i=1}^n x^{(2)}_i Z_i'  \right) 
		=
		\exp\left(
			\lambda \sigma_1 \sigma_2 \bbf{x}^{(1)}.\bbf{x}^{(2)} 
			+\frac{1}{2} \lambda \sigma_1^2 \|\bbf{x}^{(1)}\|^2
			+\frac{1}{2} \lambda \sigma_2^2 \|\bbf{x}^{(2)}\|^2
		\right)
	\end{align*}
\end{lemma}
Thus, for all $\bbf{x}^{(1)}, \bbf{x}^{(2)} \!\in\! S^n$ and $\sigma_1, \sigma_2 \! \in \! S$, using Lemma~\ref{lem:esp_comp1} and the fact that $s(\bbf{x}) \!=\! -\frac{\lambda}{2} \|\bbf{x}\|^2$ for all $\bbf{x} \!\in\! S^n$,
\begin{align*}
	\E_{\bbf{Z'}} \!\exp\!\Big(\bar{\sigma}_1 z(\bbf{x}^{(1)}) + \bar{\sigma}_1^2 s(\bbf{x}^{(1)})
	+ \bar{\sigma}_2 z(\bbf{x}^{(2)}) + \bar{\sigma}_2^2 s(\bbf{x}^{(2)}) \!\Big)
		\!= \exp\!\Big(\lambda \bar{\sigma}_1 \bar{\sigma}_2 \bbf{x}^{(1)}.\bbf{x}^{(2)} \!+ \lambda X' (\bar{\sigma}_1 (\bbf{x}^{(1)}.\bbf{X}) + \bar{\sigma}_2 (\bbf{x}^{(2)}.\bbf{X})) \!\Big)
	\end{align*}
	We have therefore
	$$
	\E_{\bbf{Z'}} U_1^2 = \left\langle \sum_{\sigma_1,\sigma_2 \in S} P_0(\sigma_1,\sigma_2) \exp\left(\lambda \bar{\sigma}_1 \bar{\sigma}_2 \bbf{x}^{(1)}.\bbf{x}^{(2)} + \lambda X' \left(\bar{\sigma}_1 (\bbf{x}^{(1)}.\bbf{X}) + \bar{\sigma}_2 (\bbf{x}^{(2)}.\bbf{X})\right) \right) \right\rangle
	$$
	Define
	\begin{align*}
		F_1:(s,r_1,r_2) \mapsto &
		\sum_{\sigma_1,\sigma_2 \in S}P_0(\sigma_1,\sigma_2) \exp\Big(\lambda \bar{\sigma}_1 \bar{\sigma}_2 s
		+ \lambda X'(\bar{\sigma}_1 r_1 + \bar{\sigma}_2 r_2) \Big)
	\end{align*}
	We have $\E_{\bbf{Z'}} U_1^2 = \big\langle F_1(\bbf{x}^{(1)}.\bbf{x}^{(2)},\bbf{x}^{(1)}.\bbf{X},\bbf{x}^{(2)}.\bbf{X}) \big\rangle$. 

	\begin{lemma} \label{lem:f_1_lipschitz}
		There exists a constant $L_{0}$ such that $F_1$ is almost surely $L_0$-Lipschitz.
	\end{lemma}
	\begin{proof}
		$F_1$ is a random function that depends only on the random variables $\epsilon_n$, $X'$ and $L_{n+1}$ (because of $\bar{\sigma}_1$ and $\bar{\sigma}_2$). $F_1$ is $\mathcal{C}^1$ on the compact $[-K_0^2, K_0^2]^3$. An easy computation show that
		$$
		\forall (s,r_1,r_2) \in [-K_0^2,K_0^2]^3, \ \|\nabla F_1(s,r_1,r_2) \| \leq 3 \lambda K_0^4 \exp(3 \lambda K_0^4)
		$$
		$F_1$ is thus $L_0$-Lipschitz with $L_0 = 3 \lambda K_0^4 \exp(3 \lambda K_0^4)$.
	\end{proof}
	\\
	
	\noindent Using Lemma~\ref{lem:f_1_lipschitz} we obtain
	\begin{align*}
		\Big\langle & | F_1(\bbf{x}^{(1)}.\bbf{x}^{(2)},\bbf{x}^{(1)}.\bbf{X},\bbf{x}^{(2)}.\bbf{X}) - F_1(Q,Q,Q) | \Big\rangle \leq 
		L_0 \Big\langle \sqrt{(\bbf{x}^{(1)}.\bbf{x}^{(2)} - Q)^2 + (\bbf{x}^{(1)}.\bbf{X} - Q)^2 + (\bbf{x}^{(2)}.\bbf{X} - Q)^2} \Big\rangle
	\end{align*}
	We recall equation~\eqref{eq:esp_v2} to notice that $F_1(Q,Q,Q) = \E_{\bbf{Z'}} V_1^2$, thus, using Proposition~\ref{prop:overlap} and Corollary~\ref{cor:overlap}
	\begin{align*}
		\E_{\epsilon} \E | \E_{\bbf{Z'}} U_1^2 - \E_{\bbf{Z'}} V_1^2 |
		\leq
		L_0 \E_{\epsilon} \E \Big\langle \sqrt{(\bbf{x}^{(1)}.\bbf{x}^{(2)} - Q)^2 + (\bbf{x}^{(1)}.\bbf{X} - Q)^2 + (\bbf{x}^{(2)}.\bbf{X} - Q)^2 } \Big\rangle
		\xrightarrow[n \to \infty]{} 0
	\end{align*}

	\vspace{0.5cm}

	\noindent\textbf{Step 2:} $\E_{\epsilon} | \E U_1 V_1 -\E V_1^2 | \xrightarrow[n \to \infty]{} 0$
	\begin{align*}
		\E_{\bbf{Z'}} U_1 V_1 &= \left\langle \E_{\bbf{Z'}} \sum_{\sigma_1,\sigma_2 \in S} P_0(\sigma_1,\sigma_2)
			\exp\left( \bar{\sigma}_1 z(\bbf{x}) + \bar{\sigma}_1 s(\bbf{x}) + \bar{\sigma}_2 \sqrt{\frac{\lambda}{n}} \sum_{i=1}^n b_i Z_i' + \lambda Q X' \bar{\sigma}_2 - \frac{\lambda Q}{2} \bar{\sigma}_2^2 
			\right)
		\right\rangle
	\end{align*}
	Applying Lemma~\ref{lem:esp_comp1} with $\bbf{x}^{(1)}=\bbf{x}$ and $\bbf{x}^{(2)}=\bbf{b}$ (and using that $\|\bbf{b}\|^2=Q$) one has
	\begin{align*}
		\E_{\bbf{Z'}} \!\exp\!\left(\! \bar{\sigma}_1 z(\bbf{x}) + \bar{\sigma}_1 s(\bbf{x}) + \bar{\sigma}_2 \sqrt{\frac{\lambda}{n}} \sum_{i=1}^n b_i Z_i' + \lambda Q X' \bar{\sigma}_2 - \frac{\lambda Q}{2} \bar{\sigma}_2^2 
		\!\right) 
		\!= \exp\!\left(\lambda \bar{\sigma}_1 \bar{\sigma}_2 \bbf{x}.\bbf{b} + \bar{\sigma}_1 \lambda (\bbf{x}.\bbf{X})X' + \bar{\sigma}_2 \lambda Q X' \right)
	\end{align*}
	Therefore,
	\vspace{-0.2cm}
	\begin{align*}
		\E_{\bbf{Z'}} U_1 V_1 &= \left\langle \sum_{\sigma_1,\sigma_2 \in S} P_0(\sigma_1,\sigma_2)
			\exp\big(\lambda \bar{\sigma}_1 \bar{\sigma}_2 \bbf{x}.\bbf{b} + \bar{\sigma}_1 \lambda (\bbf{x}.\bbf{X})X' + \bar{\sigma}_2 \lambda Q X' \big)
		\right\rangle
	\end{align*}
	We can thus identify
	\vspace{-0.2cm}
	\begin{align*}
		\E_{\bbf{Z'}} U_1 V_1 &= \langle F_1(\bbf{x}.\bbf{b},\bbf{x}.\bbf{X},Q) \rangle  \\
		\E_{\bbf{Z'}} V_1^2 &= F_1(Q,Q,Q)
	\end{align*}
	Again, using Lemma~\ref{lem:f_1_lipschitz} and the concentration result Corollary~\ref{cor:overlap} we obtain
	\begin{align*}
		\E_{\epsilon} \E | E_{\bbf{Z}'} U_1 V_1 - E_{\bbf{Z}'}V_1^2 | \xrightarrow[n \to \infty]{} 0
	\end{align*}
	This concludes the proof.
\end{proof}
\\

\noindent We have now $|\log U_1 - \log V_1 | \leq \max(U_1^{-1},V_1^{-1}) | U_1 - V_1|$. Thus, using Cauchy-Schwarz inequality,
$$
\E | \log U_1 - \log V_1 | \leq \sqrt{\E U_1^{-2} + \E V_1^{-2}} \sqrt{\E (U_1 - V_1)^2}
$$

\begin{lemma} \label{lem:u_v_moment}
	There exists a constant $C$ such that
	$$
	\E U_1^{-2} + \E V_1^{-2} \leq C
	$$
\end{lemma}

\begin{proof}
	Using Jensen inequality, we have $U_1 \geq f(\langle z(\bbf{x}) \rangle, \langle s(\bbf{x}) \rangle)$. Then
	$$
	U_1^{-2} \leq f(\langle z(\bbf{x}) \rangle, \langle s(\bbf{x}) \rangle)^{-2} \leq \sum_{\sigma \in S} P_0(\sigma) \exp(-2 \bar{\sigma} \langle z(\bbf{x}) \rangle - 2 \bar{\sigma}^2 \langle s(\bbf{x}) \rangle ) 
	$$
	It remains to show that for any value of $\sigma$, $\E \exp(-2 \bar{\sigma} \langle z(\bbf{x}) \rangle - 2 \bar{\sigma}^2 \langle s(\bbf{x}) \rangle)$ is bounded (independently of $n$ and $\epsilon_n$). $P_0$ as a bounded support, therefore
	$$
	\E \exp(-2 \bar{\sigma} \langle z(\bbf{x}) \rangle - 2 \bar{\sigma}^2 \langle s(\bbf{x}) \rangle) \leq C_0 \E \exp(-2 \bar{\sigma} \sum_{i=1}^n \sqrt{\frac{\lambda}{n}} \langle x_i \rangle Z_i') = C_0 \E \exp(2 \lambda Q \bar{\sigma}^2) \leq C_1
	$$
	for some constant $C_1$. The same kind of arguments shows that $\E V_1^{-2}$ is bounded by a constant.

\end{proof}
\\

\noindent Using the previous lemma
$
\E_{\epsilon} \E | \log U_1 - \log V_1 | \xrightarrow[n \to \infty]{} 0
$. We now compute $\E \log V_1$ explicitly.
\begin{lemma} \label{lem:esp_log_v1}
	$$
	\E \log V_1 = \epsilon_n (\E_{P_0} X^2) \E \frac{\lambda Q}{2} + (1-\epsilon_n) \E \log \sum_{\sigma \in S} P_0(\sigma) \exp \big( \sigma \sqrt{\frac{\lambda}{n}} \sum_{i=1}^n b_i Z'_i + \lambda Q \sigma X' - \frac{\lambda \sigma^2}{2} Q \big)
	$$
\end{lemma}
\begin{proof}
	It suffices to distinguish the cases $L_{n+1}\!=\!0$ and $L_{n+1}\!=\!1$. If $L_{n+1}\!=\!1$ then for all $\sigma \in S$, $\bar{\sigma}=X'$ and
	\begin{align*}
		\log V_1 &= \log \Big( \exp(X' \sqrt{\frac{\lambda}{n}} \sum_{i=1}^n b_i Z'_i + \lambda Q X'^2 - \frac{\lambda X'^2}{2} Q) \Big)
		\\
		&= X' \sqrt{\frac{\lambda}{n}} \sum_{i=1}^n b_i Z'_i + \frac{\lambda X'^2}{2} Q 
	\end{align*}
	$L_{n+1}$ is independent of all other random variables, thus
	$$
	\E \Big[ 1(L_{n+1}=1) \log V_1 \Big] = \epsilon_n (\E_{P_0} X^2) \frac{\lambda}{2} \E Q
	$$
	because the $Z_i'$ are centered, independent from $X'$ and because $X'$ is independent from $Q$. The case $L_{n+1}=0$ is obvious.

\end{proof}
\\

\noindent The variables $(b_i)_{1 \leq i \leq n}$ and $(Z_i')_{1 \leq i \leq n}$ are independent. Recall that $Q = \frac{1}{n} \sum\limits_{i=1}^n b_i^2$. Therefore, 
\vspace{-0.4cm}
$$
(Q, \frac{1}{\sqrt{n}} \sum_{i=1}^n b_i Z_i') = (Q, \sqrt{Q} Z_0)  \ \ \text{in law }
\vspace{-0.2cm}
$$
where $Z_0 \sim \mathcal{N}(0,1)$ is independent of all other random variables. The expression of $\E \log V_1$ from Lemma~\ref{lem:esp_log_v1} simplifies
$$
\E \log V_1 = \epsilon_n (\E_{P_0} X^2) \E \frac{\lambda Q}{2} + (1-\epsilon_n) \E \log \sum_{\sigma \in S} P_0(\sigma) \exp \big( \sqrt{\lambda Q} \sigma Z_0 + \lambda Q \sigma X' - \frac{\lambda \sigma^2}{2} Q \big)
\vspace{-0.4cm}
$$
thus
$$
\E_{\epsilon} \Big| \E \log U_1 - \Big( \epsilon_n (\E_{P_0} X^2) \E \frac{\lambda Q}{2} + (1-\epsilon_n) \E \log \sum_{\sigma \in S} P_0(\sigma) \exp \big( \sqrt{\lambda Q} \sigma Z_0 + \lambda Q \sigma X' - \frac{\lambda \sigma^2}{2} Q \big) \Big) \Big| \xrightarrow[n \to \infty]{} 0
$$
This proves Lemma~\ref{lem:part1}.

\subsubsection{Second part}

In this section, we handle the second term of equation~\eqref{eq:a_n_proof}. The arguments are similar to the previous section. We show here the following lemma.

\begin{lemma} \label{lem:part2}
	$$
\E_{\epsilon} \left| \E \log \big\langle \!\exp(y(\bbf{x})) \big\rangle - \frac{\lambda}{4} \E  Q^2 \right| \xrightarrow[n \to  \infty]{} 0
$$
	\end{lemma}
	Define
	\begin{align*}
		U_2 &= \big\langle \exp(y(\bbf{x})) \big\rangle \\
		V_2 &= \exp\left(\frac{\sqrt{\lambda}}{\sqrt{2}n} \sum_{i=1}^n b_i^2 Z_i'' + \frac{\sqrt{\lambda}}{n}\sum_{1 \leq i<j \leq n} b_i b_j \tilde{Z}_{i,j} 
		+ \frac{\lambda}{4}Q^2 \right)
	\end{align*}
and recall
	\begin{align*}
		y(\bbf{x}) 
		&= \frac{\sqrt{\lambda}}{\sqrt{2}n} \sum_{i=1}^n Z_i'' x_i^2 
		+ \frac{\sqrt{\lambda}}{n} \sum_{1 \leq i<j \leq n} x_i x_j \tilde{Z}_{i,j} 
		+ \frac{\lambda}{2} ((\bbf{x}.\bbf{X})^2 - \frac{1}{2} (\bbf{x}.\bbf{x})^2)
	\end{align*}

	\begin{lemma}
		$$
		\E_{\epsilon} \E \Big[ (U_2-V_2)^2 \Big] \xrightarrow[n \to \infty]{} 0
		$$
	\end{lemma}

	\begin{proof}
		We are going to show successively that $\E_{\epsilon} |\E U_2^2 -\E V_2^2| \xrightarrow[n \to \infty]{} 0$ and $\E_{\epsilon} |\E U_2 V_2 -\E V_2^2| \xrightarrow[n \to \infty]{} 0$. We write $\E_{\bbf{Z}'',\bbf{\tilde{Z}}}$ to denote the expectation with respect to the random variables $\bbf{Z}''$ and $\bbf{\tilde{Z}}$ only. Let us compute
		\begin{align}
			\E_{\bbf{Z''},\bbf{\tilde{Z}}} V_2^2 &= \E_{\bbf{Z''},\bbf{\tilde{Z}}} \exp(\frac{\sqrt{2 \lambda}}{n} \sum_{i=1}^n b_i^2 Z_i'' + \frac{2 \sqrt{\lambda}}{n}\sum_{i<j \leq n} b_i b_j \tilde{Z}_{i,j} 
			+ \frac{\lambda}{2}Q^2 ) \nonumber \\
			&= \exp(\frac{\lambda}{n^2} \sum_{i=1}^n b_i^4 + \frac{2 \lambda}{n^2} \sum_{i<j\leq n}b_i^2 b_j^2 + \frac{\lambda Q^2}{2})
			= \exp(\frac{\lambda}{n^2} \sum_{1 \leq i,j \leq n}b_i^2 b_j^2 + \frac{\lambda Q^2}{2}) \nonumber \\
			&= \exp(\frac{3 \lambda Q^2}{2}) \label{eq:esp_v2_bis}
		\end{align}
		because $\frac{1}{n} \sum\limits_{i=1}^n b_i^2 = Q$.
		\\

		\noindent\textbf{Step 1:} $\E_{\epsilon} |\E U_2^2 -\E V_2^2| \xrightarrow[n \to \infty]{} 0$
\\
As Lemma~\ref{lem:esp_comp1}, the next lemma follows from the simple fact that for $N \sim \mathcal{N}(0,1)$ and $t \in \R$, $\E e^{tN} = \exp(\frac{t^2}{2})$.
		\begin{lemma} \label{lem:esp_comp2}
			Let $\bbf{x}^{(1)},\bbf{x}^{(2)} \in S^n$ be fixed. Then
			\begin{align*}
				\E_{\bbf{Z''},\bbf{\tilde{Z}}} &\exp\Big(
					\frac{\sqrt{\lambda}}{\sqrt{2}n} \sum_{i=1}^n Z_i'' (x^{(1)}_i)^2 
					+ \frac{\sqrt{\lambda}}{n} \sum_{1 \leq i<j \leq n} x^{(1)}_i x^{(1)}_j \tilde{Z}_{i,j} 
					+
					\frac{\sqrt{\lambda}}{\sqrt{2}n} \sum_{i=1}^n Z_i'' (x^{(2)}_i)^2 
					+ \frac{\sqrt{\lambda}}{n} \sum_{1 \leq i<j \leq n} x^{(2)}_i x^{(2)}_j \tilde{Z}_{i,j} 
				\Big)
				\\
				&= \exp\Big(\frac{\lambda}{2} (\bbf{x}^{(1)}.\bbf{x}^{(2)})^2 + \frac{\lambda}{4}(\|\bbf{x}^{(1)}\|^4+\|\bbf{x}^{(2)}\|^4)\Big)
			\end{align*}
		\end{lemma}
		Therefore
		\begin{align*}
			\E_{\bbf{Z''},\bbf{\tilde{Z}}} U_2^2 = \big\langle \E_{\bbf{Z''},\bbf{\tilde{Z}}} \ \exp(y(\bbf{x}^{(1)}) + y(\bbf{x}^{(2)})) \big\rangle 
			= \big\langle \exp\big( \frac{\lambda}{2} ((\bbf{x}^{(1)}.\bbf{x}^{(2)})^2 + (\bbf{x}^{(1)}.\bbf{X})^2 + (\bbf{x}^{(2)}.\bbf{X})^2 )\big) \big\rangle
		\end{align*}
		The function $F_2: (s,r_1,r_2) \in [-K_0^2,K_0^2]^3 \mapsto
		\exp \Big( \frac{\lambda}{2} (s^2 + r_1^2 + r_2^2) \Big)$ is $L_0'$-Lipschitz, for some $L_0'>0$. Remark that $\E_{\bbf{Z}'',\bbf{\tilde{Z}}} U_2^2 = \langle F_2 (\bbf{x}^{(1)}.\bbf{x}^{(2)}, \bbf{x}^{(1)}.\bbf{X},\bbf{x}^{(2)}.\bbf{X}) \rangle$ and $\E_{\bbf{Z''},\bbf{\tilde{Z}}} V_2^2 = F_2(Q,Q,Q)$ (see equation~\eqref{eq:esp_v2_bis}). Thus
		\begin{align*}
			\E_{\epsilon} |\E \E_{\bbf{Z''},\bbf{\tilde{Z}}} U_2^2 - \E \E_{\bbf{Z''},\bbf{\tilde{Z}}}V_2^2 | 
			&\leq
			\E_{\epsilon} \E \Big\langle | F_2 (\bbf{x}^{(1)}.\bbf{x}^{(2)}, \bbf{x}^{(1)}.\bbf{X},\bbf{x}^{(2)}.\bbf{X}) - F_2 (Q,Q,Q) | \Big\rangle
			\\
			&\leq 
			L_0' 
			\E_{\epsilon} \E \Big\langle \sqrt{(\bbf{x}^{(1)}.\bbf{x}^{(2)} - Q)^2 +  (\bbf{x}^{(1)}.\bbf{X} - Q)^2 + (\bbf{x}^{(2)}.\bbf{X} - Q)^2 } \Big\rangle
			\\
			& \xrightarrow[n \to \infty]{} 0
		\end{align*}
		because of Proposition~\ref{prop:overlap} and Corollary~\ref{cor:overlap}.
		\\

		\noindent\textbf{Step 2:} $\E_{\epsilon} |\E U_2 V_2 -\E V_2^2| \xrightarrow[n \to \infty]{} 0$
\\
		Using Lemma~\ref{lem:esp_comp2} with $\bbf{x}^{(1)}=\bbf{x}$ and $\bbf{x}^{(2)}=\bbf{b}$,
		\begin{align*}
			\E_{\bbf{Z''},\bbf{\tilde{Z}}} U_2 V_2 &= \Big\langle \E_{\bbf{Z''},\bbf{\tilde{Z}}} e^{y(\bbf{x})} \exp(\frac{\sqrt{\lambda}}{\sqrt{2}n} \sum_{i=1}^n b_i^2 Z_i'' + \frac{\sqrt{\lambda}}{n}\sum_{1 \leq i<j \leq n} b_i b_j \tilde{Z}_{i,j} 
			+ \frac{\lambda}{4}Q^2 )\Big\rangle \\
			&= \Big\langle 
			\exp\big(
				\frac{\lambda}{2}(\bbf{x}.\bbf{b})^2 + \frac{\lambda}{4}(\|\bbf{x}\|^4 + \|\bbf{b}\|^4) + \frac{\lambda}{2}((\bbf{x}.\bbf{X})^2 - \frac{1}{2}\|\bbf{x}\|^4) + \frac{\lambda}{4} Q^2
			\big)
		\Big\rangle \\
		&= \Big\langle 
		\exp\big(
			\frac{\lambda}{2}(\bbf{x}.\bbf{b})^2 +  \frac{\lambda}{2}(\bbf{x}.\bbf{X})^2+ \frac{\lambda}{2} Q^2
		\big)
	\Big\rangle
	= F_2(\bbf{x}.\bbf{b},\bbf{x}.\bbf{X},Q)
\end{align*}
Using same arguments as in ``Step 1'', we obtain finally $\E_{\epsilon} |\E \E_{\bbf{Z''},\bbf{\tilde{Z}}} U_2 V_2 - \E \E_{\bbf{Z''},\bbf{\tilde{Z}}}V_2^2| \xrightarrow[n \to \infty]{} 0$.
\end{proof}
\\

\noindent We have $| \log U_2 - \log V_2 | \leq \max(U_2^{-1},V_2^{-1}) |U_2-V_2|$.
Thus, by the Cauchy-Schwarz inequality
\begin{align*}
	\E |\log U_2 - \log V_2 | \leq \sqrt{(\E U_2^{-2} + \E V_2^{-2}) \E |U_2-V_2|^2}
	\leq C \sqrt{\E |U_2 - V_2|^2}
\end{align*}
because, similarly to Lemma~\ref{lem:u_v_moment}, we have for some constant $C>0$, $(\E U_2^{-2} + \E V_2^{-2}) \leq C^2$.  
Therefore $\E_{\epsilon} | \E \log U_2 - \E \log V_2 | \xrightarrow[n \to \infty]{} 0$.
Compute
$$
\E \log V_2 = \E \Big( \frac{\sqrt{\lambda}}{\sqrt{2}n} \sum_{i=1}^n b_i^2 Z_i'' + \frac{\sqrt{\lambda}}{n}\sum_{1 \leq i<j \leq n} b_i b_j \tilde{Z}_{i,j} 
+ \frac{\lambda}{4}Q^2 \Big)
= \frac{\lambda}{4} \E Q^2
$$
because the variables $(Z_i'')_i$ and $(\tilde{Z}_{i,j})_{i,j}$ are centered and independent of the $(b_i)_i$.
\\
Thus
$
\E_{\epsilon} | \E \log U_2 - \frac{\lambda}{4} \E  Q^2) | \xrightarrow[n \to  \infty]{} 0
$
and Lemma~\ref{lem:part2} is proved.

\section{Overlap concentration without perturbation, proof of Theorem~\ref{th:overlap_concentration_without_perturbation}} \label{sec:overlap_concentration_without_perturbation}

In this section, we prove Theorem~\ref{th:overlap_concentration_without_perturbation}. The proof is an adaptation of the proof of Ghirlanda-Guerra identities from~\cite{panchenko2013SK}, section 3.7.
In order to clarify the dependencies in $\lambda$, we will use in the following the notations $H_n(\bbf{x},\lambda)$ and $\langle \cdot \rangle_{\lambda}$ to denote the Gibbs measure corresponding to the Hamiltonian $H_n(\bbf{x},\lambda)$. 
Define
$$
L_n(\bbf{x},\lambda) =  2 \sqrt{\lambda} \frac{\partial H_n}{\partial \lambda}(\bbf{x},\lambda) = \sum_{i<j} \frac{1}{\sqrt{n}} Z_{i,j} x_i x_j + \frac{2 \sqrt{\lambda}}{n} x_i x_j X_i X_j - \frac{\sqrt{\lambda}}{n} x_i^2 x_j^2
$$
We are going to show first that $L_n$ concentrates around its expected value:
\begin{proposition}
	For all $\lambda \in D$,
	$$
	\frac{1}{n} \E \Big\langle | L_n(\bbf{x},\lambda) - \E \langle L_n(\bbf{x},\lambda) \rangle | \Big\rangle_{\lambda} \xrightarrow[n \to \infty]{} 0
	$$
\end{proposition}
We start with the following lemma.
\begin{lemma}
	For all $\lambda \in D$,
	$$
	\frac{1}{n} \E \Big\langle | L_n(\bbf{x},\lambda) - \langle L_n(\bbf{x},\lambda) \rangle | \Big\rangle_{\lambda} \xrightarrow[n \to \infty]{} 0
	$$
\end{lemma}

\begin{proof}
	Let us fix $\lambda_0 \in D$ and let $\lambda' > \lambda_0$. We are going to make use of the following lemma.

	\begin{lemma}
		Let $f$ be a function such that there exists a continuous function $g$ such that, for all $t \in \R$
		$$
		\limsup_{t' \to t, t' \neq t} | \frac{f(t)-f(t')}{t-t'} | \leq g(t)
		$$
		Then for all $t_1 \leq t_2$
		$$
		| f(t_1) - f(t_2) | \leq \int_{t_1}^{t_2} g(t) dt
		$$
	\end{lemma}
	Define $f: \lambda \mapsto \E \big\langle | L_n(\bbf{x}^{(1)},\lambda) - L_n(\bbf{x}^{(2)},\lambda) | \big\rangle_{\lambda} $. Let $\lambda>\lambda_0$ be fixed. Let $\lambda' >\lambda_0$ such that $\lambda' \neq \lambda$.
	\begin{align*}
		|f(\lambda) - f(\lambda')| \leq 
		&\Big| \E \Big\langle |L_n(\bbf{x}^{(1)}, \lambda') - L_n(\bbf{x}^{(2)},\lambda') | - |L_n(\bbf{x}^{(1)}, \lambda) - L_n(\bbf{x}^{(2)},\lambda) | \Big\rangle_{\lambda'} \Big|\\
		\\
		&+ \Big| \E \big\langle |L_n(\bbf{x}^{(1)}, \lambda) - L_n(\bbf{x}^{(2)},\lambda) | \big\rangle_{\lambda'} - \E \big\langle |L_n(\bbf{x}^{(1)}, \lambda) - L_n(\bbf{x}^{(2)},\lambda) | \big\rangle_{\lambda} \Big |
	\end{align*}
	We are going to upper bound successively these two terms. For the first one
	\begin{align*}
		\Big| \E \Big\langle |L_n(\bbf{x}^{(1)}, \lambda') - L_n(\bbf{x}^{(2)},\lambda') | - |L_n(\bbf{x}^{(1)}, \lambda) - L_n(\bbf{x}^{(2)},\lambda) | \Big\rangle_{\lambda'} \Big|
		& \leq 2 \E \Big\langle |L_n(\bbf{x}^{(1)}, \lambda') - L_n(\bbf{x}^{(1)},\lambda)|  \Big\rangle_{\lambda'} 
	\end{align*}
	and
	$$
	\E \Big\langle |L_n(\bbf{x}^{(1)}, \lambda') - L_n(\bbf{x}^{(1)},\lambda)|  \Big\rangle_{\lambda'} 
	=|\sqrt{\lambda} - \sqrt{\lambda'} | \E \Big\langle | \frac{1}{n} \sum_{i<j} 2 X_i X_j x_i x_j - x_i^2 x_j^2 |  \Big\rangle_{\lambda'} 
	\leq \frac{1}{\sqrt{\lambda_0}} K_0^4 n |\lambda - \lambda'|
	$$
	So that
	$$
	\limsup_{\lambda' \to \lambda, \lambda' \neq \lambda}
	\frac{\Big| \E \Big\langle |L_n(\bbf{x}^{(1)}, \lambda') - L_n(\bbf{x}^{(2)},\lambda') | - |L_n(\bbf{x}^{(1)}, \lambda) - L_n(\bbf{x}^{(2)},\lambda) | \Big\rangle_{\lambda'} \Big|}
	{|\lambda - \lambda'|}
	\leq \frac{1}{\sqrt{\lambda_0}} K_0^4 n
	$$
	For the second term, we remark that the function $\lambda' \mapsto \E \big\langle | L_n(\bbf{x}^{(1)},\lambda) - L_n(\bbf{x}^{(2)},\lambda) | \big\rangle_{\lambda'}$ is differentiable at $\lambda$ with derivative equal to
	\begin{align*}
		\frac{1}{2 \sqrt{\lambda}} \Big( \E \big\langle | L_n(\bbf{x}^{(1)},\lambda) - L_n(\bbf{x}^{(2)},\lambda) &| (L_n(\bbf{x}^{(1)},\lambda)+ L_n(\bbf{x}^{(2)},\lambda))\big\rangle_{\lambda}
			\\
			&-
			\E \big\langle | L_n(\bbf{x}^{(1)},\lambda) - L_n(\bbf{x}^{(2)},\lambda) | \big\rangle 
			\big\langle (L_n(\bbf{x}^{(1)},\lambda)+ L_n(\bbf{x}^{(2)},\lambda))\big\rangle_{\lambda}
		\Big)
		\\
		&=
		\lambda^{-1/2} \E \big\langle | L_n(\bbf{x}^{(1)},\lambda) - L_n(\bbf{x}^{(2)},\lambda) | (L_n(\bbf{x}^{(1)},\lambda)- L_n(\bbf{x}^{(3)},\lambda))\big\rangle_{\lambda}
		\\
		&\leq \lambda^{-1/2} \E \big\langle ( L_n(\bbf{x}^{(1)},\lambda) - L_n(\bbf{x}^{(2)},\lambda) )^2 \big\rangle_{\lambda} 
	\end{align*}
	Thus for all $\lambda > \lambda_0$
	$$
	\limsup_{\lambda' \to \lambda, \lambda' \neq \lambda} | \frac{f(\lambda') - f(\lambda)}{\lambda - \lambda'} | \leq \frac{1}{\sqrt{\lambda_0}} K_0^4 n +\frac{1}{\sqrt{\lambda_0}}  \E \big\langle ( L_n(\bbf{x}^{(1)},\lambda) - L_n(\bbf{x}^{(2)},\lambda) )^2 \big\rangle_{\lambda} 
	$$
	Let $\lambda' > \lambda_0$. Using the Lemma, we have, for all $\lambda' \geq \lambda \geq \lambda_0$
	$$
	f(\lambda_0) \leq f(\lambda) + \int_{\lambda_0}^{\lambda'} \frac{1}{\sqrt{\lambda_0}} K_0^4n +\frac{1}{\sqrt{\lambda_0}} \E \big\langle ( L_n(\bbf{x}^{(1)},l) - L_n(\bbf{x}^{(2)},l) )^2 \big\rangle_{l} dl
	$$
	We integrate with respect to $\lambda$ over $[\lambda_0, \lambda']$, and write $\delta = \lambda' - \lambda_0$
	\begin{align*}
		&\delta f(\lambda_0) 
		\leq \int_{\lambda_0}^{\lambda'} f(\lambda) d\lambda + \delta \int_{\lambda_0}^{\lambda'} \frac{1}{\sqrt{\lambda_0}} K_0^4 n +\frac{1}{\sqrt{\lambda_0}}  \E \big\langle ( L_n(\bbf{x}^{(1)},\lambda) - L_n(\bbf{x}^{(2)},\lambda) )^2 \big\rangle_{\lambda} d\lambda
		\\
		&\leq \int_{\lambda_0}^{\lambda'} f(\lambda) d\lambda + 2 \delta \int_{\lambda_0}^{\lambda'} \E \big\langle ( L_n(\bbf{x}^{(1)},\lambda) - L_n(\bbf{x}^{(2)},\lambda) )^2 \big\rangle_{\lambda} d\lambda
		+\frac{1}{\sqrt{\lambda_0}} \delta^2 K_0^4 n
		\\
		&\leq \Big( \delta \int_{\lambda_0}^{\lambda'} \E \big\langle ( L_n(\bbf{x}^{(1)},\lambda) - L_n(\bbf{x}^{(2)},\lambda) )^2 \big\rangle_{\lambda} d\lambda \Big)^{1/2} 
		+ \frac{\delta}{\sqrt{\lambda_0}} \int_{\lambda_0}^{\lambda'} \E \big\langle ( L_n(\bbf{x}^{(1)},\lambda) - L_n(\bbf{x}^{(2)},\lambda) )^2 \big\rangle_{\lambda} d\lambda
		+\frac{1}{\sqrt{\lambda_0}} \delta^2 K_0^4 n
		\\
		&\leq \Big( 2 \delta \int_{\lambda_0}^{\lambda'} \E \big\langle ( L_n(\bbf{x},\lambda) - \langle L_n(\bbf{x},\lambda) \rangle_{\lambda} )^2 \big\rangle_{\lambda} d\lambda \Big)^{1/2} 
		+ 2 \frac{\delta}{\sqrt{\lambda_0}} \int_{\lambda_0}^{\lambda'} \E \big\langle ( L_n(\bbf{x},\lambda) - \langle L_n(\bbf{x},\lambda)\rangle_{\lambda} )^2 \big\rangle_{\lambda} d\lambda
		+\frac{1}{\sqrt{\lambda_0}} \delta^2 K_0^4 n
	\end{align*}
	Define $\epsilon_n = \frac{1}{n} \int_{\lambda_0}^{\lambda'} \E \big\langle ( L_n(\bbf{x},\lambda) - \langle L_n(\bbf{x},\lambda) \rangle_{\lambda} )^2 \big\rangle_{\lambda} d\lambda$. We have shown
	\begin{equation} \label{eq:ineq_epsilon}
		\frac{1}{n} \E \langle | L_n(\bbf{x}^{(1)},\lambda_0) - L_n(\bbf{x}^{(2)},\lambda_0) | \rangle_{\lambda_0} \leq \sqrt{\frac{2 \epsilon_n}{n \delta}} + \frac{2 \epsilon_n}{\sqrt{\lambda_0}} +  \frac{\delta K_0^4}{\sqrt{\lambda_0}} 
	\end{equation}
	It will be slightly more convenient to ``replace'' $\lambda$ by $\lambda^2$, we define therefore the function
	$$
	\phi_n: \lambda \mapsto \frac{1}{n}\log \int dP_0^{\otimes n}(\bbf{x}) e^{H_n(\bbf{x},\lambda^2)} + K_0^4 \lambda^2
	$$
	The derivatives of $\phi_n$ are
	\begin{align*}
		\phi_n'(\lambda) &= \frac{1}{n} \langle L_n(\bbf{x},\lambda^2) \rangle_{\lambda^2} + 2 K_0^4 \lambda \\
		\phi_n''(\lambda) &= \frac{1}{n} \big( \langle L_n(\bbf{x},\lambda^2)^2 \rangle_{\lambda^2} - \langle L_n(\bbf{x},\lambda^2) \rangle_{\lambda^2}^2 \big) + \frac{1}{n^2} \sum_{i<j} \langle 2 x_i x_j X_i X_j - x_i^2 x_j^2 \rangle + 2 K_0^4
	\end{align*}
	Define $G_n : \lambda \mapsto \E \phi_n(\lambda)$, then
	$$\forall \lambda > 0, \  G_n''(\lambda) \geq \frac{1}{n} \E \big( \langle L_n(\bbf{x},\lambda^2)^2 \rangle_{\lambda^2} - \langle L_n(\bbf{x},\lambda^2) \rangle_{\lambda^2}^2 \big)$$
	Therefore
	\begin{align*}
		G_n'(\sqrt{\lambda'}) - G_n'(\sqrt{\lambda_0}) = \int_{\sqrt{\lambda_0}}^{\sqrt{\lambda'}} G_n''(\lambda) d\lambda 
		&\geq 
		\frac{1}{n} \int_{\sqrt{\lambda_0}}^{\sqrt{\lambda'}} \E \big( \langle L_n(\bbf{x},\lambda^2)^2 \rangle_{\lambda^2} - \langle L_n(\bbf{x},\lambda^2) \rangle_{\lambda^2}^2 \big) d\lambda
		\\
		&=
		\frac{1}{n} \int_{\lambda_0}^{\lambda'} \frac{1}{2 \sqrt{\lambda}} \E \big( \langle L_n(\bbf{x},\lambda)^2 \rangle_{\lambda} - \langle L_n(\bbf{x},\lambda) \rangle_{\lambda}^2 \big) d\lambda
		\\
		&\geq \frac{1}{2 \sqrt{\lambda'}} \epsilon_n
	\end{align*}
	By convexity of $G_n$, we have, for all $y>0$
	$$
	G_n'(\sqrt{\lambda'}) - G_n'(\sqrt{\lambda_0}) \leq \frac{G_n(\sqrt{\lambda'} + y) - G_n(\sqrt{\lambda'})}{y} - \frac{G_n(\sqrt{\lambda_0}) - G_n(\sqrt{\lambda_0} - y)}{y}
	$$
	$G_n$ converge to $\psi: \lambda \mapsto \phi(\lambda^2) + K_0^4 \lambda^2$ 
	$$
	\limsup_{n \to \infty} \epsilon_n \leq 
	2 \sqrt{\lambda'} (\frac{\psi(\sqrt{\lambda'} + y) - \psi(\sqrt{\lambda'})}{y} - \frac{\psi(\sqrt{\lambda_0}) - \psi(\sqrt{\lambda_0} - y)}{y}
	)$$
	Thus, taking the limsup in $n$ in equation~\eqref{eq:ineq_epsilon}:
	$$
	\limsup_{n \to \infty} \frac{1}{n} \E \langle | L_n(\bbf{x}^{(1)},\lambda_0) - L_n(\bbf{x}^{(2)},\lambda_0) | \rangle_{\lambda_0} 
	\leq 4 \frac{\sqrt{\lambda'}}{\sqrt{\lambda_0}} (\frac{\psi(\sqrt{\lambda'} + y) - \psi(\sqrt{\lambda'})}{y} - \frac{\psi(\sqrt{\lambda_0}) - \psi(\sqrt{\lambda_0} - y)}{y})+ \frac{\delta K_0^4}{\sqrt{\lambda_0}}
	$$
	We let then $\lambda' \to \lambda_0$ ($\delta \to 0$).
	$$
	\limsup_{n \to \infty} \frac{1}{n} \E \langle | L_n(\bbf{x}^{(1)},\lambda_0) - L_n(\bbf{x}^{(2)},\lambda_0) | \rangle_{\lambda_0} 
	\leq 4 (\frac{\psi(\sqrt{\lambda_0} + y) - \psi(\sqrt{\lambda_0})}{y} - \frac{\psi(\sqrt{\lambda_0}) - \psi(\sqrt{\lambda_0} - y)}{y})
	$$
	By hypothesis, $\psi$ is differentiable in $\sqrt{\lambda_0}$, hence, by letting $y \to 0$, 
	$\frac{1}{n} \E \langle | L_n(\bbf{x}^{(1)},\lambda_0) - L_n(\bbf{x}^{(2)},\lambda_0) | \rangle_{\lambda_0} \xrightarrow[n \to \infty]{} 0$, which means
	$$
	\frac{1}{n} \E \langle | L_n(\bbf{x},\lambda_0) - \langle L_n(\bbf{x},\lambda_0) \rangle_{\lambda_0} | \rangle_{\lambda_0} \xrightarrow[n \to \infty]{} 0
	$$
\end{proof}
\\

\begin{lemma}
	For all $\lambda \in D$,
	$$
	\frac{1}{n} \E | \langle L_n(\bbf{x},\lambda) \rangle_{\lambda} - \E \langle L_n(\bbf{x},\lambda) \rangle_{\lambda} | \xrightarrow[n \to \infty]{} 0
	$$
\end{lemma}

\begin{proof}
	Let $\lambda_0 \in D$.
	We are going to reuse the functions $\phi_n$ and $G_n$ defined in the proof of the previous Lemma. Notice that $\phi_n$ and $G_n$ are both convex, differentiable functions. We have
	\begin{equation} \label{eq:control_by_deriv}
		\frac{1}{n} \E | \langle L_n(\bbf{x},\lambda_0) \rangle - \E \langle L_n(\bbf{x},\lambda_0) \rangle | 
		= \E | \phi_n'(\sqrt{\lambda_0}) - G_n'(\sqrt{\lambda_0}) |
	\end{equation}
	Let $y >0$ and define $\delta_n(y) = | \phi_n(\sqrt{\lambda_0} - y) - G_n(\sqrt{\lambda_0}-y)| +| \phi_n(\sqrt{\lambda_0} ) - G_n(\sqrt{\lambda_0})| +| \phi_n(\sqrt{\lambda_0} + y) - G_n(\sqrt{\lambda_0}+y)|$. By convexity of $\phi_n$
	\begin{align}
		\phi_n'(\sqrt{\lambda_0}) - G_n'(\sqrt{\lambda_0}) 
		&\leq \frac{\phi_n(\sqrt{\lambda_0} + y) - \phi_n(\sqrt{\lambda_0})}{y} - G_n'(\sqrt{\lambda_0}) \nonumber
		\\
		&\leq \Big| \frac{G_n(\sqrt{\lambda_0} + y) - G_n(\sqrt{\lambda_0})}{y} - G_n'(\sqrt{\lambda_0}) \Big| + \frac{\delta_n(y)}{y} \label{eq:ineq_convex1}
	\end{align}
	Analogously,
	\begin{equation} \label{eq:ineq_convex2}
		\phi_n'(\sqrt{\lambda_0}) - G_n'(\sqrt{\lambda_0}) 
		\geq \Big| \frac{G_n(\sqrt{\lambda_0}) - G_n(\sqrt{\lambda_0} - y)}{y} - G_n'(\sqrt{\lambda_0}) \Big| - \frac{\delta_n(y)}{y}
	\end{equation}
	Thus, combining~\eqref{eq:ineq_convex1} and~\eqref{eq:ineq_convex2},
	\begin{align}
		\E | \phi_n'(\sqrt{\lambda_0}) - G_n'(\sqrt{\lambda_0}) |
		\leq&
		\big| \frac{G_n(\sqrt{\lambda_0} + y) - G_n(\sqrt{\lambda_0})}{y} - G_n'(\sqrt{\lambda_0}) \big|
		\nonumber \\
		+& \big| \frac{G_n(\sqrt{\lambda_0}) - G_n(\sqrt{\lambda_0} -y)}{y} - G_n'(\sqrt{\lambda_0}) \big| + \frac{\E \delta_n(y)}{y} \label{eq:control_convex}
	\end{align}
	We are going to show that for $y \in [- \sqrt{\lambda_0} / 2, \sqrt{\lambda_0} / 2]$, $\delta_n(y) \xrightarrow[n \to \infty]{} 0$. Define 
	$$
	v_n = \sup_{\lambda \in [\sqrt{\lambda_0}/2, 2 \sqrt{\lambda_0}]} \E | \phi_n(\lambda) - \E \phi_n(\lambda)| = \sup_{\lambda \in [\lambda_0/4, 4 \lambda_0]} \E | \frac{1}{n} \log \int dP_0^{\otimes n} e^{H_n(\bbf{x},\lambda)} - F_n(\lambda)|
	$$
	\begin{lemma}
		$$
		v_n= O(n^{-1/2})
		$$
	\end{lemma}

	\begin{proof}
		Let $\lambda \in [\lambda_0/4, 4 \lambda_0]$. Let us decompose $H_n(\bbf{x},\lambda) = h_1(\bbf{x}) + h_2(\bbf{x})$ where
		\begin{align*}
			h_1(\bbf{x}) &= \sqrt{\frac{\lambda}{n}} \sum_{i<j} Z_{i,j} x_i x_j \\
			h_2(\bbf{x}) &= \sum_{i<j} \frac{\lambda}{n} x_i x_j X_i X_j - \frac{\lambda}{2n} x_i^2 x_j^2
		\end{align*}
		We are going to use the following theorem from~\cite{talagrand2010meanfield1} (Theorem 1.3.4).

		\begin{theorem}[\cite{talagrand2010meanfield1}, Theorem 1.3.4] \label{th:concentration_lip}
			Consider a Lipschitz function $F$ on $R^M$ with Lipschitz constant $\leq A$. Let $g_1, \dots, g_M$ be independent standard random variables and $\bbf{g} = (g_1, \dots, g_M)$.
			Then we have for each $t > 0$
			$$
			\PP\left(|F(\bbf{g}) - \E F(\bbf{g})| < t\right) \leq 2 \exp\left(-\frac{t^2}{4A^2}\right)
			$$
			and consequently
			$$
			\E (F(\bbf{g}) - \E F(\bbf{g}))^2 \leq 8 A^2
			$$
		\end{theorem}
		Consider $\bbf{X}$ to be fixed. We apply then this Theorem with 
		$$
		F: (z_{i,j})_{1\leq i <j\leq n} \mapsto \log \int dP_0^{\otimes n}(\bbf{x}) \exp\left( \sum_{i<j} \sqrt{\frac{\lambda}{n}}z_{i,j} x_i x_j  + \frac{\lambda}{n} x_i x_j X_i X_j - \frac{\lambda}{2n} x_i^2 x_j^2 \right)
		$$
		$F$ is $\sqrt{\lambda n} K_0^2$-Lipschitz because for all $i<j$
		$$
		\left| \frac{\partial F}{\partial z_{i,j}}(\bbf{z}) \right| = \sqrt{\frac{\lambda}{n}} \left| \langle x_i x_j \rangle \right|
		\leq \sqrt{\frac{\lambda}{n}} K_0^2
		$$
		If we denote by $\E_{\bbf{Z}}$ the expectation with respect to $\bbf{Z}$ only, Theorem~\ref{th:concentration_lip} gives that for all values of $\bbf{X}$
		$$
		\E_{\bbf{Z}} \left( \log \int dP_0^{\otimes n}(\bbf{x}) e^{H_n(\bbf{x},\lambda)} - \E_{\bbf{Z}} \log \int dP_0^{\otimes n}(\bbf{x}) e^{H_n(\bbf{x},\lambda)} \right)^2 \leq 8 K_0^4 \lambda n
		$$
		and thus $\E \Big( \log \int dP_0^{\otimes n}(\bbf{x}) e^{H_n(\bbf{x},\lambda)} - \E_{\bbf{Z}} \log \int dP_0^{\otimes n}(\bbf{x}) e^{H_n(\bbf{x},\lambda)} \Big)^2 \leq 8 K_0^4 \lambda n$.
		\\

		We are now going to show that $\E_{\bbf{Z}} \log \int dP_0^{\otimes n}(\bbf{x}) e^{H_n(\bbf{x},\lambda)}$ concentrates around its expectation (with respect to $\bbf{X}$). $\E_{\bbf{Z}} \log \int dP_0^{\otimes n}(\bbf{x}) e^{H_n(\bbf{x},\lambda)}$ is a function of $\bbf{X}$. We can easily verify that this function has ``the bounded differences property'' (see~\cite{boucheron2013concentration}, section 3.2) because $\bbf{X}$ has bounded support. Then Corollary 3.2 from~\cite{boucheron2013concentration} (which is a consequence of the Efron-Stein inequality) gives
		$$
		\E \left( \E_{\bbf{Z}} \log \int dP_0^{\otimes n}(\bbf{x}) e^{H_n(\bbf{x},\lambda)} - \E \log \int dP_0^{\otimes n}(\bbf{x}) e^{H_n(\bbf{x},\lambda)}\right)^2 \leq C n
		$$
		where $C$ is a constant that does not depend on $\lambda \in [\lambda_0 / 4, 4 \lambda_0]$. We conclude that $v_n = O(n^{-1/2})$.
	\end{proof}
	\\

	We suppose now that $y \in [- \sqrt{\lambda_0} / 2, \sqrt{\lambda_0}/2]$. Then $\delta_n(y) \xrightarrow[n \to +\infty]{} 0$. Recall that $G_n$ converges to $\psi: \lambda \mapsto \phi(\lambda^2) + 2 K_0^4 \lambda^2$ which is also convex and differentiable in $\sqrt{\lambda_0}$ (remember that $\lambda_0 \in D$). Then, Lemma~\ref{lem:deriv_convex} gives that $G_n'(\sqrt{\lambda_0}) \to \psi'(\sqrt{\lambda_0})$. Letting $n \to \infty$ in~\eqref{eq:control_convex}
	\begin{align*}
		\limsup_{n \to \infty} \E | \phi_n'(\sqrt{\lambda_0}) - G_n'(\sqrt{\lambda_0}) |
		\leq&
		\big| \frac{\psi(\sqrt{\lambda_0} + y) - \psi(\sqrt{\lambda_0})}{y} - \psi'(\sqrt{\lambda_0}) \big|
		\\
		+& \big| \frac{\psi(\sqrt{\lambda_0}) - \psi(\sqrt{\lambda_0} -y)}{y} - \psi'(\sqrt{\lambda_0}) \big| 
	\end{align*}
	We let then $y \to 0$: by differentiability of $\psi$ in $\sqrt{\lambda_0}$, the right-hand side of the previous inequality goes to $0$. We conclude using~\eqref{eq:control_by_deriv}.
\end{proof}
\\

We are now able to prove Theorem~\ref{th:overlap_concentration_without_perturbation}.
Let $\lambda \in D$.
$$
\frac{1}{n} \Big| \E \langle L_n(\bbf{x}^{(1)}) (\bbf{x}^{(1)}.\bbf{x}^{(2)})^2 \rangle - \E\langle (\bbf{x}^{(1)}.\bbf{x}^{(2)})^2 \rangle \E \langle L_n(\bbf{x}^{(1)}) \rangle \Big| \leq \frac{1}{n} K_0^4 \E \Big\langle \big| L_n(\bbf{x}) - \E \langle L_n(\bbf{x}) \rangle \big| \Big\rangle \xrightarrow[n \to \infty]{} 0
$$
Compute
$$
\frac{1}{n} \E \langle L_n(\bbf{x}) \rangle = \sqrt{\lambda} \frac{1}{n^2} \E \langle \sum_{i<j} x_i^{(1)} x_i^{(2)} x_j^{(1)} x_j^{(2)} \rangle = \frac{\sqrt{\lambda}}{2} \E \langle (\bbf{x}^{(1)}.\bbf{x}^{(2)})^2 \rangle + O(\frac{1}{n})
$$
Therefore $\E\langle (\bbf{x}^{(1)}.\bbf{x}^{(2)})^2 \rangle \E \langle L(\bbf{x}^{(1)}) \rangle = \frac{\sqrt{\lambda}}{2} (\E \langle (\bbf{x}^{(1)}.\bbf{x}^{(2)})^2 \rangle)^2$. Moreover, using Gaussian integration by parts and the Nishimori property,
\begin{align*}
	\E  \langle L_n(\bbf{x}^{(1)})(\bbf{x}^{(1)}.\bbf{x}^{(2)})^2 \rangle 
	&= \sqrt{\lambda}\E \langle (\bbf{x}^{(1)}.\bbf{X})^2 (\bbf{x}^{(1)}.\bbf{x}^{(2)})^2 \rangle + \frac{1}{n^{3/2}} \sum_{i<j} \E Z_{i,j} \langle x^{(1)}_i x^{(1)}_j (\bbf{x}^{(1)}.\bbf{x}^{(2)})^2 \rangle 
	\\
	&- \frac{\sqrt{\lambda}}{n^2} \sum_{i<j} \E \langle (x^{(1)}_i x^{(1)}_j)^2 (\bbf{x}^{(1)}.\bbf{x}^{(2)})^2 \rangle +O(\frac{1}{n}) \\
	&= O(\frac{1}{n}) + \sqrt{\lambda} \Big( \E \big\langle (\bbf{x}^{(1)}.\bbf{X})^2 (\bbf{x}^{(1)}.\bbf{x}^{(2)})^2 \big\rangle 
	\\
	&+ \sum_{i<j} \E \Big[ \big\langle (x^{(1)}_i x^{(1)}_j x^{(2)}_i x^{(2)}_j)(\bbf{x}^{(1)}.\bbf{x}^{(2)})^2 \big\rangle - \big\langle (x^{(1)}_i x^{(1)}_j x^{(3)}_i x^{(3)}_j + x^{(1)}_i x^{(1)}_j x^{(4)}_i x^{(4)}_j)(\bbf{x}^{(1)}.\bbf{x}^{(2)})^2 \big\rangle  \Big]
\Big)
\\
&= \frac{\sqrt{\lambda}}{2} \E \big\langle (\bbf{x}^{(1)} \bbf{x}^{(2)})^2 (\bbf{x}^{(1)}.\bbf{x}^{(2)})^2 \big\rangle + O(\frac{1}{n}) 
\\
&= \frac{\sqrt{\lambda}}{2} \E \langle (\bbf{x}^{(1)}.\bbf{x}^{(2)})^4 \rangle + O(\frac{1}{n})
\end{align*}
Hence $\E \Big\langle \big( (\bbf{x}^{(1)}.\bbf{x}^{(2)})^2 - \E \langle (\bbf{x}^{(1)}.\bbf{x}^{(2)})^2 \rangle \big)^2 \Big\rangle \xrightarrow[n \to \infty]{} 0$. Corollary~\ref{cor:limit_mmse} leads then to the theorem statement.

\section{Extension to general multidimensional input distributions}\label{sec:gen}

In this section, we generalize the results of Section~\ref{sec:rs_formula} to any probability distributions $P_0$ over $\R^k$ ($k \in \N^*$ is fixed) that has a finite second moment: $\E_{P_0} \|\bbf{X}\|^2 < \infty$.

\subsection{Main results}

Let $P_0$ be a probability distribution over $\R^k$ that admits a finite second moment. Consider the following Gaussian observation channel
$$
Y_{i,j} = \sqrt{\frac{\lambda}{n}} \bbf{X}_i^{\intercal} \bbf{X}_j + Z_{i,j} \ \text{for } 1 \leq i<j \leq n
$$
where $\bbf{X}_i \overset{\text{\tiny i.i.d.}}{\sim}  P_0$ and $Z_{i,j} \overset{\text{\tiny i.i.d.}}{\sim}  \mathcal{N}(0,1)$. We write $P_0^{\otimes n}= P_0 \otimes P_0 \otimes \dots \otimes P_0$ ($n$ times).
\\

We will prove in the beginning of Section~\ref{sec:proof_multidim} that the mutual information for this Gaussian channel is
$$
I(\bbf{X},\bbf{Y}) = - \E \Big[ \log \int d P_0^{\otimes n}(\bbf{x}) \exp( \sum_{i<j} \bbf{x}_i^{\intercal} \bbf{x}_j \sqrt{\frac{\lambda}{n}} Z_{i,j} - \frac{\lambda}{2n}(\bbf{x}_i^{\intercal} \bbf{x}_j - \bbf{X}_i^{\intercal} \bbf{X}_j)^2 ) \Big]
$$
We are interested in computing the limit of $\frac{1}{n} I(\bbf{X},\bbf{Y})$. To do so, it will be more convenient to consider the free energy.  Let us define the random Hamiltonian $H_n(\bbf{x}) = \sum_{i<j} \bbf{x}_i^{\intercal} \bbf{x}_j \sqrt{\frac{\lambda}{n}} (Z_{i,j} + \sqrt{\frac{\lambda}{n}} \bbf{X}_i^{\intercal} \bbf{X}_j ) - \frac{\lambda}{2n} (\bbf{x}_i^{\intercal} \bbf{x}_j)^2$. 
We define the partition function as
$$
Z_n =  \int d P_0^{\otimes n}(\bbf{x}) \ e^{H_n(\bbf{x})}
$$
and the free energy as
$$
F_n = \frac{1}{n} \E \log Z_n
$$
We will express the limit of $F_n$ using the following function
$$
\mathcal{F}: (\lambda,\bbf{q}) \in \R \times S_k^+ \mapsto -\frac{\lambda}{4}\|\bbf{q}\|^2 + \E \log \left( \int dP_0(\bbf{x}) \exp(\sqrt{\lambda} (\bbf{Z}^{\intercal} \bbf{q}^{1/2} \bbf{x})+ \lambda \bbf{x}^{\intercal} \bbf{q} \bbf{X} - \frac{\lambda}{2}\bbf{x}^{\intercal} \bbf{q} \bbf{x})\right)
$$
where $S_k^+$ denote the set of $k \times k$ symmetric positive-semidefinite matrices and $\bbf{Z} \sim \mathcal{N}(0,\bbf{I}_k)$ and $\bbf{X} \sim P_0$ are independent random variables.

\begin{theorem}[Replica-Symmetric formula, general case] \label{th:rs_formula_multidim}
	\begin{align*}
		\lim_{n \rightarrow + \infty} F_n &= \sup_{\bbf{q} \in S_k^+} \mathcal{F}(\lambda,\bbf{q})
	\end{align*}
\end{theorem}

The Replica-Symmetric formula allows us to compute the limit of the mutual information.

\begin{corollary} \label{cor:rs_information_multidim}
	$$
	\lim_{n \rightarrow + \infty} \frac{1}{n}  I(\bbf{X},\bbf{Y}) = \frac{\lambda \|\E_{P_0} (\bbf{X} \bbf{X}^{\intercal})\|^2}{4} - \sup_{\bbf{q} \in S_k^+} \mathcal{F}(\lambda,\bbf{q})
	$$
\end{corollary}

\begin{proof}
	\begin{align*}
		\frac{1}{n} I(\bbf{X},\bbf{Y}) &= -F_n - \frac{1}{n} \E \log \exp(-\frac{\lambda}{2n} \sum_{i<j} (\bbf{X}_i^{\intercal} \bbf{X}_j)^2) 
		= \frac{\lambda}{2n^2} \sum_{i<j} \E (\bbf{X}_i^{\intercal} \bbf{X}_j)^2 - F_n 
		= \frac{\lambda (n-1)}{4n} \|\E_{P_0}(\bbf{X} \bbf{X}^{\intercal})\|^2 - F_n
	\end{align*}
	and Theorem~\ref{th:rs_formula_multidim} gives the result.
\end{proof}
\\

We define $\phi: \lambda \mapsto \sup_{\bbf{q} \in S_k^+} \mathcal{F}(\lambda,\bbf{q})$. $\phi$ is the limit of $\lambda \mapsto F_n(\lambda)$, which is convex. $\phi$ is therefore convex and is thus derivable everywhere except on a countable set of points. Let $D \subset (0, + \infty)$ be the set of points where $\phi$ is derivable.

\begin{proposition} \label{prop:derivative_phi_general}
	For all $\lambda \in D$, all the maximizers $\bbf{q}$ of $\bbf{q} \in S_k^+ \mapsto \Fcal(\lambda,\bbf{q})$ have the same norm $\| \bbf{q} \|^2 = q^*(\lambda)^2$ and
	$$
	\phi'(\lambda) = \frac{q^*(\lambda)^2}{4}
	$$
\end{proposition}

The proof is the same than for Proposition~\ref{prop:derivative_phi}. 
Analogously to the unidimensional case, we define
\begin{align*}
	\MMSE_n(\lambda) 
	&= \min_{\hat{\theta}} \frac{2}{n(n-1)} \sum_{1\leq i<j\leq n} \E\left[ \left(\bbf{X}_i^{\intercal} \bbf{X}_j- \hat{\theta}_{i,j}(\bbf{Y}) \right)^2\right] \label{eq:def_mmse_min} \\
	&=\frac{2}{n(n-1)} \sum_{1\leq i<j\leq n} \E\left[ \left(\bbf{X}_i^{\intercal} \bbf{X}_j-\E\left[\bbf{X}_i^{\intercal} \bbf{X}_j|\bbf{Y} \right]\right)^2\right],
\end{align*}
where the minimum is taken over all estimators $\hat{\theta}$ (i.e.\ measurable functions of the observations $\bbf{Y}$ that could depend on auxiliary randomness).
The following result is the analog of Corollary~\ref{cor:limit_mmse} and is proved with the same arguments.

\begin{corollary} \label{cor:limit_mmse_general}
	For all $\lambda \in D$,
	$$
	\MMSE_n(\lambda) \xrightarrow[n \to \infty]{} \|\E_{P_0} \bbf{X} \bbf{X}^{\intercal} \|^2 - q^*(\lambda)^2
	$$
\end{corollary}

\subsection{Proofs} \label{sec:proof_multidim}

In this section we prove Theorem~\ref{th:rs_formula_multidim}: we will show how the proofs of Section~\ref{sec:proof_rs_formula} generalize to the general multidimensional case. 

\subsubsection{Mutual information}

We start by the general expression of the mutual information.
\begin{lemma} \label{lem:mutual_information_general}
	$$
	I(\bbf{X},\bbf{Y}) = - \E \Big[ \log \int d P_0^{\otimes n}(\bbf{x}) \exp( \sum_{i<j} \bbf{x}_i^{\intercal} \bbf{x}_j \sqrt{\frac{\lambda}{n}} Z_{i,j} - \frac{\lambda}{2n}(\bbf{x}_i^{\intercal} \bbf{x}_j - \bbf{X}_i^{\intercal} \bbf{X}_j)^2 )\Big]
	$$
\end{lemma}
\begin{proof}
	Let $\mu$ denote the Lebesgue measure on $\R^{n(n-1)/2}$.
	The mutual information between $\bbf{X}$ and $\bbf{Y}$ is defined as the Kullback-Leibler divergence between $P_{(\bbf{X},\bbf{Y})}$, the joint distribution of $(\bbf{X},\bbf{Y})$, and $P_0^{\otimes n} \otimes P_{\bbf{Y}}$, the product of the marginal distributions of $\bbf{X}$ and $\bbf{Y}$. This Kullback-Leibler divergence is well defined because $P_{(\bbf{X},\bbf{Y})}$ is absolutely continuous with respect to $P_{0}^{\otimes n} \otimes P_{\bbf{Y}}$. Indeed for any Borel set $A$ of $\R^{kn} \times \R^{n(n-1)/2}$:
	$$
	P_{(\bbf{X},\bbf{Y})}(A) = \frac{1}{(2\pi)^{n(n-1)/4}} \int \int_{\bbf{y}} 1((\bbf{x},\bbf{y}) \in A) \exp\big(- \frac{1}{2} \sum_{i<j} (y_{i,j} - \sqrt{\frac{\lambda}{n}} \bbf{x}_i^{\intercal} \bbf{x}_j)^2 \big) d\mu(\bbf{y}) d P_0^{\otimes n}(\bbf{x})
	$$
	If $A$ is a Borel set of $\R^{n(n-1)/2}$, then
	$$
	P_{\bbf{Y}}(A) = \frac{1}{(2\pi)^{n(n-1)/4}}  \int_{\bbf{y}} 1(\bbf{y} \in A) \Big( \int \exp\big(- \frac{1}{2} \sum_{i<j} (y_{i,j} - \sqrt{\frac{\lambda}{n}} \bbf{x}_i^{\intercal} \bbf{x}_j)^2 \big) d P_0^{\otimes n}(\bbf{x}) \Big) d\mu(\bbf{y}) 
	$$
	so that
	$$
	\frac{d P_{(\bbf{X},\bbf{Y})}}{d P_0^{\otimes n} \otimes P_{\bbf{Y}}} = 
	\frac{\exp\big(- \frac{1}{2} \sum_{i<j} (Y_{i,j} - \sqrt{\frac{\lambda}{n}} \bbf{X}_i^{\intercal} \bbf{X}_j)^2 \big)}
	{\int \exp\big(- \frac{1}{2} \sum_{i<j} (Y_{i,j} - \sqrt{\frac{\lambda}{n}} \bbf{x}_i^{\intercal} \bbf{x}_j)^2 \big) d P_0^{\otimes n}(\bbf{x})
	}
	$$
	We can thus compute the mutual information
	\begin{align*}
		I(\bbf{X},\bbf{Y}) 
		&= \E \log \Big(
		\frac{\exp\big(- \frac{1}{2} \sum_{i<j} (Y_{i,j} - \sqrt{\frac{\lambda}{n}} \bbf{X}_i^{\intercal} \bbf{X}_j)^2 \big)}
		{\int \exp\big(- \frac{1}{2} \sum_{i<j} (Y_{i,j} - \sqrt{\frac{\lambda}{n}} \bbf{x}_i^{\intercal} \bbf{x}_j)^2 \big) d P_0^{\otimes n}(\bbf{x})
		}
	\Big)
	\\
	&= - \E \log \Big(  \int d P_0^{\otimes n}(\bbf{x}) \exp( \sum_{i<j} \bbf{x}_i^{\intercal} \bbf{x}_j \sqrt{\frac{\lambda}{n}} Z_{i,j} - \frac{\lambda}{2n}(\bbf{x}_i^{\intercal} \bbf{x}_j - \bbf{X}_i^{\intercal} \bbf{X}_j)^2 )\Big)
\end{align*}
\end{proof}

\subsubsection{Reduction to finite distribution}
We will show in this section that it suffices to prove Theorem~\ref{th:rs_formula_multidim} for input distribution $P_0$ with finite support.
\\

Suppose the Theorem~\ref{th:rs_formula_multidim} holds for distribution over $\R^k$ with finite support. Let $P_0$ be a probability distribution that admits a finite second moment: $\E_{P_0}\|\bbf{X}\|^2 < \infty$. We are going to approach $P_0$ with distributions with finite supports.
\\
Let $0 < \epsilon \leq 1$.
Let $K_0 > 0$ such that for any marginal $\mu$ of $P_0$, $\mu([-K_0,K_0]) > 1 - \epsilon^2$.
Let $m \in \N^*$ such that $\frac{K_0}{m} \leq \epsilon$. For $x \in [-K_0,K_0]$ we will use the notation
$$
\bar{x} = 
\begin{cases}
	\frac{K_0}{m} \Big\lfloor \frac{x m}{K_0} \Big\rfloor & \ \text{if} \ x \in [-K_0,K_0] \\
	0 & \ \text{otherwise}
\end{cases}
$$
Consequently if $x \in [-K_0,K_0]$, $\bar{x} \leq x < \bar{x} + \frac{K_0}{m} \leq \bar{x} + \epsilon$. For $\bx \in [-K_0,K_0]^k$ we also define $\bbf{\bar{x}}= (\bar{x}_1, \dots, \bar{x}_k) \in \bar{S}$ where $\bar{S} = \{i \frac{K_0}{m} \ | \ i=-m, \dots, m\}^k$. Finally, we define $\bar{P}_0$ the image distribution of $P_0$ through the application $\bx \mapsto \bbf{\bar{x}}$. Let $n \geq 1$. We will note $\bar{F}_n$ the free energy corresponding to the distribution $\bar{P}_0$ and $\bar{\mathcal{F}}$ the function $\mathcal{F}$ corresponding to the distribution $\bar{P}_0$.
$\bar{P}_0$ has a finite support, we have then by assumptions
\begin{equation} \label{eq:limit_discrete}
\bar{F}_n \xrightarrow[n \to \infty]{} \sup_{\bbf{q} \in S_k^+} \bar{\mathcal{F}}(\lambda,\bbf{q})
\end{equation}

\begin{lemma} \label{lem:approx_f_discrete}
	There exists a constant $K > 0$, that depends only on $P_0$, such that, for all $n \geq 1$
	$$
	| F_n - \bar{F}_n | \leq \lambda K \epsilon
	$$
\end{lemma}

\begin{proof}
	We define, for $0 \leq t \leq 1$ and $\bbf{x}\in \R^k$
	\begin{align*}
		H_{n,t}(\bbf{x}) = 
	&\sum_{1 \leq i<j\leq n} \sqrt{\frac{\lambda t}{n}} Z_{i,j} \bbf{x}_i^{\intercal} \bbf{x}_j + \frac{\lambda t}{n} \bbf{x}_i^{\intercal} \bbf{x}_j \bbf{X}_i^{\intercal} \bbf{X}_j - \frac{\lambda t}{2n} (\bbf{x}_i^{\intercal} \bbf{x}_j)^2
	\\
	&+\sum_{1 \leq i<j\leq n} \sqrt{\frac{\lambda (1-t)}{n}} Z'_{i,j} \bbf{\bar{x}}_i^{\intercal} \bbf{\bar{x}}_j + \frac{\lambda (1-t)}{n} \bbf{\bar{x}}_i^{\intercal} \bbf{\bar{x}}_j \bbf{\bar{X}}_i^{\intercal} \bbf{\bar{X}}_j - \frac{\lambda (1-t)}{2n} (\bbf{\bar{x}}_i^{\intercal} \bbf{\bar{x}}_j)^2
	\end{align*}
	Define
	$$
	\phi(t) = \frac{1}{n} \E \log \left(  
		\int dP_0^{\otimes n}(\bbf{x}) \exp \big( H_{n,t}(\bbf{x})\big)
	\right)
	$$
	Remark that $\phi(0) = \bar{F}_n$ and $\phi(1)=F_n$.
	Let $\langle \cdot \rangle_t$ be the Gibbs measure on $S^n$ associated with the Hamiltonian $H_{n,t}$.
	For $0 < t < 1$ we have, by Gaussian integration by parts and the Nishimori property:
	$$
	\phi'(t)
	= \frac{\lambda}{2n^2}\sum_{1 \leq i < j \leq n} 
	\E \Big\langle \bbf{x}^{\intercal}_i \bbf{x}_j \bbf{X}^{\intercal}_i \bbf{X}_j - \bbf{\bar{x}}^{\intercal}_i \bbf{\bar{x}}_j \bbf{\bar{X}}^{\intercal}_i \bbf{\bar{X}}_j \Big\rangle
	= \frac{\lambda (n-1)}{4n}
	\E \Big\langle \bbf{x}^{\intercal}_1 \bbf{x}_2 \bbf{X}^{\intercal}_1 \bbf{X}_2 - \bbf{\bar{x}}^{\intercal}_1 \bbf{\bar{x}}_2 \bbf{\bar{X}}^{\intercal}_1 \bbf{\bar{X}}_2 \Big\rangle
	$$
	Therefore
	\begin{align*}
		|\phi'(t)|
	&\leq
		\frac{\lambda}{4} \left|
		\E \Big\langle 
	\bbf{x}^{\intercal}_1 \bbf{x}_2 \bbf{X}^{\intercal}_1 \bbf{X}_2 - \bbf{\bar{x}}^{\intercal}_1 \bbf{\bar{x}}_2 \bbf{X}^{\intercal}_1 \bbf{X}_2 
	+\bbf{\bar{x}}^{\intercal}_1 \bbf{\bar{x}}_2 \bbf{X}^{\intercal}_1 \bbf{X}_2 - \bbf{\bar{x}}^{\intercal}_1 \bbf{\bar{x}}_2 \bbf{\bar{X}}^{\intercal}_1 \bbf{\bar{X}}_2 
\Big\rangle
\right|
	\\
	&\leq
	\frac{\lambda}{4} \left|\E \Big\langle 
		(\bbf{x}^{\intercal}_1 \bbf{x}_2 - \bbf{\bar{x}}^{\intercal}_1 \bbf{\bar{x}}_2)\bbf{X}^{\intercal}_1 \bbf{X}_2 
	\Big\rangle \right|
	+
	\frac{\lambda}{4} \left|\E \Big\langle 
		\bbf{\bar{x}}^{\intercal}_1 \bbf{\bar{x}}_2 (\bbf{X}^{\intercal}_1 \bbf{X}_2 -\bbf{\bar{X}}^{\intercal}_1 \bbf{\bar{X}}_2)
	\Big\rangle \right|
	\\
	&\leq 
	\left(
		\E \Big\langle 
			(\bbf{x}^{\intercal}_1 \bbf{x}_2 - \bbf{\bar{x}}^{\intercal}_1 \bbf{\bar{x}}_2)^2
		\Big\rangle
		\E
		(\bbf{X}^{\intercal}_1 \bbf{X}_2)^2
	\right)^{1/2}
	+
	\left(
		\E 
			(\bbf{X}^{\intercal}_1 \bbf{X}_2 - \bbf{\bar{X}}^{\intercal}_1 \bbf{\bar{X}}_2)^2
		\E\Big\langle 
			(\bbf{\bar{x}}^{\intercal}_1 \bbf{\bar{x}}_2)^2
		\Big\rangle
	\right)^{1/2}
	\\
	&\leq 
	(\E \| \bbf{X}_1\|^2 + \E \| \bbf{\bar{X}}_1\|^2)
	\left(
		\E 
			(\bbf{X}^{\intercal}_1 \bbf{X}_2 - \bbf{\bar{X}}^{\intercal}_1 \bbf{\bar{X}}_2)^2
	\right)^{1/2}
	\leq 
	(\E \| \bbf{X}_1\|^2 + \E \| \bbf{\bar{X}}_1\|^2)
	\left(
		2\E 
		\big((\bbf{X}_1 - \bbf{\bar{X}}_1)^{\intercal}\bbf{X}_2 \big)^2
	\right)^{1/2}
	\\
	&\leq 
	(3 \E \| \bbf{X}_1\|^2 + 1)
	\left(
		4
		\sum_{i=1}^k \E (\bbf{X}_{1,i} - \bbf{\bar{X}}_{1,i})^2 \E(\bbf{X}_{2,i})^2
	\right)^{1/2}
\end{align*}
For $1 \leq i \leq k$, $\E (\bbf{X}_{1,i} - \bbf{\bar{X}}_{1,i})^2 \leq \epsilon^2 (1 + \E \bbf{X}_{1,i}^2)$.
We obtain that $|\phi'(t)| \leq \lambda K \epsilon$, where $K$ is a constant depending only on $\E \|\bbf{X}_1\|^2$.
Thus $|F_n - \bar{F}_n| = |\phi(1) - \phi(0)|  \leq \lambda K \epsilon$, which concludes the proof.
\end{proof}

\begin{lemma} \label{lem:approx_F_discrete}
	There exists a constant $K' > 0$ that depends only on $P_0$, such that
	$$
\left| \sup_{\bbf{q} \in S_k^+} \mathcal{F}(\lambda,\bbf{q})
-
\sup_{\bbf{q} \in S_k^+} \bar{\mathcal{F}}(\lambda,\bbf{q}) \right| \leq K' \epsilon
	$$
\end{lemma}
\begin{proof}
	First notice that both suprema are achieved over a common compact set $K \subset S_k^+$. Indeed, for $\bbf{q} \in S_k^+$,
	\begin{align*}
		\E \log \int dP_0(\bbf{x}) \exp(\sqrt{\lambda} (\bbf{Z}^{\intercal} \bbf{q}^{1/2} \bbf{x})+ \lambda \bbf{x}^{\intercal} \bbf{q} \bbf{X} - \frac{\lambda}{2}\bbf{x}^{\intercal} \bbf{q} \bbf{x})
		&=
		\E \log \int dP_0(\bbf{x}) e^{-\frac{1}{2} \| \bbf{Y} - \bbf{q}^{1/2} \bbf{x}\|^2} + \frac{\lambda}{2} \E \bbf{X}^{\intercal} \bbf{q} \bbf{X}
		\\
		&\leq 
		\frac{\lambda}{2} \text{Tr} \left[ \bbf{q} \E \bbf{X} \bbf{X}^{\intercal} \right] \leq \frac{\lambda}{2} \|\bbf{q}\| \|\E \bbf{X} \bbf{X}^{\intercal} \|
	\end{align*}
	One have a similar inequality for $\bar{P}_0$. Therefore both suprema are achieved over a common compact set $C \subset S_k^+$ (that depends only on $P_0$).
	Using the same kind of arguments than in the proof of Lemma~\ref{lem:approx_f_discrete}, we obtain that there exists a constant $K'$ that depends only on $P_0$ such that $\forall \bbf{q} \in C, \ | \mathcal{F}(\lambda,\bbf{q}) - \bar{\mathcal{F}}(\lambda, \bbf{q})| \leq \lambda K' \epsilon$. This proves the lemma.
\end{proof}
\\

Combining equation~\ref{eq:limit_discrete} and Lemmas~\ref{lem:approx_f_discrete} and~\ref{lem:approx_F_discrete}, we obtain that there exists $n_0 \geq 1$ such that for all $n \geq n_0$,
$$
| F_n - \sup_{\bbf{q} \in S_k^+} \mathcal{F}(\lambda,\bbf{q}) | \leq \lambda (K + K' +1) \epsilon
$$
where $K$ and $K'$ are two constants independent of $n_0$ and $\epsilon$. This proves Theorem~\ref{th:rs_formula_multidim}. It remains therefore to prove Theorem~\ref{th:rs_formula_multidim} for distribution $P_0$ with finite support $S$. In the following, we suppose such a distribution to be fixed.

\subsubsection{The lower bound: Guerra's interpolation technique}
We have the extension of Proposition~\ref{prop:guerra_bound}.

\begin{proposition} \label{prop:guerra_bound_multidim}
	\begin{equation} 
		\liminf_{n \to \infty} F_n \geq \sup_{\bbf{q} \in S_k^+} \mathcal{F}(\lambda,\bbf{q})
	\end{equation}
\end{proposition}

The proof is exactly the same as in the unidimensional case. The Nishimori property (Proposition~\ref{prop:nishimori}) applies to the general case, the computations are the same.

\subsubsection{Adding a small perturbation}
The results of Section~\ref{sec:small_perturbation} can be easily generalized.
Let us fix $\epsilon\in [0,1]$, and suppose we have access to the additional information, for $1 \leq i \leq n$
$$
\bbf{Y}'_i =
\begin{cases}
	\bbf{X}_i &\text{if } L_i = 1 \\
	* &\text{if } L_i = 0
\end{cases}
$$
where $L_i \overset{\text{\tiny i.i.d.}}{\sim}  \Ber(\epsilon)$ and $*$ is a value that does not belong to $S$. The free energy is now

$$
F_{n,\epsilon} = \frac{1}{n} \E \Big[\log \sum_{\bbf{x} \in S^n} P_0(\bbf{x}) \ e^{H_{n}(\bbf{\bar{x}})}\Big] 
$$
where $\bbf{\bar{x}} = (L_i \bbf{X}_i + (1-L_i)\bbf{x}_i)_{1 \leq i \leq n}$.
\\

The proof of Proposition~\ref{prop:approximation_f_n_epsilon}  can be adapted to the general case to obtain:
\begin{proposition} \label{prop:approximation_f_n_epsilon_epsilon}
For all $n \geq 1$ and for all $\epsilon,\epsilon' \in [0,1]$,
	$$
	|F_{n,\epsilon} - F_{n,\epsilon'} | \leq \lambda k^2 K_0^4 |\epsilon - \epsilon'|.
	$$
\end{proposition}
We define now $\epsilon$ as a uniform random variable over $[0, 1]$, independently of every other random variable. We will note $\E_{\epsilon}$ the expectation with respect to $\epsilon$. For $n \geq 1$, we define also $\epsilon_n = n^{-1/2} \epsilon \sim \mathcal{U}[0, n^{-1/2}]$. Proposition~\ref{prop:approximation_f_n_epsilon_epsilon} implies that
$$
\big| F_n - \E_{\epsilon} [F_{n, \epsilon_n}] \big| \xrightarrow[n \to \infty]{} 0.
$$
It remains therefore to compute the limit of the free energy averaged over small perturbations.

\subsubsection{Aizenman-Sims-Starr scheme}

In the multidimensional case the overlaps becomes $k \times k$ matrices. For $\bbf{x}^{(1)},\bbf{x}^{(2)} \in S^n \sim M_{n,k}(\R)$ we write
$$
\bbf{x}^{(1)}.\bbf{x}^{(2)} = \frac{1}{n} \sum_{i=1}^n \bx_i^{(1)} (\bx_i^{(2)})^{\intercal}
\in M_{k,k}(\R)
$$
In this section, $\| \cdot \|$ will denote the norm over $M_{k,k}(\R)$ defined as $\| A \| = \sqrt{\text{Tr}(A^{\intercal}A)}$.
We can adapt Section~\ref{sec:aizenman} to the general case. Define for $\bbf{x} \in S^n$
$$
H_{n}'(\bbf{x}) = \sum_{1 \leq i<j \leq n} \bbf{x}_i^{\intercal} \bbf{x}_j \sqrt{\frac{\lambda}{n+1}} (Z_{i,j} + \sqrt{\frac{\lambda}{n+1}} \bbf{X}_i^{\intercal} \bbf{X}_j ) - \frac{\lambda}{2(n+1)} (\bbf{x}_i^{\intercal} \bbf{x}_j)^2 
$$
and the Gibbs measure $\langle \cdot \rangle$ by
\begin{equation} \label{eq:def_gibbs_general}
	\langle f(\bbf{x}) \rangle = \frac{\sum_{\bbf{x} \in S^n} P_0(\bbf{x}) f(\bbf{\bar{x}}) \exp(H_n'(\bbf{\bar{x}}))}{\sum_{\bbf{x} \in S^n} P_0(\bbf{x}) \exp(H_n'(\bbf{\bar{x}}))}
\end{equation}
for any function $f$ on $S^n$, where we recall that $\bbf{\bar{x}} = (L_i \bbf{X}_i + (1-L_i)\bbf{x}_i)_{1 \leq i \leq n}$ (where $L_i \overset{\text{\tiny i.i.d.}}{\sim}  \Ber(\epsilon_n)$, independently of everything else).
We have then, by the same decomposition as in Section~\ref{sec:aizenman},
\begin{equation} \label{eq:limsup2_general}
	\limsup_{n \rightarrow \infty} F_n
	=
	\limsup_{n \rightarrow \infty} \E_{\epsilon} [F_{n,\epsilon_n}]
	\leq
	\limsup_{n \rightarrow \infty} \E_{\epsilon} [A_{n}]
\end{equation}
where
\begin{align*}
	A_{n} = \E \log \big\langle  \int_{\bbf{\sigma} \in S} d P_0(\bbf{\sigma})\exp(\bbf{\bar{\sigma}}^{\intercal} z(\bbf{x}) + \bbf{\bar{\sigma}}^{\intercal} s(\bbf{x}) \bbf{\bar{\sigma}} ) \big\rangle
	- \E  \log \big\langle \exp(y(\bbf{x})) \big\rangle 
\end{align*}
where $\bar{\sigma} = (1-L_{n+1}) \sigma + L_{n+1} \bbf{X_{n+1}}$ and
\begin{align*}
	z(\bbf{x}) &= \sum_{i=1}^n \sqrt{\frac{\lambda}{n}} Z_{i,n+1} \bbf{x}_i + \frac{\lambda}{n} (\bbf{x}_i^{\intercal} \bbf{X}_i) \bbf{X}_{n+1}  \\
	s(\bbf{x}) &= -\frac{\lambda}{2n} \sum_{i=1}^n \bbf{x}_i \bbf{x}_i^{\intercal} \\
	y(\bbf{x}) 
	&= \frac{\sqrt{\lambda}}{\sqrt{2}n} \sum_{i=1}^n Z_i'' \|\bbf{x}_i\|^2 
	+ \frac{\sqrt{\lambda}}{n} \sum_{1 \leq i<j \leq n} \bbf{x}_i^{\intercal} \bbf{x}_j \tilde{Z}_{i,j} 
	+ \frac{\lambda}{2} (\|\bbf{x}.\bbf{X}\|^2 - \frac{1}{2} \|\bbf{x}.\bbf{x}\|^2)
\end{align*}
where $\tilde{Z}_{i,j}, Z_i'' \overset{\text{\tiny i.i.d.}}{\sim}  \mathcal{N}(0,1)$ independently of any other random variables.

\subsubsection{Overlap concentration} 

Recall that $\langle \cdot \rangle$ is the Gibbs measure defined in equation~\eqref{eq:def_gibbs_general}. $\langle \cdot \rangle$ correspond to the posterior distribution of $\bbf{X}$ given $\bbf{Y}$ and $\bbf{Y'}$ in the following observation channel
\begin{align*}
	Y_{i,j} &= \sqrt{\frac{\lambda}{n+1}} \bX_i^{\intercal} \bX_j + Z_{i,j}, \ \ \text{for } \ 1 \leq i<j \leq n \\
	\bbf{Y}'_i &=
	\begin{cases}
		\bX_i &\text{if } L_i = 1 \\
		* &\text{if } L_i = 0
	\end{cases} \ \text{for} \ 1 \leq i \leq n
\end{align*} 
where $\bX_i \overset{\text{\tiny i.i.d.}}{\sim}  P_0$, $Z_{i,j} \overset{\text{\tiny i.i.d.}}{\sim}  \mathcal{N}(0,1)$ and $L_i \overset{\text{\tiny i.i.d.}}{\sim}  \Ber(\epsilon_n)$ are independent random variables. The Nishimori property (Proposition~\ref{prop:nishimori}) will thus be valid under $\langle \cdot \rangle$.
Lemma 3.1 from~\cite{andrea2008estimating} gives
\begin{lemma}
	$$
	n^{-1/2} \E_{\epsilon}  \left[ \frac{1}{n^2} \sum_{1 \leq i,j \leq j} I(\bX_i;\bX_j | \bbf{Y},\bbf{Y'}) \right] \leq \frac{2 H(P_0)}{n} 
	$$
\end{lemma}

This implies that the overlap between two replicas, i.e.\ two independent samples $\bbf{x}^{(1)}$ and $\bbf{x}^{(2)}$ from the Gibbs distribution $\langle \cdot \rangle$, concentrates. Let us define
\begin{align*}
	\bbf{Q} &= \Big\langle \frac{1}{n} \sum_{i=1}^n \bx^{(1)}_i (\bx^{(2)}_i)^{\intercal} \Big\rangle = \langle \bx^{(1)}.\bx^{(2)} \rangle \\
	\bbf{b}_i &= \langle \bx_i \rangle
\end{align*}
$\bbf{Q}$ is a random variable depending only on $(Y_{i,j})_{i<j\leq n}$ and $(\bbf{Y}_i')_{i \leq n}$. Notice that $\bbf{Q} = \bbf{b}.\bbf{b} \in S_k^+$. 

\begin{proposition}[Overlap concentration] 
	\begin{align*}
		\E_{\epsilon} \E \Big\langle \| \bx^{(1)}.\bx^{(2)} - \bbf{Q} \|^2 \Big\rangle \xrightarrow[n \to \infty]{} 0
	\end{align*}
\end{proposition}

\begin{proof}
	\begin{align*}
		\big\langle \|\bx^{(1)}.\bx^{(2)} &- \bbf{Q} \|^2 \big\rangle =
		\langle \|\bx^{(1)}.\bx^{(2)}\|^2 \rangle - \|\langle \bx^{(1)}.\bx^{(2)} \rangle\|^2 \\
		&= \frac{1}{n^2} \sum_{1 \leq i,j \leq n} \langle \bx_i^{(1)\intercal} \bx_i^{(2)} \bx_j^{(1)\intercal} \bx_j^{(2)} \rangle - \langle \bx_i^{(1)\intercal} \rangle \langle \bx_i^{(2)} \rangle \langle \bx_j^{(1)\intercal} \rangle \langle \bx_j^{(2)} \rangle \\
		&= \frac{1}{n^2} \sum_{1 \leq i,j \leq n} \|\langle \bx_i \bx_j^{\intercal} \rangle \|^2 - \| \langle \bx_i^{\intercal} \rangle \langle \bx_j \rangle \|^2 \\
		&\leq \frac{C}{n^2} \sum_{1 \leq i,j \leq n} \Big| \|\langle \bx_i \bx_j^{\intercal} \rangle \| - \| \langle \bx_i \rangle \langle \bx_j^{\intercal} \rangle \| \Big| 
		\leq \frac{C}{n^2} \sum_{1 \leq i,j \leq n} \|\langle \bx_i \bx_j^{\intercal} \rangle - \langle \bx_i \rangle \langle \bx_j^{\intercal} \rangle \| \\
		&\leq \frac{C}{n^2} \sum_{1 \leq i,j \leq n} \Big\| \sum_{\bx_i,\bx_j} \bx_i \bx_j^{\intercal} \PP(\bX_i=\bx_i, \bX_j=\bx_j | \bbf{Y},\bbf{Y}')- \bx_i \bx_j^{\intercal} \PP(\bX_i=\bx_i | \bbf{Y},\bbf{Y}') \PP(\bX_j = \bx_j |\bbf{Y},\bbf{Y}') \Big\| \\
		&\leq \frac{C'}{n^2} \sum_{1 \leq i,j \leq n} \Dtv \big(\PP(\bX_i=., \bX_j=. |\bbf{Y},\bbf{Y}'); \PP(\bX_i=. | \bbf{Y},\bbf{Y}')\otimes \PP(\bX_j =. |\bbf{Y},\bbf{Y}') \big)\\
		&\leq \frac{C''}{n^2} \sum_{1 \leq i,j \leq n} \sqrt{\Dkl \big(\PP(\bX_i=., \bX_j=. |\bbf{Y},\bbf{Y}'); \PP(\bX_i=. | \bbf{Y},\bbf{Y}')\otimes \PP(\bX_j =. |\bbf{Y},\bbf{Y}') \big)}\\
		&\leq C'' \sqrt{\frac{1}{n^2} \sum_{1 \leq i,j \leq n} \Dkl \big(\PP(\bX_i=., \bX_j=. |\bbf{Y},\bbf{Y}'); \PP(\bX_i=. | \bbf{Y},\bbf{Y}')\otimes \PP(\bX_j =. |\bbf{Y},\bbf{Y}') \big)}
	\end{align*}
	for some constants $C,C',C'' >0$, where we used Pinsker's inequality to compare the total variation distance $\Dtv$ with the Kullback-Leibler divergence $\Dkl$. So that:
	\begin{align*}
		\E_{\epsilon} \E \Big\langle \| \bx^{(1)}.\bx^{(2)} - \bbf{Q} \|^2 \Big\rangle \leq
		C'' \sqrt{\E_{\epsilon} \Big[ \frac{1}{n^2} \sum_{1 \leq i,j \leq n} I(\bX_i;\bX_j | \bbf{Y},\bbf{Y'}) \Big]}
		\ \xrightarrow[n \to \infty]{} 0
	\end{align*}
\end{proof}
\\

As a consequence of the Nishimori property, the overlap between one replica and the planted solution concentrates around the same value as the overlap between two independent replicas.
\begin{corollary} 
	$$
	\E_{\epsilon} \E \Big\langle \| \bbf{x}.\bbf{X}-\bbf{Q} \|^2 \Big\rangle \xrightarrow[n \to \infty]{} 0
	\ \ \ \text{and } \ \ 
	\E_{\epsilon} \E \Big\langle \|\bbf{x}.\bbf{b}- \bbf{Q} \|^2 \Big\rangle \xrightarrow[n \to \infty]{} 0
	$$
\end{corollary}

The remaining part of the proof is then exactly the same than for the finite, unidimensional case.

\section{Application to community detection in the stochastic block model}\label{sec:sbm}

We prove in this section the phase transition for community detection on the stochastic block model, namely Theorem~\ref{th:solvability_sbm}.

We will make use of the following notation. For $\bbf{x} \in \{1,2\}^n$ we denote $\tilde{\bbf{x}} = (\phi_p(x_1), \dots, \phi_p(x_n)) \in \left\{-\sqrt{\frac{p}{1-p}}, \sqrt{\frac{1-p}{p}} \right\}^n$, where $\phi_p(1) = \sqrt{\frac{1-p}{p}}$ and $\phi_p(2)=-\sqrt{\frac{p}{1-p}}$.

In order to apply the results we proved for matrix factorization, we first show that the mutual information between the observed graph $\bbf{G}$ and the hidden labels $\bbf{X}$ is asymptotically equal to the mutual information between $\bbf{Y}$ and $\tilde{\bbf{X}}$ in the following Gaussian observation channel
$$
Y_{i,j} = \sqrt{\frac{\lambda}{n}} \tilde{X}_i \tilde{X}_j + Z_{i,j} \qquad \text{for } 1 \leq i<j \leq n
$$
where $\tilde{X}_i = \phi_p(X_i)$ and $Z_{i,j} \overset{\text{\tiny i.i.d.}}{\sim}  \mathcal{N}(0,1)$ independently of everything else. 
The following theorem generalize the result from~\cite{deshpande2016asymptotic} to the asymmetric case.

\begin{theorem} \label{th:information_sbm}
	There exists a constant $C>0$ such that, for $d$ large enough
	$$
	\limsup_{n \to \infty} \frac{1}{n} \Big| I(\bbf{X},\bbf{G}) - I(\tilde{\bbf{X}},\bbf{Y}) \Big| \leq C \sqrt{\frac{\lambda}{d}}
	$$
\end{theorem}
Now, Theorem~\ref{th:rs_formula} allows us to compute the limit of $\frac{1}{n} I(\tilde{\bbf{X}},\bbf{Y})$. Define
\begin{align}
	\mathcal{F}_g: (\lambda,q) \mapsto -\frac{\lambda q^2 }{4} + \E_{\tilde{X}_0,Z_0} \log \Big[ 
		&p\exp(\sqrt{\frac{1-p}{p}}(\sqrt{\lambda q}Z_0 + \lambda q \tilde{X}_0) - \frac{\lambda (1-p)}{2p}q ) \nonumber
		\\
		&+(1-p)\exp(- \sqrt{\frac{p}{1-p}}(\sqrt{\lambda q}Z_0 + \lambda q \tilde{X}_0) - \frac{\lambda p}{2(1-p)}q )
	\Big]
\end{align}
where the expectation is taken over $Z_0 \sim \mathcal{N}(0,1)$ and  $\tilde{X}_0 \sim p \delta_{\phi_p(1)} + (1-p) \delta_{\phi_p(2)}$ independently from $Z_0$. Define also 
\begin{equation} \label{eq:phi_g}
	\Phi_g: \lambda \mapsto \sup_{q \geq 0} \mathcal{F}_g(\lambda,q).
\end{equation}
Corollary~\ref{cor:rs_information} gives then

\begin{corollary} \label{cor:information_sbm}
	There exists a constant $C>0$ such that, for $d$ large enough
	$$
	\limsup_{n \to \infty} \Big| \frac{1}{n} I(\bbf{X},\bbf{G}) - \big( \frac{\lambda}{4} - \Phi_g(\lambda) \big) \Big| \leq C \sqrt{\frac{\lambda}{d}}
	$$
\end{corollary}

\subsection{The limit of the mutual information: proof of Theorem~\ref{th:information_sbm}}

We are going to compute $I(\bbf{X},\bbf{G})$ and $I(\tilde{\bbf{X}},\bbf{G})$. We recall that for $\bbf{x} \in \{1,2\}^n$ we denote $\tilde{\bbf{x}} = (\phi_p(x_1), \dots, \phi_p(x_n)) \in S_p^n $, where $\phi_p(1) = \sqrt{\frac{1-p}{p}}$, $\phi_p(2)=-\sqrt{\frac{p}{1-p}}$ and $S_p= \left\{-\sqrt{\frac{p}{1-p}}, \sqrt{\frac{1-p}{p}} \right\}$.
\\

\noindent For $x \in \{1,2\}$ we define $P_0(x)=P_0(\tilde{x})=p$, if $x=1$ and $P_0(x)=P_0(\tilde{x})=1-p$, if $x=2$. For $\bbf{x} \in \{1,2\}^n$ we will write, with a slight abuse of notation
$$
P_0(\bbf{x}) = P_0(\tilde{\bbf{x}}) = \prod_{i=1}^n P_0(x_i)
$$
The following lemma is a consequence of the general result from Lemma~\ref{lem:mutual_information_general}.
\begin{lemma}
	$$
	I(\tilde{\bbf{X}},\bbf{Y}) = - \E \left[ \log \sum_{\tilde{\bbf{x}} \in S_p^n} P_0(\tilde{\bbf{x}}) \exp\left( \sum_{i<j} (\tilde{x}_i \tilde{x}_j - \tilde{X}_i \tilde{X}_j) \sqrt{\frac{\lambda}{n}} Z_{i,j} - \frac{\lambda}{2n}(\tilde{x}_i \tilde{x}_j - \tilde{X}_i \tilde{X}_j)^2 \right)\right]
	$$
\end{lemma}
Define $V_{i,j} = \epsilon ( G_{i,j} - \E(G_{i,j} \vert X_i, X_j))$
\begin{lemma}
	For $d$ large enough,
	$$
	I(\bbf{X},\bbf{G}) \!=\! - \E \Big[ \log \!\!\!\! \sum_{\bbf{x} \in \{1,2\}^n} \!\!\!\! P_0(\bbf{x}) \exp( \sum_{i<j} (\tilde{x}_i \tilde{x}_j - \tilde{X}_i \tilde{X}_j) V_{i,j}  
			- \frac{1}{2} \epsilon ((\tilde{x}_i \tilde{x}_j)^2 - (\tilde{X}_i \tilde{X}_j)^2) V_{i,j}
	- \frac{\lambda}{2n}(\tilde{x}_i \tilde{x}_j - \tilde{X}_i \tilde{X}_j)^2 )\Big] + O(n \epsilon)
	$$
\end{lemma}

\begin{proof}
	By definition, $I(\bbf{X},\bbf{G}) = \E \log \frac{\PP(\bbf{X},\bbf{G})}{\PP(\bbf{X}) \PP(\bbf{G})} = - \E \log \frac{\PP(\bbf{G})}{\PP(\bbf{G}|\bbf{X})}$. Thus
	\begin{align*}
		I(\bbf{X},\bbf{G}) = - \E \log \frac{\sum_{\bbf{x} \in \{1,2\}^n} P_0(\bbf{x}) \PP(\bbf{G}|\bbf{x})}{\PP(\bbf{G}|\bbf{X})}
	\end{align*}
	Recall that $\PP(\bbf{G}|\bbf{x}) = \prod_{i<j} M_{x_i,x_j}^{G_{i,j}} (1-M_{x_i,x_j})^{1-G_{i,j}} = \exp(\sum_{i<j} G_{i,j} \log M_{x_i,x_j} + (1-G_{i,j}) \log (1 - M_{x_i,x_j}))$. This leads to
	\begin{equation} \label{eq:i_x_g}
	I(\bbf{X},\bbf{G}) = - \E \Big[ \log \big(  \sum_{\bbf{x} \in \{1,2\}^n} P_0(\bbf{x}) \exp( \sum_{i<j} G_{i,j} \log(\frac{M_{x_i,x_j}}{M_{X_i,X_j}}) + (1 - G_{i,j}) \log( \frac{1 - M_{x_i,x_j}}{1- M_{X_i,X_j}} ) )\big) \Big]
\end{equation}
	Notice that $M_{x_i,x_j} = \frac{d}{n} (1 + \tilde{x}_i \tilde{x}_j \epsilon)$.
	Therefore $\log(\frac{M_{x_i,x_j}}{M_{X_i,X_j}}) = \log(\frac{1 + \tilde{x}_i \tilde{x}_j \epsilon}{1 + \tilde{X}_i \tilde{X}_j \epsilon})$. By the Taylor-Lagrange inequality, there exist a constant $C>0$ such that, for $\epsilon$ small enough (i.e.\ for $d$ large enough):
	\begin{align}
		&\Big| \log(\frac{1 + \tilde{x}_i \tilde{x}_j \epsilon}{1 + \tilde{X}_i \tilde{X}_j \epsilon}) - \epsilon (\tilde{x}_i \tilde{x}_j - \tilde{X}_i \tilde{X}_j) + \frac{1}{2} \epsilon^2 ((\tilde{x}_i \tilde{x}_j)^2 - (\tilde{X}_i \tilde{X}_j)^2) \Big| \leq C \epsilon^3 \label{eq:dl_log_1}
		\\
		& \Big| \log( \frac{1 - M_{x_i,x_j}}{1- M_{X_i,X_j}} ) + \frac{d}{n} \epsilon ( \tilde{x}_i \tilde{x}_j - \tilde{X}_i \tilde{X}_j ) \Big| \leq C \frac{d^2}{n^2} \label{eq:dl_log_2}
	\end{align}
	By summation and triangle inequality:
	\begin{align*}
		&\sum_{i<j} G_{i,j} \log(\frac{M_{x_i,x_j}}{M_{X_i,X_j}}) + (1 - G_{i,j}) \log( \frac{1 - M_{x_i,x_j}}{1- M_{X_i,X_j}} ) 
		\\
		&=
		\sum_{i<j} \epsilon (\tilde{x}_i \tilde{x}_j - \tilde{X}_i \tilde{X}_j) G_{i,j} 
		- \frac{1}{2} \epsilon^2 ((\tilde{x}_i \tilde{x}_j)^2 - (\tilde{X}_i \tilde{X}_j)^2) G_{i,j} 
		- (1 - G_{i,j}) (\tilde{x}_i \tilde{x}_j - \tilde{X}_i \tilde{X}_j) \frac{d}{n} \epsilon + \Delta_n 
	\end{align*}
	where $|\Delta_n| \leq C (d^2 + \sum_{i<j} G_{i,j} \epsilon^3)$ because of equations~\eqref{eq:dl_log_1} and~\eqref{eq:dl_log_2}. Then
	\begin{align*}
		&\sum_{i<j} G_{i,j} \log(\frac{M_{x_i,x_j}}{M_{X_i,X_j}}) + (1 - G_{i,j}) \log( \frac{1 - M_{x_i,x_j}}{1- M_{X_i,X_j}} ) 
		\\
		&=
		\sum_{i<j} \Big( \epsilon (\tilde{x}_i \tilde{x}_j - \tilde{X}_i \tilde{X}_j) (G_{i,j} - \frac{d}{n} - \frac{d \tilde{X}_i \tilde{X}_j \epsilon}{n}) 
			+ \epsilon^2 \frac{d}{n} (\tilde{x}_i \tilde{x}_j - \tilde{X}_i \tilde{X}_j) \tilde{X}_i \tilde{X}_j
			- \frac{1}{2} \epsilon^2 ((\tilde{x}_i \tilde{x}_j)^2 - (\tilde{X}_i \tilde{X}_j)^2) G_{i,j} 
			\\
		&+ G_{i,j} (\tilde{x}_i \tilde{x}_j - \tilde{X}_i \tilde{X}_j) \frac{d}{n} \epsilon \Big) + \Delta_n
			\\
			&=
			\sum_{i<j} \Big( \epsilon (\tilde{x}_i \tilde{x}_j - \tilde{X}_i \tilde{X}_j) V_{i,j} 
				+ \epsilon^2 \frac{d}{n} (\tilde{x}_i \tilde{x}_j - \tilde{X}_i \tilde{X}_j) \tilde{X}_i \tilde{X}_j
				\\
				&- \frac{1}{2} \epsilon^2 ((\tilde{x}_i \tilde{x}_j)^2 - (\tilde{X}_i \tilde{X}_j)^2) V_{i,j} 
				- \frac{1}{2} \epsilon^2 ((\tilde{x}_i \tilde{x}_j)^2 - (\tilde{X}_i \tilde{X}_j)^2) \frac{d}{n}(1+\tilde{X}_i \tilde{X}_j \epsilon)
			+ G_{i,j} (\tilde{x}_i \tilde{x}_j - \tilde{X}_i \tilde{X}_j) \frac{d}{n} \epsilon \Big) + \Delta_n 
			\\
			&=
			\sum_{i<j} \Big(  \epsilon (\tilde{x}_i \tilde{x}_j - \tilde{X}_i \tilde{X}_j) V_{i,j} 
				- \frac{1}{2} \epsilon^2 ((\tilde{x}_i \tilde{x}_j)^2 - (\tilde{X}_i \tilde{X}_j)^2) V_{i,j} 
				\\
				&+ \epsilon^2 \frac{d}{n} (\tilde{x}_i \tilde{x}_j - \tilde{X}_i \tilde{X}_j) (\tilde{X}_i \tilde{X}_j
			- \frac{\tilde{x}_i \tilde{x}_j + \tilde{X}_i \tilde{X}_j}{2} (1+\tilde{X}_i \tilde{X}_j \epsilon))
		+ G_{i,j} (\tilde{x}_i \tilde{x}_j - \tilde{X}_i \tilde{X}_j) \frac{d}{n} \epsilon \Big) + \Delta_n 
		\\
		&=
		\sum_{i<j} \Big( \epsilon (\tilde{x}_i \tilde{x}_j - \tilde{X}_i \tilde{X}_j) V_{i,j} 
			- \frac{1}{2} \epsilon^2 ((\tilde{x}_i \tilde{x}_j)^2 - (\tilde{X}_i \tilde{X}_j)^2) V_{i,j} 
			- \frac{\lambda}{2n} (\tilde{x}_i \tilde{x}_j - \tilde{X}_i \tilde{X}_j)^2
		+ O(\frac{\epsilon}{n}) + G_{i,j} (\tilde{x}_i \tilde{x}_j - \tilde{X}_i \tilde{X}_j) \frac{d}{n} \epsilon \Big) 
		\\
		&+ \Delta_n 
	\end{align*}
	Notice that $G_{i,j} \in \{0,1\}$, we have therefore, for some constant $C' >0$,
	$$
	\vert \sum_{i<j} G_{i,j}(\tilde{x}_i \tilde{x}_j - \tilde{X}_i \tilde{X}_j) \frac{d}{n} \vert \leq C' \frac{d}{n} \sum_{i<j} G_{i,j}
	$$
	We use then the following lemma to control $\sum_{i<j} G_{i,j}$.
	\begin{lemma} For $d$ large enough and $n$ large enough, we have
		$$
		\PP \big( \sum_{i<j} G_{i,j} > 2dn \big) \leq \exp(-\frac{1}{2}nd)
		$$
	\end{lemma}
	\begin{proof}
		Conditionally to $\bbf{X}$, the variables $(G_{i,j})_{i<j}$ are independent and distributed as Bernoulli random variables with expectations equal to $\frac{ad}{n}$, $\frac{bd}{n}$ or $\frac{cd}{n}$, depending of the $\bbf{X}$ variables. Notice that $a,b,c \xrightarrow[d \to \infty]{} 1$, therefore for $d$ large enough the variables $G_{i,j}$ are stochastically bounded by independent identically distributed variables $B_{i,j} \sim \Ber(\frac{2d}{n})$. Thus, using standard concentration result for binomial random variables,
		$$
		\PP( \sum_{i<j} G_{i,j} > 2dn )
		\leq \PP( \sum_{i<j} B_{i,j} > 2dn )
		\leq \exp \big( - \frac{1}{2} n(n-1) \text{kl}(\frac{4d}{n-1},\frac{2d}{n}) \big)
		$$
		where $\text{kl}(\frac{4d}{n-1},\frac{2d}{n})=\frac{4d}{n-1} \log (\frac{2n}{n-1}) + (1- \frac{4d}{n-1}) \log \frac{1- 4d/(n-1)}{1-2d/n} > \frac{d}{n-1}$ for $n$ large enough.
	\end{proof}
	\\

	\noindent We have now
	\begin{align*}
		I(\bbf{X},\bbf{G}) 
		&= - \E \Big[ \log \sum_{\tilde{\bbf{x}} \in S_p^n} P_0(\tilde{\bbf{x}}) \exp\Big( \Delta_n + O(n\epsilon) + O(\frac{d}{n} \sum_{i<j} G_{i,j}) \\
		&+
		\sum_{i<j} (\tilde{x}_i \tilde{x}_j - \tilde{X}_i \tilde{X}_j) V_{i,j}  
		- \frac{1}{2} \epsilon ((\tilde{x}_i \tilde{x}_j)^2 - (\tilde{X}_i \tilde{X}_j)^2) V_{i,j}
- \frac{\lambda}{2n}(\tilde{x}_i \tilde{x}_j - \tilde{X}_i \tilde{X}_j)^2 \Big) \Big] \\
&= - \E \Big[ \log \sum_{\tilde{\bbf{x}} \in S_p^n} P_0(\tilde{\bbf{x}}) \exp\Big( O((\frac{d}{n} + \epsilon^3) \sum_{i<j} G_{i,j}) \\
&+
\sum_{i<j} (\tilde{x}_i \tilde{x}_j - \tilde{X}_i \tilde{X}_j) V_{i,j}  
- \frac{1}{2} \epsilon ((\tilde{x}_i \tilde{x}_j)^2 - (\tilde{X}_i \tilde{X}_j)^2) V_{i,j}
			- \frac{\lambda}{2n}(\tilde{x}_i \tilde{x}_j - \tilde{X}_i \tilde{X}_j)^2 \Big) \Big] + O(n\epsilon) 
		\end{align*}
		Distinguishing the cases $\sum_{i<j} G_{i,j} > 2dn$ (which happens with an exponentially small probability) and $\sum_{i<j} G_{i,j} \leq 2dn$, we obtain the desired result.
	\end{proof}

	\subsubsection{Lindeberg argument}

		We recall the Lindeberg generalization theorem (Theorem 2 from~\cite{korada2011lindeberg}).
		\begin{theorem}[Lindeberg generalization theorem] \label{th:lindeberg}
			Let $(U_i)_{1 \leq i \leq n}$ and $(V_i)_{1 \leq i \leq n}$ be two collection of random variables with independent components and $f: \R^n \to \R$ a $\mathcal{C}^3$ function. Denote $a_i = | \E U_i - \E V_i |$ and $b_i = | \E U_i^2 - \E V_i^2 |$. Then
			\vspace{-0.2cm}
			\begin{align*}
				|\E f(U) - \E f(V) | & \leq
				\sum_{i=1}^n \!\Big(\!
					a_i \E | \partial_i f(U_{1:i-1},0,V_{i+1:n}) | + \frac{b_i}{2} \E | \partial^2_i f(U_{1:i-1}, 0,V_{i+1:n})| \\
					&+ \frac{1}{2} \E \!\int_{0}^{U_i}\!\!\! |\partial^3_i f(U_{1:i-1}, 0,V_{i+1:n})|(U_i - s)^2 ds 
					+ \frac{1}{2} \E\! \int_{0}^{V_i}\!\!\! |\partial^3_i f(U_{1:i-1}, 0,V_{i+1:n})|(V_i - s)^2 ds
				\Big)
			\end{align*}
		\end{theorem}
	Define
	\vspace{-0.3cm}
	$$
	J(\bbf{X},\bbf{Z})= - \E \log \sum_{\tilde{\bbf{x}} \in S_p^n} P_0(\tilde{\bbf{x}}) \exp( \sum_{i<j} (\tilde{x}_i \tilde{x}_j - \tilde{X}_i \tilde{X}_j) \sqrt{\frac{\lambda}{n}} Z_{i,j} 
		- \frac{1}{2} \epsilon ((\tilde{x}_i \tilde{x}_j)^2 - (\tilde{X}_i \tilde{X}_j)^2) \sqrt{\frac{\lambda}{n}} Z_{i,j} 
	- \frac{\lambda}{2n}(\tilde{x}_i \tilde{x}_j - \tilde{X}_i \tilde{X}_j)^2 )
	\vspace{-0.3cm}
	$$
	We will show, using Theorem~\ref{th:lindeberg}, that $J(\bbf{X},\bbf{Z})$ is close to $I(\bbf{X},\bbf{G})$.
	\begin{lemma}
		$$
		\frac{1}{n} \vert I(\bbf{X},\bbf{G}) - J(\bbf{X},\bbf{Z}) \vert = O(\epsilon)
		$$
	\end{lemma}

	\begin{proof}
		We apply here Theorem~\ref{th:lindeberg} conditionally to $\bbf{X}$ to the function
		$$
		\Phi(u) = - \log \sum_{\tilde{\bbf{x}} \in S_p^n} P_0(\tilde{\bbf{x}}) \exp( \sum_{i<j} (\tilde{x}_i \tilde{x}_j - \tilde{X}_i \tilde{X}_j) u_{i,j} 
			- \frac{1}{2} \epsilon ((\tilde{x}_i \tilde{x}_j)^2 - (\tilde{X}_i \tilde{X}_j)^2) u_{i,j}
		- \frac{\lambda}{2n}(\tilde{x}_i \tilde{x}_j - \tilde{X}_i \tilde{X}_j)^2 )
		$$
		$\Phi$ is $\mathcal{C}^3$ with bounded derivatives (because $\alpha$ is bounded).
		Notice that $I(\bbf{X},\bbf{G}) = \E \Phi(\bbf{V})$ and $J(\bbf{X},\bbf{Z}) = \E \Phi(\sqrt{\frac{\lambda}{n}} \bbf{Z})$.
		Let us compute $V_{i,j}$ moments, conditionally to $\bbf{X}$.
		\begin{align*}
			\E (V_{i,j} | \bbf{X}) &= 0 \\
			\E (V_{i,j}^2 | \bbf{X}) &= \epsilon^2 \text{Var}(G_{i,j} \vert \bbf{X}) 
			=
			\epsilon^2 \frac{d}{n}(1 + \tilde{X}_i \tilde{X}_j \epsilon) (1 + O(\frac{d}{n}))
			\\
			&=
			(1-\epsilon) \frac{\lambda}{n} (1 + \tilde{X}_i \tilde{X}_j \epsilon) (1 + O(\frac{d}{n}))
			\\
			&= \frac{\lambda}{n} + O(\frac{\epsilon}{n}) = \frac{\lambda}{n} \E(Z_{i,j}^2) + O(\frac{\epsilon}{n})
		\end{align*}
		Analogously, $\E(V_{i,j}^3 | \bbf{X}) = O(\frac{\epsilon}{n})$. Using the Lindeberg generalization theorem we obtain
		\begin{align*}
			\vert \E (\Phi(\sqrt{\frac{\lambda}{n}} \bbf{Z})) - \E(\Phi(\bbf{V})) \vert &\leq C'' \sum_{i<j} O(\frac{\epsilon}{n}) = O(n \epsilon)
		\end{align*}
	\end{proof}

	\subsubsection{Gaussian interpolation}

	It remains to show

	\begin{lemma}
		$$
		I(\bbf{X},\bbf{Y}) = J(\bbf{X},\bbf{Z}) + O(n\epsilon)
		$$
	\end{lemma}

	\begin{proof}
		The $Z_{i,j}$ are centered and independent of $X_{i,j}$ we can therefore simplify:
		\begin{align*}
			&I(\bbf{X},\bbf{Y})=
			- \E \Big[ \log \sum_{\tilde{\bbf{x}} \in S_p^n} P_0(\tilde{\bbf{x}}) \exp( \sum_{i<j}  \sqrt{\frac{\lambda}{n}} Z_{i,j} \tilde{x}_i \tilde{x}_j - \frac{\lambda}{2n}(\tilde{x}_i \tilde{x}_j - \tilde{X}_i \tilde{X}_j)^2 )\Big]
			\\
			&J(\bbf{X},\bbf{Z}) = - \E \Big[ \log \sum_{\tilde{\bbf{x}} \in S_p^n} P_0(\tilde{\bbf{x}}) \exp( \sum_{i<j} \sqrt{\frac{\lambda}{n}} Z_{i,j} \tilde{x}_i \tilde{x}_j
			- \frac{\lambda}{2n}(\tilde{x}_i \tilde{x}_j - \tilde{X}_i \tilde{X}_j)^2 
	- \frac{1}{2} \epsilon (\tilde{x}_i \tilde{x}_j)^2  \sqrt{\frac{\lambda}{n}} Z_{i,j}) \Big]
\end{align*}
We define:
\begin{align*}
	H(\bbf{x},\bbf{X},\bbf{Z},\epsilon) &= \sum_{i<j} \tilde{x}_i \tilde{x}_j \sqrt{\frac{\lambda}{n}} Z_{i,j} 
	- \frac{\lambda}{2n}(\tilde{x}_i \tilde{x}_j - \tilde{X}_i \tilde{X}_j)^2
	- \frac{1}{2} \epsilon (\tilde{x}_i \tilde{x}_j)^2  \sqrt{\frac{\lambda}{n}} Z_{i,j}
	\\
	F(\epsilon) &= \E \log \sum_{\bbf{x} \in \{1,2\}^n} P_0(\bbf{x}) \exp(H(\bbf{x},\bbf{X},\bbf{Z},\epsilon)) 
\end{align*}
Notice that $F(0)= I(\bbf{X},\bbf{Y})$ and $F(\epsilon)=J(\bbf{X},\bbf{Z})$. We are going to control the derivative of $F$. We note $\langle \cdot \rangle$  the expectation with respect to the Gibbs measure:$\langle g(\tilde{\bbf{x}}) \rangle := \frac{\sum_{\tilde{\bbf{x}}} P_0(\tilde{\bbf{x}}) g(\tilde{\bbf{x}}) \exp(H(\bbf{x},\bbf{X},\bbf{Z},\epsilon))}{\sum_{\tilde{\bbf{x}}} P_0(\tilde{\bbf{x}}) \exp(H(\bbf{x},\bbf{X},\bbf{Z},\epsilon))}$.
$F$ is derivable and
$$
F'(\epsilon) = - \frac{1}{2} \sqrt{\frac{\lambda}{n}} \sum_{i<j} \E \big[ Z_{i,j} \langle (\tilde{x}_i \tilde{x}_j)^2 \rangle \big]
$$
Here $\langle (\tilde{x}_i \tilde{x}_j)^2 \rangle$ is a continuously differentiable function of $Z_{i,j}$ and 
\begin{align*}
	\partial_{Z_{i,j}} \langle (\tilde{x}_i \tilde{x}_j)^2 \rangle 
	&=
	\langle (\tilde{x}_i \tilde{x}_j)^2 \sqrt{\frac{\lambda}{n}} ( \tilde{x}_i \tilde{x}_j - \frac{1}{2} \epsilon (\tilde{x}_i \tilde{x}_j)^2 ) \rangle - \langle (\tilde{x}_i \tilde{x}_j)^2 \rangle \langle \sqrt{\frac{\lambda}{n}} ( \tilde{x}_i \tilde{x}_j - \frac{1}{2} \epsilon (\tilde{x}_i \tilde{x}_j)^2 ) \rangle
	\\
	&=
	\sqrt{\frac{\lambda}{n}} \text{Cov}_{\langle \cdot \rangle}(\tilde{x}_i \tilde{x}_j, ( \tilde{x}_i \tilde{x}_j - \frac{1}{2} \epsilon (\tilde{x}_i \tilde{x}_j)^2 ))
\end{align*}
Using Gaussian integration by parts: $F'(\epsilon) = - \frac{\lambda}{2n} \sum_{i<j} \E \Big[  \text{Cov}_{\langle \cdot \rangle}(\tilde{x}_i \tilde{x}_j, ( \tilde{x}_i \tilde{x}_j - \frac{1}{2} \epsilon (\tilde{x}_i \tilde{x}_j)^2 )) \Big]$.
The $\tilde{x}_i$ are bounded, so we have
$$
\vert F'(\epsilon) \vert \leq \frac{\lambda}{2n} \sum_{i<j} \E \vert \text{Cov}_{\langle \cdot \rangle}(\tilde{x}_i \tilde{x}_j, ( \tilde{x}_i \tilde{x}_j - \frac{1}{2} \epsilon (\tilde{x}_i \tilde{x}_j)^2 )) \vert
\leq Cn
$$
We conclude $\vert F(0) - F(\epsilon) \vert \leq C n \epsilon$.
\end{proof}

\subsection{From mutual information to solvability: proof of Theorem~\ref{th:solvability_sbm}}

We are first going to introduce an estimation metric that will allow us to make the link between the minimum mean square error for matrix factorization, and the overlap for community detection. Define
\begin{equation} \label{eq:def_mmse_sbm}
	\MMSE^{G}_n(\lambda) = \min_{\hat{\theta}} \frac{2}{n(n-1)} \sum_{i<j} \E \big( \tilde{X}_i \tilde{X}_j - \hat{\theta}_{i,j}(\bbf{G}) \big)^2
	= \frac{2}{n(n-1)} \sum_{i<j} \E \big( \tilde{X}_i \tilde{X}_j -  \E(\tilde{X}_i \tilde{X}_j | \bbf{G}) \big)^2
\end{equation}
where the minimum is taken over all function $\hat{\theta}$ of $\bbf{G}$ (and possibly of some auxiliary randomization). By considering the trivial estimator $\hat{\theta}=0$, we see that $\MMSE^{G}_n(\lambda) \in [0,1]$. This estimation metric correspond (up to a vanishing error term) to the derivative of the mutual information between the graph $\bbf{G}$ and the labels $\bbf{X}$.

\begin{proposition}
Let $\lambda_0 >0$. There exists a constant $C>0$ such that, for all $\lambda \in (0,\lambda_0]$, $d\geq 1$ and $n \geq 1$
\begin{equation} \label{eq:control_mmse}
	\Big| \frac{1}{n} \frac{\partial I(\bbf{X},\bbf{G})}{\partial \lambda} - \frac{1}{4} \MMSE^G_n(\lambda) \Big| \leq C \Big( d^{-1/2} + \frac{d}{n} + \frac{d^{1/2} \lambda^{-1/2}}{n} + d^{3/2} n^{-2} \lambda^{-1/2} \Big)
\end{equation}
\end{proposition}

\begin{proof}
	We are going to differentiate $H(\bbf{X} | \bbf{G})$ with respect to $\lambda$. To do so we will use a differentiation formula from~\cite{deshpande2016asymptotic} (Lemma 7.1), which was first proved in~\cite{measson2009generalized}. Let us recall the setting (taken from~\cite{deshpande2016asymptotic}) of this Lemma.

	For $n$ an integer, denote by $\Pair$ the set of unordered pairs in $[n]$ (in particular $\#\Pair = \binom{n}{2})$). We will use $e,e_1,e_2,\dots$ to denote elements of $\Pair$. For for each $e = (i,j)$ we are given a one-parameter family of discrete noisy channels indexed by $\theta\in J$ (with $J = (a_1,a_2)$ a non-empty interval), with finite input alphabet $\mathcal{X}_0$ and finite output alphabet $\mathcal{Y}$. Concretely, for any $e$, we have a transition probability 
	\begin{align} \label{eq:transition_probability}
		\{p_{e,\theta}(y|x)\}_{x\in \mathcal{X}_0,y\in \mathcal{Y}}\, ,
	\end{align}
	which is differentiable in $\theta$. We shall omit the subscript
	$\theta$ since it will be clear from the context.

	We then consider  $\bbf{X}= (X_1,X_2,\dots,X_n)$ a random vector in $\mathcal{X}^n$, and $\bbf{Y} = (Y_{ij})_{(i,j)\in\Pair}$ a set of observations in $\mathcal{Y}^{\Pair}$ that are conditionally independent given $\bbf{X}$. Further $Y_{ij}$ is the noisy observation of $X_i X_j \in \mathcal{X}_0$ through the channel $p_{ij}(\,\cdot\,|\,\cdot\,)$. In formulae, the joint probability density function of $\bbf{X}$ and $\bbf{Y}$ is
	\begin{align}
		p_{\bbf{X},\bbf{Y}} (\bbf{x},\bbf{y}) = p_{\bbf{X}}(\bbf{x}) \prod_{(i,j)\in \Pair}p_{ij}(y_{ij}|x_i x_j)\,.
	\end{align}
	This obviously include the two-groups stochastic block model as a special case. In that case $\bbf{Y} = \bbf{G}$ is just the adjacency matrix of the graph. In the following we write $\bbf{Y}_{-e}= (Y_{e'})_{e'\in\Pair\setminus e}$ for the set of observations excluded $e$, and $X_e = X_{i}X_j$ for $e=(i,j)$. 

	\begin{lemma} \label{lem:diff_entropy}
		With the above notation, we have:
		\begin{align}
			\frac{\partial H(\bbf{X}|\bbf{Y})}{\partial \theta} = \sum_{e\in\Pair}\sum_{x_e,y_e}\frac{\partial p_e(y_e|x_e)}{\partial \theta}
			\E\Big\{p_{X_e|\bbf{Y}_{-e}}(x_e|\bbf{Y}_{-e})\log
			\Big[\sum_{x'_e}\frac{p_e(y_e|x'_e)}{p_e(y_e|x_e)}p_{X_e|\bbf{Y}_{-e}}(x'_e|\bbf{Y}_{-e})\Big]\Big\}
			\label{eq:BigDerivative}
		\end{align}
	\end{lemma}

  We apply Lemma~\ref{lem:diff_entropy} to the stochastic block model. Let $\lambda_0 >0$ and $\lambda \in (0,\lambda_0]$.
  Instead of having $G_{i,j} \in \{0,1\}$, it will be more convenient to consider $G_{i,j} \in \{-1,1\}$: $G_{i,j} = 1$ if $i \sim j$, $G_{i,j} = -1$ else. Notice that neither the mutual information nor $\MMSE^G_n$ are affected by this change. $G_{i,j}$ is, conditionally to $\tilde{X}_i \tilde{X}_j$, independent of any other random variable, and distributed as follows
  $$
  \PP(G_{i,j} = 1 | \tilde{X}_i \tilde{X}_j)=\frac{d}{n} (1 + \sqrt{\frac{\lambda}{d}} \tilde{X}_i \tilde{X}_j)
  $$
  The transition probability from equation~\eqref{eq:transition_probability} is then $p_{\lambda} (g_{i,j} | \tilde{x}_i \tilde{x}_j) = \frac{1-g_{i,j}}{2} + g_{i,j}\frac{d}{n}(1 + \tilde{x}_i \tilde{x}_j \sqrt{\frac{\lambda}{d}})$. Thus
  $$
  \frac{\partial}{\partial \lambda} p_{\lambda} (g_{i,j} | \tilde{x}_i \tilde{x}_j) 
  = \frac{1}{2n} g_{i,j} \tilde{x}_i \tilde{x}_j \sqrt{\frac{d}{\lambda}}
  $$
  Lemma~\ref{lem:diff_entropy} gives
  \begin{align}
	  &\frac{\partial H(\bbf{X}|\bbf{G})}{\partial \lambda} \nonumber
	  \\
	  &=\frac{1}{2n}  \sqrt{\frac{d}{\lambda}} 
	  \sum_{i<j} 
	  \underbrace{%
	  \sum_{\overset{\tilde{x}_i,\tilde{x}_j}{g_{i,j}}} g_{i,j} \tilde{x}_i \tilde{x}_j
	  \E \Big[
		  p_{(\tilde{X}_i,\tilde{X}_j)|\bbf{G}_{-(i,j)}}(\tilde{x}_i,\tilde{x}_j | \bbf{G}_{-(i,j)})
		  \log \big(
			  \sum_{\tilde{x}_i',\tilde{x}_j'} p_{\lambda}(g_{i,j}| \tilde{x}_i' \tilde{x}_j') p_{(\tilde{X}_i,\tilde{X}_j)|\bbf{G}_{-(i,j)}}(\tilde{x}_i', \tilde{x}_j'|\bbf{G}_{-(i,j)})
		  \big)
	  \Big]
	  }_{A} \nonumber
	  \\
	  &-\frac{1}{2n}  \sqrt{\frac{d}{\lambda}} 
	  \sum_{i<j} 
	  \underbrace{%
	  \sum_{\overset{\tilde{x}_i,\tilde{x}_j}{g_{i,j}}} g_{i,j} \tilde{x}_i \tilde{x}_j
	  \E \Big[
		  p_{(\tilde{X}_i,\tilde{X}_j)|\bbf{G}_{-(i,j)}}(\tilde{x}_i,\tilde{x}_j | \bbf{G}_{-(i,j)})
		  \log \big(
			  p_{\lambda}(g_{i,j}| \tilde{x}_i \tilde{x}_j)
		  \big)
	  \Big]
	  }_{B} \label{eq:compute_dH}
  \end{align}
  Compute 
  \begin{align*}
	  B &= 
	  \sum_{\overset{\tilde{x}_i,\tilde{x}_j}{g_{i,j}}} g_{i,j} \tilde{x}_i \tilde{x}_j
	  p(\tilde{x}_i,\tilde{x}_j)
	  \log \big(
		  p_{\lambda}(g_{i,j}| \tilde{x}_i \tilde{x}_j)
	  \big)
	  \\
	  &=
	  p^2 \frac{1-p}{p} \big( \log (\frac{ad}{n}) - \log(1-\frac{ad}{n})\big)
	  -2 p(1-p) \big( \log (\frac{bd}{n}) - \log(1-\frac{bd}{n})\big)
	  +(1-p)^2 \frac{p}{1-p} \big( \log (\frac{cd}{n}) - \log(1-\frac{cd}{n})\big)
	  \\
	  &=
	  p(1-p) \log \big(\frac{ac}{b^2}\big) + p(1-p) d (2 \frac{b}{n}- \frac{a}{n} - \frac{c}{n}) +O(\frac{d^2}{n^2})
	  \\
	  &=
	  p(1-p) \epsilon \big( \frac{1-p}{p} + \frac{p}{1-p} + 2 \big) + O(\epsilon^2) + O(\frac{d^2}{n^2})
	  \\
	  &= \epsilon + O(\epsilon^2) + O(\frac{d^2}{n^2})
  \end{align*}
  and
  \begin{align*}
	  A &=
	  \sum_{g_{i,j}} g_{i,j} \E \Big[
		  \E(\tilde{X}_i \tilde{X}_j | \bbf{G}_{-(i,j)})
		  \log \big(
			  \sum_{\tilde{x}_i',\tilde{x}_j'} p_{\lambda}(g_{i,j}| \tilde{x}_i' \tilde{x}_j') p_{(\tilde{X}_i,\tilde{X}_j)|\bbf{G}_{-(i,j)}}(\tilde{x}_i', \tilde{x}_j'|\bbf{G}_{-(i,j)})
		  \big)
	  \Big]
	  \\
	  &=
	  \E \Big[
		  \E(\tilde{X}_i \tilde{X}_j | \bbf{G}_{-(i,j)})
		  \log \big(
			  \frac{%
			  \sum_{\tilde{x}_i',\tilde{x}_j'} p_{\lambda}(1 | \tilde{x}_i' \tilde{x}_j') p_{(\tilde{X}_i,\tilde{X}_j)|\bbf{G}_{-(i,j)}}(\tilde{x}_i', \tilde{x}_j'|\bbf{G}_{-(i,j)})
			  }{%
			  \sum_{\tilde{x}_i',\tilde{x}_j'} p_{\lambda}(-1| \tilde{x}_i' \tilde{x}_j') p_{(\tilde{X}_i,\tilde{X}_j)|\bbf{G}_{-(i,j)}}(\tilde{x}_i', \tilde{x}_j'|\bbf{G}_{-(i,j)})
			  }
		  \big)
	  \Big]
  \end{align*}
  Define $\hat{a}_{i,j} = \E(\tilde{X}_i \tilde{X}_j | \bbf{G}_{-(i,j)})$.
  \begin{align*}
	  A &=
	  \E \Big[
		  \hat{a}_{i,j}
		  \log \big(
			  \frac{%
			  \sum_{\tilde{x}_i',\tilde{x}_j'} \frac{d}{n}(1 + \tilde{x}_i' \tilde{x}_j' \sqrt{\frac{\lambda}{d}}) p_{(\tilde{X}_i,\tilde{X}_j)|\bbf{G}_{-(i,j)}}(\tilde{x}_i', \tilde{x}_j'|\bbf{G}_{-(i,j)})
			  }{%
			  \sum_{\tilde{x}_i',\tilde{x}_j'} (1- \frac{d}{n}(1 + \tilde{x}_i' \tilde{x}_j' \sqrt{\frac{\lambda}{d}})) p_{(\tilde{X}_i,\tilde{X}_j)|\bbf{G}_{-(i,j)}}(\tilde{x}_i', \tilde{x}_j'|\bbf{G}_{-(i,j)})
			  }
		  \big)
	  \Big]
	  \\
	  &=
	  \E \Big[
		  \hat{a}_{i,j}
		  \log \big(
			  \frac{%
			  \frac{d}{n}(1 + \hat{a}_{i,j} \sqrt{\frac{\lambda}{d}}) 
			  }{%
			  1- \frac{d}{n}(1 + \hat{a}_{i,j} \sqrt{\frac{\lambda}{d}}) 
			  }
		  \big)
	  \Big]
	  =
	  \E \Big[
		  \hat{a}_{i,j}
		  \big(\log \frac{d}{n} + \epsilon \hat{a}_{i,j} + O(\epsilon^2) + O(\frac{d}{n})
		  \big)
	  \Big]
  \end{align*}
  $\E \hat{a}_{i,j} = \E \tilde{X}_i \tilde{X}_j = (\E \tilde{X}_1)^2 = 0$ therefore $A= \epsilon \E \hat{a}_{i,j}^2 + O(\epsilon^2 + \frac{d}{n})$.
  By replacing $A$ and $B$ in~\eqref{eq:compute_dH}, we have
  \begin{align}
	  \frac{\partial H(\bbf{X}|\bbf{G})}{\partial \lambda}
	  &=\frac{1}{2n}  \epsilon^{-1} 
	  \sum_{i<j}  \Big(
		  \epsilon \E \hat{a}_{i,j}^2 + O(\epsilon^2) + O(\frac{d}{n}) 
	  \Big)
	  -\frac{1}{2n}  \epsilon^{-1} 
	  \sum_{i<j} \Big(
		  \epsilon + O(\epsilon^2) + O(\frac{d^2}{n^2})
	  \Big) \nonumber
	  \\
	  &= \frac{1}{2n} \sum_{i<j} \Big( \E \hat{a}_{i,j}^2 -1 \Big) + O(n \epsilon) 
	  + O(d^{1/2} \lambda^{-1/2} + d^{3/2} n^{-1} \lambda^{-1/2})\label{eq:dH_hat}
  \end{align}
  Define $a_{i,j} = \E (\tilde{X}_i \tilde{X}_j | \bbf{G})$. Using Bayes rule, we have
  $$
  p(\tilde{x}_i, \tilde{x}_j | \bbf{G}) = \frac{p(G_{i,j}|\tilde{x}_i \tilde{x}_j) p(\tilde{x}_i, \tilde{x}_j | \bbf{G}_{-(i,j)})}{\sum_{\tilde{x}_i',\tilde{x}_j'}p(G_{i,j}|\tilde{x}_i' \tilde{x}_j') p(\tilde{x}_i', \tilde{x}_j' | \bbf{G}_{-(i,j)})}
  $$
  If $G_{i,j}=1$, then
  $$
  p(\tilde{x}_i, \tilde{x}_j | \bbf{G}) = \frac{\frac{d}{n}(1 + \tilde{x}_i \tilde{x}_j \epsilon) p(\tilde{x}_i, \tilde{x}_j | \bbf{G}_{-(i,j)})}{\sum_{\tilde{x}_i',\tilde{x}_j'} \frac{d}{n}(1 + \tilde{x}_i' \tilde{x}_j' \epsilon) p(\tilde{x}_i', \tilde{x}_j' | \bbf{G}_{-(i,j)})}
  =
  \frac{(1 + \epsilon \tilde{x}_i \tilde{x}_j) p(\tilde{x}_i,\tilde{x}_j|\bbf{G}_{-(i,j)})}{1 + \epsilon \hat{a}_{i,j}}
  $$
  Thus
  $$
  a_{i,j} = \frac{\hat{a}_{i,j} + \epsilon \E((\tilde{X}_i \tilde{X}_j)^2 | \bbf{G}_{-(i,j)})}{1 + \epsilon \hat{a}_{i,j}} = \hat{a}_{i,j} + O(\epsilon)
  $$
  If $G_{i,j} = -1$,
  $$
  p(\tilde{x}_i, \tilde{x}_j | \bbf{G}) = \frac{(1-\frac{d}{n}(1 + \tilde{x}_i \tilde{x}_j \epsilon)) p(\tilde{x}_i, \tilde{x}_j | \bbf{G}_{-(i,j)})}{1 - \sum_{\tilde{x}_i',\tilde{x}_j'} \frac{d}{n}(1 + \tilde{x}_i' \tilde{x}_j' \epsilon) p(\tilde{x}_i', \tilde{x}_j' | \bbf{G}_{-(i,j)})}
  =
  (1-\frac{d}{n}(1 + \tilde{x}_i \tilde{x}_j \epsilon)) p(\tilde{x}_i, \tilde{x}_j | \bbf{G}_{-(i,j)}) + O(\frac{d}{n})
  $$
  Therefore $a_{i,j} = \hat{a}_{i,j} + O(\frac{d}{n})$. Equation~\eqref{eq:dH_hat} becomes then
  \begin{align}
	  \frac{\partial H(\bbf{X}|\bbf{G})}{\partial \lambda}
	  &= \frac{1}{2n} \sum_{i<j} \Big( \E a_{i,j}^2 -1 \Big) + O(n \epsilon + d) 
	  +O(d^{1/2} \lambda^{-1/2} + d^{3/2} n^{-1} \lambda^{-1/2}) \nonumber
	  \\
	  &= - \frac{1}{2n} \sum_{i<j} \E \Big(
	  (\tilde{X}_i \tilde{X}_j - \E(\tilde{X}_i \tilde{X}_j |\bbf{G}))^2 
  \Big)+O(n\epsilon +d + d^{1/2} \lambda^{-1/2} + d^{3/2} n^{-1} \lambda^{-1/2})
		  \end{align}
		  Decomposing $I(\bbf{X};\bbf{G}) = H(\bbf{X}) - H(\bbf{X}|\bbf{G})$ we obtain the desired result.
	  \end{proof}
	  \\

	  Consequently, if one consider a sufficiently large $d$ (in order to apply Corollary~\ref{cor:information_sbm}) and one integrate equation~\eqref{eq:control_mmse} from $0$ to $\lambda > 0$, and let $n$ tend to infinity,

	  \begin{equation} \label{eq:int_mmse}
		  \limsup_{n \to \infty} \Big| \int_0^{\lambda} \MMSE^G_n(\lambda') d\lambda' - (\frac{\lambda}{4} - \Phi_g(\lambda)) \Big| \leq C \lambda d^{-1/2} + C' \epsilon \leq K \epsilon
	  \end{equation}
	  for some constant (depending on $\lambda$ but not on $d$) $K>0$.

	  \begin{proposition} \label{prop:mmse}
		  For $\lambda < \lambda_c(p)$
		  \begin{equation} \label{eq:lim_mmse_1}
			  \lim_{d \to \infty} \liminf_{n \to \infty} \MMSE^G_n (\lambda) = 1
		  \end{equation}
		  For $\lambda > \lambda_c(p)$
		  \begin{equation}\label{eq:lim_mmse_2}
			  \limsup_{d \to \infty} \limsup_{n \to \infty} \MMSE^G_n (\lambda) <1
		  \end{equation}
	  \end{proposition}

	  \begin{proof}
		  This is a consequence of equation~\eqref{eq:int_mmse} and the definition of $\lambda_c(p)$.
		  We will show~\eqref{eq:lim_mmse_1} first. $\Phi_g$ is continuous, $\Phi_g(0) = 0$ and for almost every $\lambda \in ] 0, \lambda_c[$, $\Phi_g'(\lambda) =0$. Therefore $\Phi_g(\lambda_c)=0$. Equation~\eqref{eq:int_mmse} gives then $\limsup_{n \to \infty} | \int_0^{\lambda_c} \frac{1}{4}\MMSE^G_n(\lambda) d\lambda - \frac{\lambda_c}{4} | \leq K \epsilon$,
		  which gives
		  $$
		  \limsup_{n \to \infty} \int_0^{\lambda_c} |1 - \MMSE^G_n(\lambda) | d\lambda  \leq 4K \sqrt{\frac{\lambda_c}{d}}
		  $$
		  because $\MMSE^G_n(\lambda) \leq 1$. Equation~\eqref{eq:lim_mmse_1} follows.
		  \\

		  Equation~\eqref{eq:lim_mmse_2} is proved analogously. If $\limsup_{d \to \infty} \limsup_{n \to \infty} \MMSE^G_n(\lambda_0) =1$ for some $\lambda_0 > \lambda_c$, then $\limsup_{d \to \infty} \limsup_{n \to \infty} | \int_0^{\lambda_0} \frac{1}{4}\MMSE^G_n(\lambda) d\lambda - \frac{\lambda_0}{4} | =0$ which implies (by equation~\eqref{eq:int_mmse}) that $\Phi_g(\lambda_0)=0$. This is absurd because $\Phi_g(\lambda_c)=0$ and $\Phi_g$ is strictly increasing on $]\lambda_c, \lambda_0[$.
	  \end{proof}
	  \\

	  The following two lemmas will be useful to make the link between the MMSE and the overlap.
	  \begin{lemma} \label{lem:partition}
		  Let $n \in \N^*$. Let $(A_1,A_2)$ and $(B_1,B_2)$ be two partitions of $\{1, \dots,n\}$. Then
		  \begin{align}
			  \# A_1 \cap B_1 - \frac{1}{n} \# A_1 \# B_1 &= \# A_2 \cap B_2 - \frac{1}{n} \# A_2 \# B_2 \label{eq:partition1} \\
			  \# A_1 \cap B_1 - \frac{1}{n} \# A_1 \# B_1 &= - \Big( \# A_1 \cap B_2 - \frac{1}{n} \# A_1 \# B_2 \Big)\label{eq:partition2} 
		  \end{align}
	  \end{lemma}
	  \begin{proof}
		  We prove~\eqref{eq:partition1} first.
		  Remark that $\# A_2 \cap B_2 = \# B_2 - (\# A_1 - \# B_1 \cap A_1)$ and $\# A_2 \# B_2 = n^2 - n(\# A_1 + \# B_1)+ \# A_1 \# B_1$. So that
		  \begin{align*}
			  \# A_1 \cap B_1 - \frac{1}{n} \# A_1 \# B_1 &= \# A_2 \cap B_2 - \frac{1}{n} \# A_2 \# B_2  
			  + \# B_2 - \# A_1 - n + \# A_1 + \# B_1 \\
			  &= \# A_2 \cap B_2 - \frac{1}{n} \# A_2 \# B_2  
		  \end{align*}

		  To prove~\eqref{eq:partition2}, write $\# A_1 \cap B_1 = \# A_1  - \# A_1 \cap B_2$. Thus
		  \begin{align*}
			  \# A_1 \cap B_1 - \frac{1}{n} \# A_1 \# B_1 &= \# A_1  - \# A_1 \cap B_2 - \frac{1}{n} \# A_1 (n - \# B_2) 
													   =- \Big( \# A_1 \cap B_2 - \frac{1}{n} \# A_1 \# B_2 \Big)
		  \end{align*}
	  \end{proof}

	  \begin{lemma} \label{lem:overlap_equivalence}
		  Let $\bbf{x}=(x_1, \dots, x_n) \in \{1,2\}^n$. Recall that $\tilde{\bbf{x}}= (\phi_p(x_1), \dots, \phi_p(x_n))$ where $\phi_p(1) = \sqrt{\frac{1-p}{p}}$ and $\phi_p(2)=-\sqrt{\frac{p}{1-p}}$. Then
		  $$
		  \frac{1}{n} \Big| \sum_{i=1}^n \tilde{X}_i \tilde{x}_i  \Big| 
		  = \frac{1}{2p(1-p)} \overlap(\bbf{x},\bbf{X}) + O \Big(|\frac{S_1(\bbf{X})}{n} - p | + |\frac{S_2(\bbf{X})}{n} - (1-p) | \Big)
		  $$
	  \end{lemma}

	  \begin{proof}
		  \begin{align*}
			  \frac{1}{n} \Big| \sum_{i=1}^n \tilde{X}_i \tilde{x}_i  \Big| 
			  &= \frac{1}{n} \Big| \frac{1-p}{p} \# S_1(\bbf{x}) \cap S_1(\bbf{X}) + \frac{p}{1-p} \# S_2(\bbf{x}) \cap S_2(\bbf{X}) - (n - \# S_1(\bbf{x}) \cap S_1(\bbf{X}) - \# S_2(\bbf{x}) \cap S_2(\bbf{X})) \Big|
			  \\
			  &= \frac{1}{n} \Big| \frac{1}{p} \# S_1(\bbf{x}) \cap S_1(\bbf{X}) + \frac{1}{1-p} \# S_2(\bbf{x}) \cap S_2(\bbf{X}) - n  \Big|
			  \\
			  &= \frac{1}{n} \Big| \frac{1}{p} (\# S_1(\bbf{x}) \cap S_1(\bbf{X}) - p \# S_1(\bbf{x})) + \frac{1}{1-p} (\# S_2(\bbf{x}) \cap S_2(\bbf{X}) - (1-p) \# S_2(\bbf{x})) \Big|
			  \\
			  &= \frac{1}{n} \Big| \frac{1}{p} (\# S_1(\bbf{x}) \cap S_1(\bbf{X}) - \frac{1}{n} \# S_1(\bbf{X}) \# S_1(\bbf{x})) + \frac{1}{1-p} (\# S_2(\bbf{x}) \cap S_2(\bbf{X}) - \frac{1}{n} \# S_2(\bbf{X}) \# S_2(\bbf{x})) \Big| \\
			  &+ O\Big( \frac{S_1(\bbf{x})}{pn} |p - \frac{S_1(\bbf{X})}{n}| + \frac{S_2(\bbf{x})}{(1-p)n} |(1-p) - \frac{S_2(\bbf{X})}{n}| \Big)
		  \end{align*}

		  Using Lemma~\ref{lem:partition}, one obtain then
		  $$
		  \frac{1}{n} \Big| \sum_{i=1}^n \tilde{X}_i \tilde{x}_i  \Big| 
		  = \frac{1}{2p(1-p)} \overlap(\bbf{x},\bbf{X}) + O(|\frac{S_1(\bbf{X})}{n} - p | + |\frac{S_2(\bbf{X})}{n} - (1-p) |)
		  $$

	  \end{proof}
	  \\

	  We are now going to prove Theorem~\ref{th:solvability_sbm}. Remark that
	  \begin{align*}
		  \MMSE^G_n(\lambda) &= \frac{2}{n(n-1)} \sum_{i<j} \E \big( \tilde{X}_i \tilde{X}_j -  \E(\tilde{X}_i \tilde{X}_j | \bbf{G}) \big)^2 \\
									   &= \frac{2}{n(n-1)} \sum_{i<j} \E \big( 1 -  (\E(\tilde{X}_i \tilde{X}_j | \bbf{G}))^2 \big)
	  \end{align*}

	  We will denote $\langle \cdot \rangle_G$ the expectation with respect to the posterior distribution $\PP(.|\bbf{G})$. We have then $\E (\E(\tilde{X}_i \tilde{X}_j | \bbf{G})^2) = \E \langle \tilde{X}_i \tilde{X}_j \tilde{x}_i \tilde{x}_j \rangle_G$, where $\tilde{x}$ is sampled, conditionally to $\bbf{G}$, from $\PP(.|\bbf{G})$. Thus
	  \begin{equation} \label{eq:mmse_overlap}
		  \MMSE^G_n(\lambda) = 1 - \frac{1}{n^2} \sum_{i,j}\E \langle \tilde{X}_i \tilde{X}_j \tilde{x}_i \tilde{x}_j \rangle_G + o(1) = 1 - \E \Big\langle (\frac{1}{n} \sum_{i=1}^n \tilde{X}_i \tilde{x}_i)^2 \Big\rangle_G + o(1)
	  \end{equation}

	  Suppose $\lambda > \lambda_c$. Then~\eqref{eq:mmse_overlap} and Proposition~\ref{prop:mmse} imply
	  $$
	  \liminf_{d \to \infty} \liminf_{n \to \infty} \E \langle |\frac{1}{n} \sum_{i=1}^n \tilde{x}_i \tilde{X}_i | \rangle_G >0
	  $$
	  Using Lemma~\ref{lem:overlap_equivalence}, this gives 
	  $$
	  \liminf_{d \to \infty} \liminf_{n \to \infty} \E \langle \overlap(\bbf{x},\bbf{X}) \rangle_G >0
	  $$
	  where $\bbf{x}$ is sampled according to the posterior distribution of $\bbf{X}$. Sampling from the posterior distribution $\PP(\bbf{X}=.|\bbf{G})$ provides thus an estimator that achieves a non zero overlap: the community detection problem is solvable.
	  \\

	  Suppose now $\lambda<\lambda_c$. Suppose that the community detection problem is solvable. There is therefore an estimator $\bbf{a}$ that achieves a non zero overlap. Lemma~\ref{lem:overlap_equivalence} gives then
	  $$
	  \alpha := \liminf_{d \to \infty} \liminf_{n \to \infty} \frac{1}{n} \E |\sum_{i=1}^n \tilde{X}_i \tilde{a}_i | > 0
	  $$
Compute now for $\delta \in (0,1]$
\begin{align*}
	\frac{2}{n(n-1)} \sum_{i<j} \E (\tilde{X}_i \tilde{X}_j - \delta \tilde{a}_i \tilde{a}_j)^2
	&= \frac{1}{n^2} \sum_{i,j} \E (\tilde{X}_i \tilde{X}_j - \delta \tilde{a}_i \tilde{a}_j)^2 + o(1) \\
	&=\frac{1}{n^2} \sum_{i\neq j} \E (\tilde{X}_i^2 \tilde{X}_j^2) + O(\delta^2) - \frac{2 \delta}{n^2} \E (\sum_{i=1}^n \tilde{X}_i \tilde{a}_i )^2 + o(1) \\
	&\leq 1 + O(\delta^2) - 2 \delta\Big( \frac{1}{n} \E |\sum_{i=1}^n \tilde{X}_i \tilde{a}_i | \Big)^2 + o(1)
\end{align*}
So that 
$$
\liminf_{d \to \infty} \liminf_{n\to \infty} \MMSE^G_n(\lambda) \leq
\liminf_{d \to \infty} \liminf_{n \to \infty} \frac{2}{n(n-1)} \sum_{i<j} \E (\tilde{X}_i \tilde{X}_j - \delta \tilde{a}_i \tilde{a}_j)^2 = 1 - 2 \delta \alpha^2 + O(\delta^2)
$$
The right-hand side will be strictly inferior to $1$ for $\delta$ sufficiently small. This is contradictory with Proposition~\ref{prop:mmse} (recall that $\lambda < \lambda_c$). The community detection problem is not solvable. Theorem~\ref{th:solvability_sbm} is proved.

\bibliographystyle{plain}
\bibliography{./references.bib}

\begin{thebibliography}{10}

\bibitem{abbe2015detection}
Emmanuel Abbe and Colin Sandon.
\newblock Detection in the stochastic block model with multiple clusters: proof
  of the achievability conjectures, acyclic bp, and the information-computation
  gap.
\newblock {\em arXiv preprint arXiv:1512.09080}, 2015.

\bibitem{aizenman2003extended}
Michael Aizenman, Robert Sims, and Shannon~L Starr.
\newblock Extended variational principle for the sherrington-kirkpatrick
  spin-glass model.
\newblock {\em Physical Review B}, 68(21):214403, 2003.

\bibitem{baik2005phase}
Jinho Baik, G{\'e}rard Ben~Arous, and Sandrine P{\'e}ch{\'e}.
\newblock Phase transition of the largest eigenvalue for nonnull complex sample
  covariance matrices.
\newblock {\em Annals of Probability}, pages 1643--1697, 2005.

\bibitem{bandeira2016tightness}
Afonso~S Bandeira, Nicolas Boumal, and Amit Singer.
\newblock Tightness of the maximum likelihood semidefinite relaxation for
  angular synchronization.
\newblock {\em Mathematical Programming}, pages 1--23, 2016.

\bibitem{banks2016information}
Jess Banks, Cristopher Moore, Nicolas Verzelen, Roman Vershynin, and Jiaming
  Xu.
\newblock Information-theoretic bounds and phase transitions in clustering,
  sparse pca, and submatrix localization.
\newblock {\em arXiv preprint arXiv:1607.05222v2}, 2016.

\bibitem{barbier2016mutual}
Jean Barbier, Mohamad Dia, Nicolas Macris, Florent Krzakala, Thibault Lesieur,
  and Lenka Zdeborov{\'a}.
\newblock Mutual information for symmetric rank-one matrix estimation: A proof
  of the replica formula.
\newblock In {\em Advances in Neural Information Processing Systems}, pages
  424--432, 2016.

\bibitem{benaych2011eigenvalues}
Florent Benaych-Georges and Raj~Rao Nadakuditi.
\newblock The eigenvalues and eigenvectors of finite, low rank perturbations of
  large random matrices.
\newblock {\em Advances in Mathematics}, 227(1):494--521, 2011.

\bibitem{boucheron2013concentration}
St{\'e}phane Boucheron, G{\'a}bor Lugosi, and Pascal Massart.
\newblock {\em Concentration inequalities: A nonasymptotic theory of
  independence}.
\newblock Oxford university press, 2013.

\bibitem{caltagirone2016asymmetric}
Francesco Caltagirone, Marc Lelarge, and L{\'e}o Miolane.
\newblock Recovering asymmetric communities in the stochastic block model.
\newblock {\em arXiv preprint arXiv:1610.03680}, 2016.

\bibitem{coja2016information}
Amin Coja-Oghlan, Florent Krzakala, Will Perkins, and Lenka Zdeborova.
\newblock Information-theoretic thresholds from the cavity method.
\newblock {\em arXiv preprint arXiv:1611.00814}, 2016.

\bibitem{decelle2011asymptotic}
Aurelien Decelle, Florent Krzakala, Cristopher Moore, and Lenka Zdeborov{\'a}.
\newblock Asymptotic analysis of the stochastic block model for modular
  networks and its algorithmic applications.
\newblock {\em Physical Review E}, 84(6):066106, 2011.

\bibitem{deshpande2016asymptotic}
Yash Deshpande and Emmanuel Abbe.
\newblock Asymptotic mutual information for the balanced binary stochastic
  block model.
\newblock {\em Information and Inference}, page iaw017, 2016.

\bibitem{deshpande2014information}
Yash Deshpande and Andrea Montanari.
\newblock Information-theoretically optimal sparse pca.
\newblock In {\em 2014 IEEE International Symposium on Information Theory},
  pages 2197--2201. IEEE, 2014.

\bibitem{feral2007largest}
Delphine F{\'e}ral and Sandrine P{\'e}ch{\'e}.
\newblock The largest eigenvalue of rank one deformation of large wigner
  matrices.
\newblock {\em Communications in mathematical physics}, 272(1):185--228, 2007.

\bibitem{guerra2003broken}
Francesco Guerra.
\newblock Broken replica symmetry bounds in the mean field spin glass model.
\newblock {\em Communications in mathematical physics}, 233(1):1--12, 2003.

\bibitem{guo2005mutual}
Dongning Guo, Shlomo Shamai, and Sergio Verd{\'u}.
\newblock Mutual information and minimum mean-square error in gaussian
  channels.
\newblock {\em IEEE Transactions on Information Theory}, 51(4):1261--1282,
  2005.

\bibitem{hajek2016information}
Bruce Hajek, Yihong Wu, and Jiaming Xu.
\newblock Information limits for recovering a hidden community.
\newblock In {\em Information Theory (ISIT), 2016 IEEE International Symposium
  on}, pages 1894--1898. IEEE, 2016.

\bibitem{iba1999nishimori}
Yukito Iba.
\newblock The nishimori line and bayesian statistics.
\newblock {\em Journal of Physics A: Mathematical and General}, 32(21):3875,
  1999.

\bibitem{korada2009exact}
Satish~Babu Korada and Nicolas Macris.
\newblock Exact solution of the gauge symmetric p-spin glass model on a
  complete graph.
\newblock {\em Journal of Statistical Physics}, 136(2):205--230, 2009.

\bibitem{korada2010tight}
Satish~Babu Korada and Nicolas Macris.
\newblock Tight bounds on the capacity of binary input random cdma systems.
\newblock {\em IEEE Transactions on Information Theory}, 56(11):5590--5613,
  2010.

\bibitem{korada2011lindeberg}
Satish~Babu Korada and Andrea Montanari.
\newblock Applications of the lindeberg principle in communications and
  statistical learning.
\newblock {\em IEEE Transactions on Information Theory}, 57(4):2440--2450,
  2011.

\bibitem{krzakala2016mutual}
Florent Krzakala, Jiaming Xu, and Lenka Zdeborov{\'a}.
\newblock Mutual information in rank-one matrix estimation.
\newblock In {\em Information Theory Workshop (ITW), 2016 IEEE}, pages 71--75.
  IEEE, 2016.

\bibitem{DBLP:conf/allerton/LesieurKZ15}
Thibault Lesieur, Florent Krzakala, and Lenka Zdeborov{\'{a}}.
\newblock {MMSE} of probabilistic low-rank matrix estimation: Universality with
  respect to the output channel.
\newblock In {\em 53rd Annual Allerton Conference on Communication, Control,
  and Computing, Allerton 2015, Allerton Park {\&} Retreat Center, Monticello,
  IL, USA, September 29 - October 2, 2015}, pages 680--687. {IEEE}, 2015.

\bibitem{DBLP:conf/isit/LesieurKZ15}
Thibault Lesieur, Florent Krzakala, and Lenka Zdeborov{\'{a}}.
\newblock Phase transitions in sparse {PCA}.
\newblock In {\em {IEEE} International Symposium on Information Theory, {ISIT}
  2015, Hong Kong, China, June 14-19, 2015}, pages 1635--1639. {IEEE}, 2015.

\bibitem{measson2009generalized}
Cyril M{\'e}asson, Andrea Montanari, Thomas~J Richardson, and R{\"u}diger
  Urbanke.
\newblock The generalized area theorem and some of its consequences.
\newblock {\em IEEE Transactions on Information Theory}, 55(11):4793--4821,
  2009.

\bibitem{mezard1987spin}
Marc M{\'e}zard, Giorgio Parisi, and Miguel Virasoro.
\newblock {\em Spin glass theory and beyond: An Introduction to the Replica
  Method and Its Applications}, volume~9.
\newblock World Scientific Publishing Co Inc, 1987.

\bibitem{milgrom2002envelope}
Paul Milgrom and Ilya Segal.
\newblock Envelope theorems for arbitrary choice sets.
\newblock {\em Econometrica}, 70(2):583--601, 2002.

\bibitem{andrea2008estimating}
Andrea Montanari.
\newblock Estimating random variables from random sparse observations.
\newblock {\em European Transactions on Telecommunications}, 19(4):385--403,
  2008.

\bibitem{montanari2015finding}
Andrea Montanari.
\newblock Finding one community in a sparse graph.
\newblock {\em Journal of Statistical Physics}, 161(2):273--299, 2015.

\bibitem{neeman2014asymSBM}
Joe Neeman and Praneeth Netrapalli.
\newblock Non-reconstructability in the stochastic block model.
\newblock {\em arXiv preprint arXiv:1404.6304}, 2014.

\bibitem{nishimori2001statistical}
Hidetoshi Nishimori.
\newblock {\em Statistical physics of spin glasses and information processing:
  an introduction}, volume 111.
\newblock Clarendon Press, 2001.

\bibitem{panchenko2013SK}
Dmitry Panchenko.
\newblock {\em The Sherrington-Kirkpatrick model}.
\newblock Springer Science \& Business Media, 2013.

\bibitem{perry2016optimality}
Amelia Perry, Alexander~S Wein, Afonso~S Bandeira, and Ankur Moitra.
\newblock Optimality and sub-optimality of pca for spiked random matrices and
  synchronization.
\newblock {\em arXiv preprint arXiv:1609.05573}, 2016.

\bibitem{rangan2012iterative}
Sundeep Rangan and Alyson~K Fletcher.
\newblock Iterative estimation of constrained rank-one matrices in noise.
\newblock In {\em Information Theory Proceedings (ISIT), 2012 IEEE
  International Symposium on}, pages 1246--1250. IEEE, 2012.

\bibitem{saade2015spectral}
Alaa Saade, Marc Lelarge, Florent Krzakala, and Lenka Zdeborov{\'a}.
\newblock Spectral detection in the censored block model.
\newblock In {\em 2015 IEEE International Symposium on Information Theory
  (ISIT)}, pages 1184--1188. IEEE, 2015.

\bibitem{saade2016clustering}
Alaa Saade, Marc Lelarge, Florent Krzakala, and Lenka Zdeborov{\'a}.
\newblock Clustering from sparse pairwise measurements.
\newblock In {\em Information Theory (ISIT), 2016 IEEE International Symposium
  on}, pages 780--784. IEEE, 2016.

\bibitem{talagrand2010meanfield1}
Michel Talagrand.
\newblock {\em Mean field models for spin glasses: Volume I: Basic examples},
  volume~54.
\newblock Springer Science \& Business Media, 2010.

\bibitem{talagrand2010meanfield2}
Michel Talagrand.
\newblock {\em Mean field models for spin glasses: Volume II: Advanced
  Replica-Symmetry and Low Temperature}, volume~55.
\newblock Springer Science \& Business Media, 2011.

\bibitem{zdeborova2016statistical}
Lenka Zdeborov{\'a} and Florent Krzakala.
\newblock Statistical physics of inference: Thresholds and algorithms.
\newblock {\em Advances in Physics}, 65(5):453--552, 2016.

\end{thebibliography}
\end{document}